\documentclass[12pt]{amsart}

\usepackage{array}
\usepackage{amssymb}
\usepackage{hyperref} 
\usepackage{enumitem}

\usepackage{amsmath}
\usepackage{latexsym}
\usepackage{extarrows}
\usepackage{enumerate}
\usepackage{txfonts}
\usepackage{mathtools}
\usepackage{bm}
\usepackage{tikz-cd}
\usepackage{xcolor}
\usepackage{quiver}

\makeatletter
\@namedef{subjclassname@2020}{%
  \textup{2020} Mathematics Subject Classification}
\makeatother

\usepackage[T1]{fontenc}

\theoremstyle{plain}

\newtheorem{theorem}{Theorem}[section]
\newtheorem{proposition}[theorem]{Proposition}

\newtheorem{lemma}[theorem]{Lemma}
\newtheorem{corollary}[theorem]{Corollary}
\newtheorem{conjecture}[theorem]{Conjecture}

\theoremstyle{definition}

\newtheorem{definition}[theorem]{Definition}

\newtheorem{remark}[theorem]{Remark}

\newtheorem{example}[theorem]{Example}
\newtheorem{examples}[theorem]{Examples}

\newcommand\E{\mathbb{E}}
\newcommand\Z{\mathbb{Z}}
\newcommand\R{\mathbb{R}}
\newcommand\T{\mathbb{T}}
\newcommand\C{\mathbb{C}}
\newcommand\N{\mathbb{N}}
\newcommand\F{\mathbb{F}}

\newcommand\Hom{\operatorname{Hom}}

\newcommand\eps{\varepsilon}

\newcommand\Poly{\mathrm{Poly}}

\newcommand\TV{\operatorname{TV}}
\newcommand\Sym{\mathrm{Sym}}

\begin{document}


\baselineskip=17pt


\title[Non-Abramov system]{A Host--Kra $\F_2^\omega$-system of order $5$ that is not Abramov of order $5$, and non-measurability of the inverse theorem for the $U^6(\F_2^n)$ norm}

\author[A. Jamneshan]{Asgar Jamneshan}
\address{Mathematical Institute\\ University of Bonn\\ 
Endenicher Allee 60 \\
53115 Bonn, Germany }
\email{ajamnesh@math.uni-bonn.de}

\author[O. Shalom]{Or Shalom}
\address{Department of Mathematics\\ Bar Ilan University \\ 
Ramat Gan \\
5290002, Israel}
\email{Or.Shalom@math.biu.ac.il}

\author[T. Tao]{Terence Tao}
\address{Department of Mathematics\\ University of California \\ 
Los Angeles \\
CA 90095-1555, USA}
\email{tao@math.ucla.edu}

\date{\today}

\begin{abstract}  It was conjectured by Bergelson, Tao, and Ziegler \cite{btz} that every Host--Kra $\F_p^\omega$-system of order $k$ is an Abramov system of order $k$.  This conjecture has been verified for $k \leq p+1$.  In this paper we show that the conjecture fails when $k=5, p=2$.  We in fact establish a stronger (combinatorial) statement, in that we produce a bounded function $f: \F_2^n \to \C$ of large Gowers norm $\|f\|_{U^6(\F_2^n)}$ which (as per the inverse theorem for that norm) correlates with a non-classical quintic phase polynomial $e(P)$, but with the property that all such phase polynomials $e(P)$ are ``non-measurable'' in the sense that they cannot be well approximated by functions of a bounded number of random translates of $f$. 
A simpler version of our construction can also be used to answer a question of Candela, Gonz\'alez-S\'anchez, and Szegedy \cite{CGSS}. 
\end{abstract}

\subjclass[2020]{Primary 37A05, 11B30; Secondary 28D15, 37A35.} 

\keywords{}

\maketitle

\section{Introduction}

Let $p$ be a prime, and let $k \geq 1$.  We consider two statements associated to these parameters: the (now-proven) inverse conjecture \cite{green-tao-inverseu3}, \cite{samorodnitsky} \cite{green-tao-finite} for the Gowers norms in characteristic $p$, and the Bergelson--Tao--Ziegler conjecture \cite{btz}.  We begin with the former.  Given any finite abelian group $G = (G,+)$, we define the Gowers uniformity norm $\|f\|_{U^{k+1}(G)} \geq 0$ of a function $f \colon G \to \C$ by the formula
$$ \|f\|_{U^{k+1}(G)}^{2^{k+1}} \coloneqq \E_{x,h_1,\dots,h_{k+1} \in G} \prod_{\omega \in \{0,1\}^{k+1}} {\mathcal C}^{|\omega|} f\left(x + \omega \cdot \vec h\right)$$
where ${\mathcal C} \colon z \mapsto \bar{z}$ denotes complex conjugation, $\omega = (\omega_1,\dots,\omega_{k+1})$, $|\omega| \coloneqq \omega_1 + \dots + \omega_{k+1}$, $\vec h \coloneqq (h_1,\dots,h_{k+1})$, $\omega \cdot \vec h$ denotes the dot product
$$ \omega \cdot \vec h \coloneqq \omega_1 h_1 + \dots + \omega_{k+1} h_{k+1},$$
$\E_{x \in A} \coloneqq \frac{1}{|A|} \sum_{x \in A}$ denotes the averaging operation, and $|A|$ denotes the cardinality of a finite set $A$.  If $P \colon G \to \T$ is a function taking values in the unit circle $\T \coloneqq \R/\Z$, then we have $\|e(P)\|_{U^{k+1}(G)} \leq 1$, with equality precisely when $P$ is a \emph{(non-classical) polynomial of degree $k$}, as defined in Definition \ref{poly-def}; here $e \colon \T \to \C$ is the fundamental character $e(\theta) \coloneqq e^{2\pi i \theta}$.  The space of polynomials $P \colon G \to \T$ of degree at most $k$ is an abelian group which we denote $\Poly^k(G)$.  By convention, $\Poly^0(G)$ will denote the constant functions $\T$, and $\Poly^k(G) = \{0\}$ for all $k<0$ (thus non-zero constants have degree $0$, and zero has degree $-\infty$).

For each $p,k$, we can then form the following claim:

\begin{conjecture}[Inverse conjecture for the Gowers norm]\label{inv-conj}  For every $\eta>0$ there exists $c=c(k,p,\eta)>0$ such that, whenever $G = \F_p^n$ is an elementary abelian $p$-group and $f \colon G \to {\mathbb D}$ is a function taking values in the unit disk ${\mathbb D} \coloneqq \{ z \in \C: |z| \leq 1\}$ and $\|f\|_{U^{k+1}(G)} \geq \eta$, there exists $P \in \Poly^k(G)$ such that $|\E_{x \in G} f(x) e(-P(x))| \geq c$.
\end{conjecture}

This conjecture has now been established for all values of $k, p$ \cite{taoziegler}.  The case $k=1$ is trivial, the case $k=2$ follows from standard Fourier analytic calculations, and the case $k=3$ was previously established in \cite{green-tao-inverseu3} (for $p>2$) and \cite{samorodnitsky} (for $p=2$).  In \cite{taozieglerhigh}, this conjecture was shown to be a consequence of a conjecture in ergodic theory which we now pause to introduce.  Define an \emph{$\F_p^\omega$-system} to be a (countably generated) probability space $(X,\mu)$ equipped with a measure-preserving action $T^h \colon X \to X$, $h \in \F_p^\omega$ of the group $\F_p^\omega \coloneqq \varinjlim \F_p^n$ (the vector space over $\F_p$ with a countably infinite basis).  One can define analogues of the Gowers uniformity norms $\|f\|_{U^{k+1}(X)}$ (known as \emph{Gowers--Host--Kra seminorms}) for $f \in L^\infty(X)$, and one can similarly define the group $\Poly^k(X)$ of polynomials $P \colon X \to \T$ (defined up to almost everywhere equivalence) as
$$\Poly^k(X) \coloneqq \{ P: \|e(P)\|_{U^{k+1}(X)} = 1 \};$$
see \cite{taozieglerhigh} for details.  An $\F_p^\omega$-system is said to be \emph{of order at most $k$} if $\| f\|_{U^{k+1}(X)} > 0$ for any non-zero element $f$ of $L^\infty(X)$ (where elements of the latter are defined up to almost everywhere equivalence).  We then have

\begin{conjecture}[Bergelson--Tao--Ziegler conjecture]\label{btz-conj}\cite[Remark 1.25]{btz}  Let $X$ be an ergodic $\F_p^\omega$-system of order at most $k$.  Then the $\sigma$-algebra of $X$ is generated (modulo null sets) by the polynomials in $\Poly^k(X)$.
\end{conjecture}

We remark that the ergodicity hypothesis on $X$ can in fact be removed by ergodic decomposition, but we will not need to do so here.

In \cite{taozieglerhigh}, a variant of the Furstenberg correspondence principle was used to show that Conjecture \ref{btz-conj} implied Conjecture \ref{inv-conj} for any given choice of $p,k$.  In \cite{btz}, Conjecture \ref{btz-conj} was established in the high characteristic case $k+1 \leq p$; combining the two results, this also gave Conjecture \ref{inv-conj} in this regime.  The full case of Conjecture \ref{inv-conj} was subsequently established in \cite{taoziegler} by a different method; alternate proofs of some or all of the cases of this conjecture have since been given in \cite{milicevic}, \cite{milicevic-2}, \cite{CGSS}, \cite{tidor}, \cite{milicevic-u56}.  In particular Conjecture \ref{btz-conj} was established in \cite[Theorem 1.12]{CGSS} in the slightly larger range $k \leq p+1$ (and an alternate proof of Conjecture \ref{inv-conj} was given for all $k,p$).  We also remark that in \cite[Theorem 1.20]{btz}, a weaker version of Conjecture \ref{btz-conj} was established in which $\Poly^k(X)$ was replaced by some unspecified subalgebra of $\Poly^{C(p,k)}(X)$ for some constant $C(p,k)$ depending only on $p,k$.  We also note that several other structural results on ergodic $\F_p^\omega$-systems are known; see in particular \cite{CGSS}, \cite{jst-tdsystems}.

Although it was not explicitly noted in \cite{taozieglerhigh}, Conjecture \ref{btz-conj} in fact gives a stronger version of Conjecture \ref{inv-conj} in which the polynomial $P$ produced by the conjecture is (approximately) ``measurable'' with respect to the original function $f$ together with random shifts.  More precisely, consider the following more complicated strengthening of Conjecture \ref{inv-conj}.

\begin{conjecture}[Strong inverse conjecture for the Gowers norm]\label{inv-conj-strong}  Let $\eta>0$, and let $\eps \colon \N \to \R^+$ be a decreasing function.  Then there exists $M = M(k,p,\eta,\eps())$ such that the following holds. If $G = \F_p^n$ is an elementary abelian $p$-group and $f \colon G \to {\mathbb D}$ satisfies $\|f\|_{U^{k+1}(G)} \geq \eta$, then, drawing a random tuple $\vec h = (h_1,\dots,h_M)$ uniformly from $G^M$, we have with probability at least $1/2$ that  there exist
$1 \leq m \leq M$, $P \in \Poly^k(G)$ and a function $F \colon {\mathbb D}^{\F_p^M} \to \C$
of Lipschitz constant at most $M$ (using say the Euclidean metric on ${\mathbb D}^{\F_p^M}$), such that
$$|\E_{x \in G} f(x) e(-P(x))| \geq \frac{1}{m}$$
and
$$ \left|\E_{x \in G} e(P(x)) - F\left( \left(f(x+a \cdot \vec h)\right)_{a \in \F_p^M}\right)\right| \leq \eps(m).$$
\end{conjecture}

The numerical value of the probability $1/2$ here is inessential and could be replaced by any other constant between $0$ and $1$.  Roughly speaking, Conjecture \ref{inv-conj-strong} is a strengthening of Conjecture \ref{inv-conj} in which the polynomial $P$ produced by that conjecture is well approximated by some combination of random shifts of $f$, where the degree $\eps(m)$ of approximation can be guaranteed to be much better than the level $\frac{1}{m}$ of correlation between the polynomial $P$ and the original function $f$.  The Lipschitz property of $F$ is unimportant, since one can easily discretize $f$ to take on a bounded number of values, but we retain it for minor technical reasons\footnote{The Lipschitz property is used in Appendix \ref{app} to prove the correspondence principle (see Theorem \ref{implication} below), where it is used in the proof of Proposition \ref{prop-sampling} (cf. \cite[Proposition 3.13]{taozieglerhigh}) in the form of the Arzel\`{a}--Ascoli theorem.}. 

\begin{example}
When $k=1$, Conjecture \ref{inv-conj-strong} can be established by standard Fourier-analytic arguments, which we now briefly sketch (suppressing precise quantitative bounds in order to simplify the exposition).  If $f \colon G \to {\mathbb D}$ has large $U^2(G)$ norm, then $f$ has a large inner product with the convolution $f * f * \tilde f$, where $\tilde f(x) \coloneqq \overline{f}(-x)$.  Furthermore, this convolution can be uniformly approximated by a bounded linear combination of characters $e(\xi \cdot x)$, which is standard.  To isolate one such frequency $\xi$, one may choose a large number of random shifts $h_1,\dots,h_M$ and then, with high probability, construct a single linear combination $\lambda$ of the delta masses $\delta_{h_1},\dots,\delta_{h_M}$ such that $f * f * \tilde f * \lambda$ is uniformly\footnote{Take $\xi_0$ to be a frequency with $|\widehat f(\xi_0)|$ maximal, and define  
\[
\lambda = \frac{1}{M c_0}\sum_{j=1}^M \overline{e(\xi_0 \cdot h_j)} \delta_{-h_j},
\qquad 
c_0 = \widehat{(f*f*\tilde f)}(\xi_0),
\]
for i.i.d.\ uniform shifts $h_j$. Then $\widehat\lambda(\xi_0)=1/c_0$ exactly, while for every other significant frequency $\xi\neq\xi_0$, the value 
\[
\widehat\lambda(\xi)=\tfrac{1}{M c_0}\sum_j e((\xi-\xi_0)\cdot h_j)
\]
is an average of mean-zero, unit-modulus i.i.d. variables. By Hoeffding’s inequality and a union bound over the finitely many ``large'' Fourier coefficients, these averages are simultaneously small with high probability, so $f*f*\tilde f*\lambda$ is uniformly close to the pure character $x\mapsto e(\xi_0\cdot x)$.
} close to the character $x \mapsto e(\xi \cdot x)$.    Finally, by further random sampling of the $f * \tilde f$ factor (and increasing $M$ if necessary), one can approximate $f * f * \tilde f * \lambda$ in $L^1$ by a bounded linear combination of shifts of $f$ along linear combinations of $h_1,\dots,h_M$.  This yields the case $k=1$ of Conjecture \ref{inv-conj-strong}; we leave the details to the interested reader.

For $k=2,3$ (and $p=2$), the strong inverse conjecture is reminiscent\footnote{We are indebted to James Leng for this observation.} of the quadratic Goldreich--Levin theorem from \cite{tulsiani-wolf} (and the more recent cubic Goldreich--Levin theorem from \cite{kim-li-tidor}), which gives a polynomial (in $n$) time randomized algorithm to reconstruct the polynomial $P$ from the function $f$; however, the strong inverse conjecture is (in principle) stronger than these Goldreich--Levin type results, in that it should (after some additional effort) yield a \emph{bounded-time} (rather than polynomial-time) randomized algorithm to obtain an \emph{approximation} to the polynomial $P$.  Such algorithms are similar in spirit\footnote{We are indebted to Avi Wigderson for this remark.} to implicit (or ``local'') list decoding algorithms for Reed--Muller codes, as discussed for instance in \cite{Sudan2000}, \cite{gkz}.
\end{example}

In Appendix \ref{app} we will modify the arguments in \cite{taozieglerhigh} to show

\begin{theorem}[Application of correspondence principle]\label{implication}  For any given choice of $k$ and $p$, Conjecture \ref{btz-conj} implies Conjecture \ref{inv-conj-strong} (and hence also Conjecture \ref{inv-conj}).
\end{theorem}

In particular, from the previously mentioned results of \cite{CGSS}, Conjecture \ref{inv-conj-strong} holds in the high characteristic case $k \leq p+1$; also, from \cite[Theorem 1.20]{btz} one can establish a weaker version of Conjecture \ref{inv-conj-strong} in which the polynomial $P$ is of degree at most $C(p,k)$ rather than $k$ for some quantity $C(p,k)$ depending only on $p,k$.  

However, the low characteristic case presents additional difficulties; for instance, a key ``exact roots'' property for polynomials in order $k$ $\F_p^\omega$-systems is known to fail in low characteristic \cite[Appendix E]{taoziegler}.  In fact we are able to construct the following counterexample, which is the main result of our paper.

\begin{theorem}[Counterexample to strong inverse conjecture]\label{counter}  
Conjecture \ref{inv-conj-strong} fails when $p=2$ and $k=5$.
\end{theorem}

Combining Theorem \ref{counter} with the contrapositive of Theorem \ref{implication}, we conclude that Conjecture \ref{btz-conj} also fails when $p=2$ and $k=5$; see also Remark \ref{Direct} for how one might give a more direct construction of a counterexample to that conjecture.  Our construction was located numerically, but we give a human-verifiable proof of the theorem here, taking advantage in particular of several technical simplifications available in the $p=2$ case (in particular, we take advantage of the ability to identify the $n$-dimensional cube $\{0,1\}^n$ with the $n$-dimensional vector space $\F_2^n$, for instance in \eqref{nbox}).  It would be interesting to determine the complete range of $p,k$ for which Conjecture \ref{inv-conj-strong} and Conjecture \ref{btz-conj} holds; for instance, the case $p=2,k=4$ remains unresolved for both conjectures, and we have not been able to rigorously establish that these conjectures are monotone\footnote{There is precedent for failure of monotonicity in this subject: in \cite{milicevic-counter}, \cite[Corollary 10]{milicevic-u56} it was shown that a conjecture of Tidor \cite{tidor} on approximately symmetric multilinear forms fails for $4$-linear forms but is true again for $5$-linear forms.  We thank Luka Mili\'cevi\'c for this example.} in $k$.

Informally, Theorem \ref{counter} asserts that in characteristic two, there exist ``pseudo-quintic'' functions $f$ which have large $U^6(\F_2^n)$ norm, and in fact correlate with a genuine quintic $e(P)$, but that the quintics that $f$ correlates with will be ``non-measurable'' in the sense that they cannot be approximated\footnote{To justify this intuition, observe from the Stone--Weierstrass theorem that if ${\mathcal F}$ is the $\sigma$-algebra generated by a function $f \in L^\infty(X)$ and its $G$-translates, then a function $g \in L^\infty(X)$ will be ${\mathcal F}$-measurable iff it can approximated in $L^1$ to arbitrary accuracy by a polynomial combination of a finite number of translates of $f$ and $\overline{f}$.  To ``finitize'' this assertion by replacing the measure-preserving system $X$ with a finite space $\F_2^n$, we informally replace ``finite number of translates'' with ``bounded number of translates'', to prevent the notion of ``measurability'' from becoming trivial in the finitary setting.} in $L^1$ by a polynomial combination of boundedly many translates of $f$ and its complex conjugate.  Instead, one has to use ``non-measurable'' operations, such as taking exact roots of polynomials as in \cite{taoziegler}, in order to locate such quintics $e(P)$.

\begin{remark}  Recently, quantitative versions of Conjecture \ref{inv-conj} for $p=2$ and $k=3,4,5$ have been established in \cite{tidor, milicevic-u56}; in particular the paper \cite{milicevic-u56} covers the case $p=2,k=5$ of Theorem \ref{counter}.  This is however not a contradiction; a crucial step \cite[Proposition 3.5]{tidor} in both those papers (a special case of Theorem \ref{cocyc-trivial} below) is the ability to represent a ``strongly symmetric $k$-linear form'' as the $k$-fold derivative of a degree $k$ polynomial, and this step is ``non-measurable'' as it requires one to expand the form into monomials using a choice of basis for $\F_2^n$.
\end{remark}

\subsection{Overview of proof}

We now give an informal, high-level description of our proof of Theorem \ref{counter}, deferring more precise details to later sections.  Roughly speaking, it would suffice to exhibit, for any sufficiently large $n$, a function $S \colon \F_2^n \to \frac{1}{2}\Z/\Z$ which was ``pseudo-quintic'' in the sense that the Gowers norm $\|e(S)\|_{U^6(\F_2^n)}$ was large, but such that $e(S)$ did not correlate in any significant fashion with $e(P)$ for any genuine quintic polynomial $P \colon \F_2^n \to \T$ which was somehow ``measurable'' with respect to $S$ and related functions.

One way to ensure that the Gowers norm $\|e(S)\|_{U^6(\F_2^n)}$ is large is to enforce some structure on the sixth derivative $d^6 S \colon \F_2^n \times (\F_2^n)^6 \to \frac{1}{2}\Z/\Z$ of $S$, defined by the formula
$$ (d^6 S)_{h_1,\dots,h_6}(x) \coloneqq \partial_{h_1} \dots \partial_{h_6} S(x)$$
where $\partial_h S(x) \coloneqq S(x+h)-S(x)$.  Indeed, a routine application of the Gowers--Cauchy--Schwarz inequality and Fourier decomposition reveals that if $(d^6 S)_{h_1,\dots,h_6}(x)$ can be expressed in terms of a bounded number of quintic or lower degree polynomials applied to the various vertices $x + \omega \cdot \vec h$ of the $6$-dimensional cube $(x + \omega \cdot \vec h)_{\omega \in \{0,1\}^6}$, then $e(S)$ will have large $U^6(\F_2^n)$ norm (see Lemma \ref{gowers} for a rigorous version of this implication).  As it turns out, we will be able to construct a counterexample in which $d^6 S$ is a function of a (randomly chosen) \emph{quadratic} polynomial $Q \colon \F_2^n \to \F_2^2$ taking values in the Klein four-group $X_2 \coloneqq \F_2^2$.  That is to say, $S$ will be chosen to obey the equation
\begin{equation}\label{S-eq}
(d^6 S)_{h_1,\dots,h_6}(x) = \rho\left( \left(Q(x+\omega \cdot \vec h)\right)_{\omega \in \{0,1\}^6} \right)
\end{equation}
for some function $\rho \colon C^6(X_2) \to \frac{1}{2}\Z/\Z$ whose domain $C^6(X_2) \subset X_2^{\{0,1\}^6}$ is a space of ``$6$-cubes'' in $X_2$ that contains all possible values of the tuple
$\left(Q(x+\omega \cdot \vec h)\right)_{\omega \in \{0,1\}^6}$.  In fact, $C^6(X_2)$ can be described explicitly as the set of all tuples of the form
$$ \left( x + \sum_{i=1}^6 h_i \omega_i + \sum_{1 \leq i < j \leq 6} h_{ij} \omega_i \omega_j\right)_{\omega \in \{0,1\}^6}$$
for $x, h_i, h_{ij} \in X_2$.  (In the language of nilspaces that we will use later, we are equipping $X_2$ with the nilspace structure associated to the degree two filtration ${\mathcal D}^2(\F_2^2)$ on the Klein four-group.)

The function $\rho$ has to obey a certain number of properties in order to be able to find a solution $S$ to the equation \eqref{S-eq}.  Firstly, $\rho$ must be symmetric with respect to permutations of $\{1,\dots,6\}$ and must also obey a certain ``cocycle equation'' arising from the identity $\partial_{h+k} S = \partial_h S + T^h \partial_k S$, where $T^h S(x) \coloneqq S(x+h)$ is the shift map.  These properties can be formalized in the language of nilspaces by requiring $\rho$ to be a \emph{$5$-cocycle} on $X_2$ taking values in $\frac{1}{2}\Z/\Z$; see Definition \ref{nil-cocycle-def} for details.  However, the property of being a $5$-cocycle is not yet sufficient to guarantee a solution to \eqref{S-eq}. 
In order to locate a solution, we will require the cocycle $\rho$ to obey an additional property that we call ``strong $2$-homogeneity''.  This property asserts that $\rho$ takes the form
$$ \rho( (x_\omega)_{\omega \in \{0,1\}^6} ) = \sum_{\omega \in \{0,1\}^5} (-1)^{5-|\omega|} \psi( x_{\omega 0}, x_{\omega 1} )$$
for all $(x_\omega)_{\omega \in \{0,1\}^6}$ in $C^6(X_2)$ and some function $\psi \colon C^1(X_2) \to \T$ on the space of pairs $C^1(X_2) = X_2 \times X_2$ on $X_2$, such that $2\psi$ is a ``cubic'' polynomial on $C^1(X_2)$ with respect to a certain natural nilspace structure on $C^1(X_2)$; see Definition \ref{strong-cocycle} for a precise statement.  This turns out to be sufficient to guarantee the existence of the pseudo-quintic function $S \colon \F_2^n \to \frac{1}{2}\Z/\Z$; see Theorem \ref{cocyc-trivial-strong} and Lemma \ref{lift} for precise statements.  

We would still like to ensure that $S$ does not correlate with a quintic phase $e(P)$ where $P$ can be well approximated in terms of $S$ and its translates. An obstruction to this claim would occur if the cocycle $\rho$ was a ``$5$-coboundary'' in the sense that $\rho$ takes the form
$$ \rho( (x_\omega)_{\omega \in \{0,1\}^6} ) = \sum_{\omega \in \{0,1\}^6} (-1)^{6-|\omega|} F(x_\omega)$$
for all $(x_\omega)_{\omega \in \{0,1\}^6}$ in $C^6(X_2)$ and some function $F \colon X_2 \to \T$.  Indeed, if this were the case, then one could rearrange \eqref{S-eq} as 
$$ d^6 (S - F(Q)) = 0$$
and thus we have $e(S) = e(P) e(F(Q))$ for some quintic polynomial $P \in \Poly^5(\F_2^n)$.  Morally speaking, this relation indicates that $e(P)$ correlates with $e(S)$, and that $P$ should be well approximated by $S$ and its translates. Indeed,  note that from the identity $e(S)=e(P)\,e(F(Q))$ it follows that the correlation of $e(S)$ with $e(P)$ is simply the expectation of $e(F(Q(x)))$. Here 
$Q \colon \F_2^n \to X_2$ is a quadratic map with values in the finite group $X_2 = \F_2^2$, and 
$F \colon X_2 \to \T$. Since $Q$ is essentially equidistributed on $X_2$, the quantity 
$\E_x\, e(F(Q(x)))$ is very close to $\E_{y \in X_2} e(F(y))$. For a generic choice of $F$ this latter 
average is non-zero, so one expects this correlation to be nontrivial.\footnote{Note that this does not contradict 
the usual Gauss-sum estimates for scalar-valued quadratic phases $e(R(x))$, since $F \circ Q$ is in 
general not a scalar quadratic polynomial but merely a function factoring through the finite quotient $X_2$.}
Moreover, from the relation $e(P)= e(S)\,\overline{e(F(Q(x)))}$ together with \eqref{S-eq}, one sees that one can approximate $P$ by combining information from $S$ and its shifts.

The key step in our argument is thus to locate a $5$-cocycle $\rho \colon C^6(X_2) \to \frac{1}{2}\Z/\Z$ which is strongly $2$-homogeneous, but not a $5$-coboundary.  This is accomplished in Section \ref{nontriv-sec}.  We remark that this claim involves a finite system of linear equations on a finite-dimensional vector space over $\F_2$, and can be verified numerically by standard linear algebra packages (and in particular through calculations of certain Smith normal forms of matrices); indeed, we used such computer-assisted calculations to lead us to this particular claim.  However, we were subsequently able to describe the cocycle $\rho$ and verify its properties in a completely human-verifiable fashion; see Section \ref{nontriv-sec} for details.

\begin{remark} With our specific choice of $\rho$, we can describe the solutions to \eqref{S-eq} more explicitly as
$$ S = \frac{\binom{R}{2} Q_2}{2} + P$$
where $Q = (Q_1,Q_2)$, $R \colon \F_2^n \to \Z/4\Z$ is a cubic polynomial which is a ``square root'' of $Q_1$ in the sense that $2 \frac{R}{4} = \frac{Q_1}{2} \mod 1$ (or equivalently $R = Q_1 \mod 2$), and $P \colon \F_2^n \to \T$ is an arbitrary quintic polynomial (we can require $P$ to take values in $\frac{1}{2}\Z/\Z$ if we wish $S$ to also take values in this group).  See Lemma \ref{expl}. Heuristically, the presence of the square root in this construction prevents the quintic $P$ (which correlates with $S$) from being ``measurable'' with respect to $S$ and its shifts, although actually demonstrating this rigorously requires a surprisingly large amount of effort.
\end{remark}

In order to convert this explicit cocycle $\rho$ into an actual counterexample to Conjecture \ref{inv-conj-strong} we will rely heavily on the theory of \emph{nilspaces}, as developed for instance in \cite{candela0}, although we will mostly only need to work with \emph{finite} nilspaces, as opposed to compact or measurable nilspaces.  A central role is played in particular by a certain explicit $5$-step finite nilspace  $X_{5,5}$.
As a set, $X_{5,5}$ is given as
$$ X_{5,5} = X_2 \times \frac{1}{2^5}\Z/\Z = \F_2^2 \times \frac{1}{2^5} \Z/\Z$$
but the cube structure on $X_{5,5}$ is somewhat non-trivial, involving the cocycle $\rho \colon C^6(X_2) \to \frac{1}{2}\Z/\Z$ mentioned previously.  Roughly speaking, the nilspace $X_{5,5}$ is the abstraction of a pair $(Q,S)$ of functions, in which $Q$ is itself a pair $Q = (Q_1,Q_2)$ of classical quadratic polynomials (taking values in $\F_2$), and $S$ is a ``pseudo-quintic'' taking values in $\frac{1}{2^5}\Z/\Z$ that obeys the identity \eqref{S-eq}.  It will turn out not to be possible to correlate $S$ with any genuine quintics that only arise from $Q, S$, and a bounded (and randomly selected) number of their translates. The actual verification that these translates do not actually provide any useful information for the purpose of constructing a quintic turns out to be rather tricky, requiring one to show that a certain nilspace extension ``splits'': see Lemma \ref{exist}.  A simpler version $X_{5,1}$ of the nilspace $X_{5,5}$, in which the cyclic group $\frac{1}{2^5}\Z/\Z$ is replaced by $\frac{1}{2}\Z/\Z$, can also be used to quickly answer a question of Candela, Gonz\'alez-S\'anchez, and Szegedy \cite[Question 5.18]{CGSS} in the negative, thus giving a weaker version of Theorem \ref{counter}; see Proposition \ref{injcont}.

\begin{remark}
We identify the core of the proof as solving a cohomological problem: finding finite abelian $2$-groups equipped with a cube structure that support $2$-homogeneous $k$-cocycles which are not $k$-coboundaries. The cubes constructed from nilspaces built from such cocycles encode functions with large Gowers norms that violate the strong inverse Gowers conjecture \ref{inv-conj-strong}.

Our argument proceeds by contradiction, demonstrating that if such functions were to satisfy Conjecture \ref{inv-conj-strong}, then the associated cubes would exhibit an asymptotic equidistribution property. This equidistribution, however, would imply the vanishing of the cohomology associated with the underlying cocycle. Thus, our proof establishes a link between vanishing cohomology and equidistribution.

We propose that a deeper exploration of this connection between cohomology and equidistribution is crucial for a more conceptual understanding of the failure and for identifying the full range of the failure of Conjecture \ref{btz-conj}.

We note that what we refer to as equidistribution is termed the "balanced property" in nilspace literature, where its significance in inverse Gowers theory has been highlighted (see, e.g., \cite{CGSS}).
\end{remark}

\subsection{Acknowledgments}
 A.J. was funded by the Deutsche Forschungsgemeinschaft (DFG, German Research Foundation) Heisenberg Grant -- 547294463.
OS was supported by NSF grant DMS-1926686 and ISF grant 3056/21.  Over the course of this research, TT was supported by a Simons Investigator grant, the James and Carol Collins Chair, the Mathematical Analysis \& Application Research Fund, and by NSF grants DMS-1764034 and DMS-2347850, and is particularly grateful to recent donors to the Research Fund. We thank Tim Austin, Pablo Candela, and  Luka Mili\'cevi\'c  for corrections and comments.  The authors are also particularly grateful to the anonymous referee for a very careful reading of the manuscript and for many helpful suggestions that improved it. 

\subsection{Notation}\label{notation-sec}

We identify the field $\F_2$ with the cyclic group $\Z/2\Z$.  If $a$ is an element of a cyclic group $\Z/q\Z$, we use $\frac{a}{q}$ to denote the corresponding element of the finite subgroup $\frac{1}{q}\Z/\Z$ of the unit circle $\T = \R/\Z$, thus
$$ \frac{a + q\Z}{q} = \frac{a}{q} \mod 1.$$
We observe that the binomial coefficient $n \mapsto \binom{n}{2}$ is well-defined as a map from $\Z/4\Z$ to $\F_2$; indeed, we have $\binom{n}{2} = 0 \mod 2$ when $n=0,1 \mod 4$ and $\binom{n}{2} = 1 \mod 2$ when $n = 2,3 \mod 4$.

\section{A characterization of coboundaries on $\F_2^n$}

Let $G = (G,+)$ be a discrete abelian group.  As discussed in Appendix \ref{nilspace-app}, $G$ can be given the structure ${\mathcal D}^1(G)$ of a degree one filtered abelian group, and hence a nilspace.  Given a function $F \colon G \to \T$ from $G$ to the torus $\T$, this gives a derivative map $d^{k+1} F \colon G^{[k+1]} \to \T$ for every $k \geq 0$.  We can describe this map more explicitly by using the identification $G \times G^{k+1} \equiv G^{[k+1]}$ given by the formula
\begin{equation}\label{hik}
(x, \vec h)  \equiv \left(x + \omega \cdot \vec h\right)_{\omega \in \{0,1\}^{k+1}}
\end{equation}
for $x \in G$ and $\vec h = (h_1,\dots,h_{k+1}) \in G^{k+1}$, and then writing
\begin{align*}
(d^{k+1} F)_{h_1,\dots,h_{k+1}}(x) &\coloneqq d^{k+1} F\left(\left (x + \omega \cdot \vec h\right)_{\omega \in \{0,1\}^{k+1}} \right) \\
&= \sum_{\omega \in \{0,1\}^{k+1}} (-1)^{k+1-|\omega|} F\left(x + \omega \cdot \vec h\right) \\
&= \partial_{h_1} \dots \partial_{h_{k+1}} F(x).
\end{align*}
Thus for instance we have
$$ \Poly^k(G) = \{ F \colon G \to \T: d^{k+1} F = 0 \}$$
for any $k \geq 0$.

In a similar spirit, a $k$-cocycle $\rho \colon G^{[k+1]} \to \T$ as defined in Definition \ref{nil-cocycle-def} can now be parameterized by $\rho_{h_1,\dots,h_{k+1}} \colon G \to \T$ for each $h_1,\dots,h_{k+1} \in G$ obeying the following two axioms:
\begin{itemize}
    \item (Symmetry) $\rho_{h_1,\dots,h_{k+1}}$ is symmetric under permutation of the parameters $h_1,\dots,h_{k+1}$.
    \item (Cocycle)  One has the identity
    \begin{equation}\label{cocycle-h1}
    \rho_{h_1+h'_1,h_2,\dots,h_{k+1}} = \rho_{h_1,h_2,\dots,h_{k+1}} + T^{h_1} \rho_{h'_1,h_2,\dots,h_{k+1}}
    \end{equation}
    for all $h_1,h'_1,h_2,\dots,h_{k+1} \in G$, where (as in Appendix \ref{nilspace-app}) $T^h$ denotes the translation operator
    $$ T^h F(x) \coloneqq F(x+h).$$
\end{itemize}

We describe the cocycle property \eqref{cocycle-h1} in terms of the first shift $h_1$ only, but of course by the symmetry property, we have cocycle behavior with respect to all the other shifts as well. 
In the language of Definition \ref{nil-cocycle-def}, $d^{k+1} F$ is a $k$-coboundary, and thus also a $k$-cocycle.  

When $G$ is an elementary abelian $2$-group, there is a further constraint on $k$-coboundaries $d^{k+1} F$, coming from the identity
\begin{equation}\label{2h}
0 = \partial_{2h} = 2\partial_h + \partial_h^2
\end{equation}
for any $h \in G$, which implies that
\begin{equation}\label{hhh}
\partial_{h_1}^2 \partial_{h_2} = \partial_{h_2}^2 \partial_{h_1}
\end{equation}
for all $h_1, h_2 \in G$.  This leads to the additional ``$2$-homogeneity'' constraint
\begin{equation}\label{dkf}
 d^{k+1} F_{h_1,h_1,h_2,h_3,\dots,h_k} = d^{k+1} F_{h_2,h_2,h_1,h_3,\dots,h_k}
 \end{equation}
 whenever $k \geq 2$ and $h_1,\dots,h_k \in G$ (our choice of terminology here is inspired by \cite{CGSS}) .  This motivates the following definition:
 
 \begin{definition}[$2$-homogeneous cocycles on elementary abelian $2$-groups]  Let $G$ be an elementary abelian $2$-group, and let $\rho \colon G^{[k+1]} \to \T$ be a $k$-cocycle for some $k \geq 0$.  If $k \geq 2$, we say that $\rho$ is \emph{$2$-homogeneous} if we have
 \begin{equation}\label{dkf-again}
 \rho_{h_1,h_1,h_2,h_3,\dots,h_k} = \rho_{h_2,h_2,h_1,h_3,\dots,h_k}
 \end{equation}
 whenever $h_1,\dots,h_k \in G$.  For $k < 2$, we declare all $k$-cocycles to automatically be $2$-homogeneous.
 \end{definition}

\begin{remark}\label{ctr}  Not all cocycles on elementary abelian $2$-groups obey the $2$-homogeneity condition \eqref{dkf-again}.  For instance, if $G = \F_2^2$ is generated by $e_1=(1,0),e_2=(0,1)$, then by letting $\rho \colon G^{[3]} \to \T$ be the symmetric trilinear form
$$ \rho_{h_1,h_2,h_3}(x) \coloneqq \frac{h_1^{(2)} h_2^{(1)} h_3^{(1)} + h_1^{(1)} h_2^{(2)} h_3^{(1)} + h_1^{(1)} h_2^{(1)} h_3^{(2)}}{2} \mod 1,$$
where $h_i = (h_i^{(1)}, h_i^{(2)}) \in G$, one can verify that $\rho$ is a $2$-cocycle on the elementary abelian $2$-group $G$ that does not obey \eqref{dkf-again}.  This $2$-cocycle will be related to a non-trivial (but now $2$-homogeneous) $5$-cocycle on the degree $2$ filtered abelian group ${\mathcal D}^2(\F_2^2)$ that we will construct in the next section.
\end{remark}

We have just established that every $k$-coboundary on an elementary abelian $2$-group is $2$-homogeneous.
We now provide a converse to this above observation when $G = \F_2^n$.  

\begin{theorem}[All $2$-homogeneous $\T$-valued cocycles are coboundaries for elementary abelian $2$-groups]\label{cocyc-trivial} Let $G = \F_2^n$ be an elementary abelian $2$-group, and let $k \geq 0$.  Then every $2$-homogeneous $k$-cocycle $\rho \colon G^{[k+1]} \to \T$ is a $k$-coboundary.
\end{theorem}

Informally, this theorem asserts that the equation $d^k F = \rho$ can be solved for some $F \colon G \to \T$ if and only if $\rho$ is a $2$-homogeneous $k$-cocycle.  This fact will be useful to us when the time comes to solve the equation \eqref{S-eq}, as discussed in the introduction.

\begin{remark}
A notable special case of this theorem occurs when $\rho_{h_1,\dots,h_k}$ is constant for each $h_1,\dots,h_k$, then the $2$-homogeneous $k$-cocycle $\rho$ is what is referred to as a \emph{non-classical symmetric multilinear form} in \cite{tidor} and a \emph{strongly symmetric multilinear form} in \cite{milicevic-u56}, and the potential $F$ produced by this theorem is then a (non-classical) polynomial of degree $k$.  This special case of Theorem \ref{cocyc-trivial} was previously established in \cite[Proposition 3.5]{tidor}.
\end{remark}

\begin{proof}
We first consider the base case $k=0$.  From the cocycle identity we have
$$ \rho_{x+h}(0) = \rho_x(0) + \rho_h(x)$$
for all $x,h \in G$.  Hence we have $\rho = dF$ where $F(x) \coloneqq \rho_x(0)$.

Now suppose inductively that $k>0$ and the claim has already been proven for $k-1$.  
For each $h_1 \in G$, the function $\rho_{h_1} \colon G^{[k]} \to \T$ defined by $(\rho_{h_1})_{h_2,\dots,h_{k+1}}(x) \coloneqq \rho_{h_1,\dots,h_{k+1}}(x)$ can be easily verified to be a $2$-homogeneous $(k-1)$-cocycle.  Hence by the induction hypothesis, there exists $F_{h_1} \colon G \to \T$ such that 
\begin{equation}\label{rhof}
\rho_{h_1} = d^{k} F_{h_1}. 
\end{equation}
Since $\rho_{h_1}$ is a cocycle in $h_1$, we have
$$ d^{k} F_{h_1+h'_1} = d^{k} F_{h_1} + T^{h_1} d^{k} F_{h'_1} $$
for all $h_1,h'_1 \in G$.  In other words, we have the quasi-cocycle condition
\begin{equation}\label{fhh}
F_{h_1+h'_1} - F_{h_1} - T^{h_1} F_{h'_1} \in \Poly^{k-1}(G).
\end{equation}
Also, from the symmetry between $h_1$ and $h_2$ of $(\rho_{h_1})_{h_2,\dots,h_{k+1}}$, we have that
$$ \partial_{h_3,\dots,h_{k+1}} (\partial_{h_2} F_{h_1} - \partial_{h_1} F_{h_2}) = 0$$
for all $h_1,\dots,h_{k+1} \in G$, or in other words we have the quasi-curlfree condition
\begin{equation}\label{curl-free}
\partial_{h_2} F_{h_1} - \partial_{h_1} F_{h_2} \in \Poly^{k-2}(G)
\end{equation}
for all $h_1,h_2 \in G$.  Finally, when $k \geq 2$, we have from \eqref{dkf} that
$$ \partial_{h_3} \dots \partial_{h_{k}} ( \partial_{h_1}^2 F_{h_2} - \partial_{h_2}^2 F_{h_1} ) = 0$$
for all $h_1,\dots,h_{k} \in G$, or equivalently
$$ \partial_{h_1}^2 F_{h_2} - \partial_{h_2}^2 F_{h_1}  \in \Poly^{k-3}(G)$$
and hence (by \eqref{2h})
\begin{equation}\label{curl-free-2}
2(\partial_{h_2} F_{h_1} - \partial_{h_1} F_{h_2}) \in \Poly^{k-3}(G).
\end{equation}
This constraint is implied by \eqref{curl-free} when $k>2$ thanks to \eqref{double}, but is not redundant for $k=2$.

We will show that the properties \eqref{fhh}, \eqref{curl-free}, \eqref{curl-free-2} imply that there exists $\phi \colon G \to \T$ such that
\begin{equation}\label{fhp}
 F_h - \partial_h \phi \in \Poly^{k-1}(G)
 \end{equation}
 for all $h \in G$.  If \eqref{fhp} holds, then by applying $d^k$ and using \eqref{rhof} we conclude that $\rho - d^{k+1} \phi = 0$, giving the claim.

It remains to establish \eqref{fhp}.  We prove this by a further induction on the dimension $n$.  The case $n=0$ is trivial, so suppose $n \geq 1$ and that the claim has already been proven for $n-1$.  Now split $G = \F_2^{n-1} \times \F_2$ and let $e = (0,1)$ be the generator for the $\F_2$ factor.  The operator $\partial_e$ is annihilated by $1+T^e$ since $(1+T^e) \partial_e = \partial_{2e} = 0$.  Also, for $k>2$, the operator $1+T^e = 2+\partial_e$ maps $\Poly^{k-2}(G)$ to $\Poly^{k-3}(G)$ thanks to \eqref{double}, hence
from \eqref{curl-free} and the previous sentence, we have 
$$ \partial_h (1+T^e) F_e = (1+T^e) \partial_h F_e \in (1+T^e) (\partial_e F_h + \Poly^{k-2}(G))\subset \Poly^{k-3}(G)$$
for all $h \in G$, hence
\begin{equation}\label{1f}
 (1+T^e) F_e \in \Poly^{k-2}(G).
 \end{equation}
The same argument works when $k=2$, where we use \eqref{curl-free-2} instead of \eqref{curl-free} to handle the $2$ component of $1+T^e = 2 + \partial_e$ applied to $\partial_h F_e - \partial_e F_h$.  The conclusion \eqref{1f} also holds when $k=1$, since in this case the expression \eqref{curl-free} vanishes.

Applying Lemma \ref{inv-lem}, we may find $F'_e \in \Poly^{k-1}(G)$ such that
$$ (1+T^e) F_e = (1+T^e) F'_e$$
Since $F_e - F'_e$ is annihilated by $1+T^e$, it sums to zero on each of the $2$-element cosets of $\langle e \rangle = \{0,e\}$, and we may therefore write
$$ F_e - F'_e = \partial_e \phi$$
for some $\phi \colon G \to \T$.   For instance, if we arbitrarily select a complementary (index two) subspace $H$ to $\langle e \rangle$ in $G$, we can set $\phi(x)\coloneqq 0$ and $\phi(x+e) \coloneqq F_e(x)-F'_e(x)$ for $x \in H$; other choices for $\phi$ are also possible.

If we then write
$$ F''_h \coloneqq F_h - F'_e - \partial_h \phi$$
we see that $F''_h$ obeys the same axioms \eqref{fhh}, \eqref{curl-free}, \eqref{curl-free-2} as $F_h$, but with the additional property that $F''_e = 0$.  In particular from \eqref{curl-free} we have
$$ \partial_e F''_{(h,0)} \in \Poly^{k-2}(G) $$
for all $h \in \F_2^{n-1}$.  Since $\partial_e F''_{(h,0)}(x,1) = - \partial_e F''_{(h,0)}(x,0)$, we thus have
$$ \partial_e F''_{(h,0)}(x,x_n) = (-1)^{x_n} G_h(x)$$
for all $x \in \F_2^{n-1}$ and some $G_h \in \Poly^{k-2}(\F_2^{n-1})$.  If we set $H_h \colon G \to \T$ be the function
$$ H_h(x,x_n) \coloneqq 1_{x_n=1} G_h(x),$$
then
\begin{equation}\label{phpf}
\partial_e H_h = \partial_e F''_{(h,0)}
\end{equation}
is a polynomial of degree $k-2$ on $G$, while
$$ \partial_{h_1} \dots \partial_{h_{k-1}} H_h = 0$$
whenever $h_1,\dots,h_{k-1} \in \F_2^{n-1}$.  From this (and Lemma \ref{cubechar}) we conclude that $H_h \in \Poly^{k-1}(G)$.  By \eqref{phpf}, $F''_{(h,0)} - H_h$ is $e$-invariant and can be thus viewed as a function on $\F_2^{n-1}$.  One then verifies that the functions $F''_{(h,0)} - H_h$ obey the same axioms \eqref{fhh}, \eqref{curl-free}, \eqref{curl-free-2} as $F_h$, but on $\F_2^{n-1}$ rather than $\F_2^n$.  Applying the inner induction hypothesis and lifting back to $G$, we can find an $e$-invariant $\phi'' \colon G \to \T$ such that
$$ F''_{(h,0)} - H_h - \partial_{(h,0)} \phi'' \in \Poly^{k-1}(G)$$
for all $h \in \F_2^{n-1}$, thus
\begin{equation}\label{fph}
 F''_h - \partial_h \phi'' \in \Poly^{k-1}(G)
 \end{equation}
for all $h \in \F_2^{n-1} \times \{0\}$.
On the other hand, from \eqref{fhh} (now applied to $F''$) and the vanishing of $F''_e$, we have 
$$ F''_{h+e} - F''_{h} \in \Poly^{k-1}(G)$$
Thus from the $e$-invariance of $\phi''$, we see that 
$$ (F''_{h+e} - \partial_{h+e} \phi'') - (F''_{h} - \partial_{h} \phi'')  \in \Poly^{k-1}(G)$$
and hence the property \eqref{fph} holds for all $h \in \F_2^n$, not just $h \in \F_2^{n-1} \times \{0\}$.  In particular,
$$ F_h - \partial_h (\phi+\phi'') \in \Poly^{k-1}(G)$$
for all $h$, thus closing the induction.
\end{proof}

The above theorem applies to cocycles taking values in $\T$.  For our application (and in particular, to solve the equation \eqref{S-eq}) we will need a variant of this theorem that applies to cocycles taking values in the smaller group $\frac{1}{2}\Z/\Z$, which is an elementary abelian $2$-group.  For this, we will need a stronger version of the $2$-homogeneity condition, which we only define for $k \geq 3$, but which we will define on more general nilspaces than elementary abelian $2$-groups with the degree $1$ filtration.

 \begin{definition}[Strongly $2$-homogeneous cocycles]\label{strong-cocycle}  Let $X$ be a finite nilspace, let $k \geq 3$, and let $\rho \colon C^{k+1}(X) \to \frac{1}{2}\Z/\Z$ be a $k$-cocycle taking values in the elementary abelian $2$-group $\frac{1}{2}\Z/\Z$.  We say that $\rho$ is \emph{strongly $2$-homogeneous} if we have $\rho = d^{k} \psi$ for some function $\psi \colon C^1(X) \to \T$ with $2\psi \in \Poly^{k-2}(C^1(X))$, where the nilspace structure on $C^1(X)$ is defined in Remark \ref{cube-nilspace}.
 \end{definition}

We first observe that strongly $2$-homogeneous cocycles on ${\mathcal D}^1(\F_2^n)$ are $2$-homogeneous (viewed as cocycles in $\T$).  Indeed, since $\rho = d^k \psi$ and $k \geq 3$, we have
$$  \rho_{h_1,h_1,h_2,h_3,\dots,h_k} = \partial_{h_1}^2 \partial_{h_2} (d^{k-3} \psi)_{h_3,\dots,h_k}$$
and
$$  \rho_{h_1,h_2,h_2,h_3,\dots,h_k} = \partial_{h_2}^2 \partial_{h_1} (d^{k-3} \psi)_{h_3,\dots,h_k}$$
and the condition \eqref{dkf-again} follows from \eqref{hhh}.  Now we obtain a variant of Theorem \ref{cocyc-trivial}.

\begin{theorem}[All strongly $2$-homogeneous cocycles are $\frac{1}{2}\Z/\Z$-valued coboundaries for elementary abelian $2$-groups]\label{cocyc-trivial-strong} Let $G = \F_2^n$ for some natural number $n$ (endowed with the degree one filtration ${\mathcal D}^1(G)$), and let $k \geq 3$.  Then a $k$-cocycle $\rho \colon G^{[k+1]} \to \frac{1}{2}\Z/\Z$ is a $k$-coboundary (in $\frac{1}{2}\Z/\Z$ rather than in $\T$) if and only if it is strongly $2$-homogeneous.
\end{theorem}

\begin{proof}   First suppose that $\rho$ is a $k$-coboundary in $\frac{1}{2}\Z/\Z$, thus $\rho = d^{k+1} F$ for some $F \colon G \to \frac{1}{2}\Z/\Z$.  Then we can write $\rho = d^{k} \psi$ with $\psi \coloneqq dF$; since $2F=0$, we have $2\psi=0$, and so $\rho$ is certainly strongly $2$-homogeneous.

Conversely, suppose that $\rho$ is strongly $2$-homogeneous.
Applying Theorem \ref{cocyc-trivial} (viewing $\rho$ as a cocycle in the larger group $\T$), we already have
$$ \rho = d^{k+1} F$$
for some $F \colon G \to \T$.  However, we are not done yet, because this function $F$ does not necessarily lie in the smaller group $\frac{1}{2}\Z/\Z$.  To address this issue, we exploit the further properties of the strongly $2$-homogeneous cocycle $\rho$.  Writing $\rho = d^k \psi$, we have
$$ d^k (dF - \psi) = 0$$
or equivalently
$$ dF - \psi \in \Poly^{k-1}(C^1(G)).$$
Multiplying by $2$ using Proposition \ref{double-prop}, we conclude that
$$ d(2F) - 2\psi \in \Poly^{k-2}(C^1(G));$$
since $2\psi$ also lies in $\Poly^{k-2}(C^1(G))$ by hypothesis, we conclude
$$ d(2F) \in \Poly^{k-2}(C^1(G))$$
or equivalently
$$ 2F \in \Poly^{k-1}(C^1(G)).$$
By \eqref{double}, we may thus write $2F = 2F'$ for some $F' \in \Poly^k(G)$.  Then $F-F'$ takes values in $\frac{1}{2}\Z/\Z$ and
$$ \rho = d^{k+1} F = d^{k+1}(F-F'),$$
giving the claim.
\end{proof}

\section{A non-trivial cocycle}\label{nontriv-sec}

Henceforth we take $k=5$ and $p=2$.  Theorem \ref{cocyc-trivial-strong} asserts, roughly speaking, there are no ``non-trivial'' $k$-cocycles on degree one filtrations ${\mathcal D}^1(\F_2^n)$, where by ``non-trivial'' we mean a $k$-cocycle which is strongly $2$-homogeneous but not a $k$-coboundary.  However, it turns out that this claim breaks down as soon as $n=2$ if one instead considers the degree two filtration ${\mathcal D}^2(\F_2^n)$.  More precisely, the main result of this section is as follows.  For the remainder of the paper, we take $X_2$ to be the $2$-step nilspace
\begin{equation}\label{Y-def}
X_2 \coloneqq {\mathcal D}^2(\F_2^2),
\end{equation}
which is also $2$-homogeneous thanks to Proposition \ref{phom}.  

\begin{theorem}[A non-trivial cocycle]\label{nontriv-cocycle}  There exists a strongly $2$-homogeneous $5$-cocycle $\rho \colon C^6(X_2) \to \frac{1}{2}\Z/\Z$ on $X_2$ taking values in $\frac{1}{2}\Z/\Z$, which is not a $5$-coboundary (when viewed as a cocycle in $\T$). 
\end{theorem}

In the remainder of this section we establish this theorem; our original discovery of this cocycle was computer-assisted, and indeed one could easily verify the claims in this theorem from standard linear algebra packages, but we provide a human-verifiable proof of this theorem below.

It will be convenient to adopt the following notation from \cite[Definitions 6.1, 6.3]{taoziegler}. 

\begin{definition}[Concatenation and symmetric square]\cite{taoziegler} If $V$ is a vector space over a field $\F$, and $S \colon V^k \to \F$ and $T \colon V^l \to \F$ are symmetric multilinear forms, we define the \emph{concatenation} $S * T \colon V^{k+l} \to \F$ to be the symmetric multilinear form
$$ S * T( h_1,\dots,h_{k+l} ) \coloneqq \sum_{\{1,\dots,k+l\} = \{i_1,\dots,i_k\} \cup \{j_1,\dots,j_l\}} S(h_{i_1},\dots,h_{i_k}) T(h_{j_1},\dots,h_{j_l})$$
and similarly define the symmetric square $\Sym^2(S) \colon V^{2k} \to \F$ to be the symmetric multilinear form
\begin{align*}
&\Sym^2(S)( h_1,\dots,h_{2k} ) \\
&\quad \coloneqq \sum_{\{ \{i_1,\dots,i_k\}, \{j_1,\dots,j_k\}\}: \{1,\dots,2k\} = \{i_1,\dots,i_k\} \cup \{j_1,\dots,j_k\}} S(h_{i_1},\dots,h_{i_k}) S(h_{j_1},\dots,h_{j_k}).
\end{align*}
\end{definition}

One can similarly define higher symmetric powers $\Sym^m(S) \colon V^{mk} \to \F$, but we will only need the $m=2$ case here.

\begin{examples} If $B \colon V^2 \to \F$ is a symmetric  bilinear form, then $\Sym^2(B) \colon V^4 \to \F$ is the symmetric quartilinear form
$$ \Sym^2(B)(a,b,c,d) \coloneqq B(a,b) B(c,d) + B(a,c) B(b,d) + B(a,d) B(b,c),$$
while if $L \colon V \to F$ is a linear form, then $L*B \colon V^3 \to \F$ is the trilinear form
$$ L * B(a,b,c) \coloneqq L(a) B(b,c) + L(b) B(a,c) + L(c) B(a,b)$$
and $B*B = 2 \Sym^2(B)$; in particular, in characteristic two we have $B*B=0$.  By identifying $\F_2$ with a subgroup of $\T$, the trilinear form in Remark \ref{ctr} can be written 
\begin{equation}\label{rhoh}
\rho_{h_1,h_2,h_3}(x) = \frac{\Sym^2(L_1) * L_2(h_1,h_2,h_3)}{2} \mod 1,
\end{equation}
where $L_1,L_2 \colon \F_2^2 \to \F_2$ are the coordinate functions $L_i(x_1,x_2) \coloneqq x_i$. 
\end{examples}

A $6$-cube in $X_2 = {\mathcal D}^2(\F_2^2)$ can be viewed as a pair $(Q^{(1)}, Q^{(2)})$, where $Q^{(1)}, Q^{(2)} \colon \F_2^6 \to \F_2$ are quadratic polynomials (cf. Definition \ref{poly-def}), so in particular their second derivatives can be viewed as symmetric bilinear forms $d^2 Q^{(i)} \colon \F_2^6 \times \F_2^6 \to \F_2$, defined for $i=1,2$ by the formula
$$ d^2 Q^{(i)} (h, k) \coloneqq \partial_h \partial_k Q^{(i)}$$
(note that the right-hand side is a constant and thus identifiable with an element of $\F_2$).  We then define the cocycle $\rho$ by
\begin{equation}\label{rhoqf}
\rho(Q^{(1)}, Q^{(2)}) \coloneqq \frac{\Sym^2(d^2 Q^{(1)}) * (d^2 Q^{(2)})(e_1,\dots,e_6)}{2} \mod 1
\end{equation}
with $e_1,\dots,e_6$ the standard basis of $\F_2^6$; observe the analogy with the construction in \eqref{rhoh}.

One can describe $\rho$ more explicitly as follows.  Instead of using the pair $(Q^{(1)}, Q^{(2)})$, one can alternatively parameterize a $6$-cube in $X_2$ as a tuple
\begin{equation}\label{6-cube}
 \left(x + \sum_{i=1}^6 h_i \omega_i + \sum_{1 \leq i < j \leq 6} h_{ij} \omega_i \omega_j\right)_{\omega \in \{0,1\}^6}
 \end{equation}
for some $x, h_i, h_{ij} \in X_2$. We write $x$ in coordinates as $x = (x^{(1)}, x^{(2)})$ for $x^{(1)}, x^{(2)} \in \F_2$, and similarly write $h_i = (h_i^{(1)}, h_i^{(2)})$ and $h_{ij} = (h_{ij}^{(1)}, h_{ij}^{(2)})$; the polynomials $Q^{(k)}$, $k=1,2$ in the previous description of a $6$-cube in $X_2$ then take the form
$$ Q^{(k)}(\omega_1,\dots,\omega_6) = x^{(k)} + \sum_{i=1}^6 h^{(k)}_i \omega_i + \sum_{1 \leq i < j \leq 6} h^{(k)}_{ij} \omega_i \omega_j,$$
so in particular
$$ d^2 Q^{(k)}( \omega, \omega' ) = \sum_{1 \leq i < j \leq 6} h^{(k)}_{ij} (\omega_i \omega'_j + \omega'_i \omega_j)$$
for $\omega = (\omega_1,\dots,\omega_6)$, $\omega' = (\omega'_1,\dots,\omega'_6)$ in $\F_2^6$.
From \eqref{rhoqf} we conclude that the cocycle $\rho$ applied to the $6$-cube \eqref{6-cube} is then given by the formula
\begin{equation}\label{had}
\begin{split}
&\rho \left( \left(x + \sum_{i=1}^6 h_i \omega_i + \sum_{1 \leq i < j \leq 6} h_{ij} \omega_i \omega_j\right)_{\omega \in \{0,1\}^6} \right)\\ &\quad \coloneqq \frac{\sum_{\{ \{a,b\},\{c,d\}\}, \{e,f\}: \{1,\dots,6\} = \{a,b\} \cup \{c,d\} \cup \{e,f\}} h^{(1)}_{ab} h^{(1)}_{cd} h^{(2)}_{ef}}{2} \mod 1 
\end{split}
\end{equation}
where the sum is over the $\frac{1}{2!} \frac{6!}{2! 2! 2!} = 45$ different ways one can partition $\{1,\dots,6\}$ into three doubleton sets $\{a,b\}, \{c,d\}, \{e,f\}$, where we only sum once for each choice of $\{ \{a,b\}, \{c,d\}\}$ and $\{e,f\}$ (so that each monomial $h^{(1)}_{ab} h^{(1)}_{cd} h^{(2)}_{ef}$ occurs at most once).   

The function $\rho$ is clearly symmetric with respect to permutations of the indices $1,\dots,6$.  If we fix the $h_{ij}$ for $1 < i < j \leq 6$, then this function is linear in the remaining variables $h_{1i}$, $1 < i < 6$, from which it is easy to verify that $\rho$ obeys the cocycle property in Definition \ref{nil-cocycle-def}(ii).  Thus $\rho$ is a $5$-cocycle.

Suppose for contradiction that $\rho$ is a $5$-coboundary, thus there is a function $F \colon X_2 \to \T$ such that
\begin{equation}\label{help}
\begin{split}
& \rho \left( \left(x + \sum_{i=1}^6 h_i \omega_i + \sum_{1 \leq i < j \leq 6} h_{ij} \omega_i \omega_j\right)_{\omega \in \{0,1\}^6} \right) \\
&\quad = \sum_{\omega \in \{0,1\}^6} (-1)^{|\omega|} F\left( x + \sum_{i=1}^6 h_i \omega_i + \sum_{1 \leq i < j \leq 6} h_{ij} \omega_i \omega_j \right)
\end{split}
\end{equation}
whenever $x, h_i, h_{ij} \in X_2$.  We now descend from this sixth order equation on $X_2={\mathcal D}^2(\F_2^2)$ to a third order equation on ${\mathcal D}^1(\F_2^2)$ as follows. We restrict to those cubes in which all the $h_i$ and $h_{ij}$ vanish except for $h_{12}, h_{34}, h_{56}$, which we relabel as $k_1, k_2, k_3$ respectively.  Then the right-hand side of \eqref{help} cancels down to
$$ \sum_{\omega \in \{0,1\}^3} (-1)^{3-|\omega|} F\left( x + \sum_{i=1}^3 k_i \omega_i \right)$$
while the right-hand side of \eqref{had} simplifies to
$$ \frac{k^{(1)}_1 k^{(1)}_2 k^{(2)}_3 + k^{(1)}_1 k^{(2)}_2 k^{(1)}_3 + k^{(2)}_1 k^{(1)}_2 k^{(1)}_3}{2} \mod 1$$
and hence on ${\mathcal D}^1(\F_2^2)$  we have
$$ (d^3 F)_{k_1,k_2,k_3} = \frac{k^{(1)}_1 k^{(1)}_2 k^{(2)}_3 + k^{(1)}_1 k^{(2)}_2 k^{(1)}_3 + k^{(2)}_1 k^{(1)}_2 k^{(1)}_3}{2} \mod 1$$
for all $k_1,k_2,k_3 \in \F_2^2$.
However, as observed in Remark \ref{ctr}, the right-hand side does not obey the $2$-homogeneity condition \eqref{dkf-again} and so cannot be a coboundary on ${\mathcal D}^1(\F_2^2)$, giving the desired contradiction.

Finally, we need to show that $\rho = d^5 \psi$ for some $\psi \colon X_2^{[1]} \to \T$ with $2\psi$ a cubic polynomial.  We let $[] \colon \F_2 \to \Z/4\Z$ be any section of the projection map $\mod 2 \colon \Z/4\Z \to \F_2$; in particular one has $[0]^2 = 0 \mod 4$ and $[1]^2 = 1 \mod 4$ regardless of the choice of section.  An element of $C^1(X_2)$ takes the form $(x,x+h)$ with $x,h \in \F_2^2$.  We write $x = (x^{(1)}, x^{(2)})$, $h = (h^{(1)}, h^{(2)})$ and define $\psi$ by the formula
\begin{equation}\label{def-psi}
\psi(x,x+h) \coloneqq \frac{[x^{(1)}]^2 [h^{(2)}]^2}{4} + \frac{x^{(1)} h^{(1)} x^{(2)}}{2} \mod 1.
\end{equation} 
We first verify that $2\psi$ is a cubic polynomial.  Since $[x]^2=x^2=x \mod 2$, we have
$$ 2 \psi(x,x+h) = \frac{x^{(1)} h^{(2)}}{2}\ \mod 1.$$
According to Remark \ref{cube-nilspace}, a $4$-cube in $C^1(X_2)$ corresponds to a $5$-cube in $X_2$, and by Lemma \ref{lem-filtgrpnilspace}, a $5$-cube in $X_2$ can be computed according to the formula in \eqref{tup}, and takes the form
$$ \left( \left(x + \sum_{i=1}^4 h_i \omega_i + \sum_{1 \leq i < j \leq 4} h_{ij} \omega_i \omega_j, x + \sum_{i=1}^4 h_i \omega_i + \sum_{1 \leq i < j \leq 4} h_{ij} \omega_i \omega_j + h_0 + \sum_{i=1}^4 h_{0i} \omega_i \right) \right)_{\omega \in \{0,1\}^4}$$
for some $x, h_0, h_i, h_{0i}, h_{ij} \in X_2$ (cf. Example \ref{example-cubes}).  The function $d^4(2\psi)$ applied to this cube is then equal to
$$ \sum_{\omega \in \{0,1\}^4} (-1)^{|\omega|}
\frac{(x^{(1)} + \sum_{i=1}^4 h^{(1)}_i \omega_i + \sum_{1 \leq i < j \leq 4} h^{(1)}_{ij} \omega_i \omega_j) (h^{(2)}_0 + \sum_{i=1}^4 h^{(2)}_{0i} \omega_i)}{2} \mod 1.$$
But the numerator is cubic in the $\omega_i$ and thus does not contain any monomials of the form $\omega_1 \omega_2 \omega_3 \omega_4$.  This expression therefore vanishes, and so $2\psi$ is cubic as required.

It remains to show that $\rho = d^5 \psi$.  As before, from Remark \ref{cube-nilspace}, Lemma \ref{lem-filtgrpnilspace}, and the formula in \eqref{tup}, a $5$-cube in $C^1(X_2)$ takes the form 
$$ \left( \left(x + \sum_{i=1}^5 h_i \omega_i + \sum_{1 \leq i < j \leq 4} h_{ij} \omega_i \omega_j, x + \sum_{i=1}^5 h_i \omega_i + \sum_{1 \leq i < j \leq 5} h_{ij} \omega_i \omega_j + h_0 + \sum_{i=1}^5 h_{0i} \omega_i \right) \right)_{\omega \in \{0,1\}^5}$$
for some $x, h_0, h_i, h_{0i}, h_{ij} \in X_2$.  By the definition \eqref{def-psi} of $\psi$, the function $d^5 \psi$ applied to this cube is the sum of
\begin{equation}\label{term-1}
 \sum_{\omega \in \{0,1\}^5} (-1)^{5-|\omega|}
\frac{[X^{(1)}(\omega)]^2 [H^{(2)}(\omega)]^2}{4} \mod 1
\end{equation}
and
\begin{equation}\label{term-2}
 \sum_{\omega \in \{0,1\}^5} (-1)^{5-|\omega|}
\frac{X^{(1)}(\omega) H^{(1)}(\omega) X^{(2)}(\omega)}{2} \mod 1
\end{equation}
where
$$ X^{(a)}(\omega) \coloneqq x^{(a)} + \sum_{i=1}^5 h^{(a)}_i \omega_i + \sum_{1 \leq i < j \leq 5} h^{(a)}_{ij} \omega_i \omega_j$$
and
$$ H^{(a)}(\omega) \coloneqq h^{(a)}_0 + \sum_{i=1}^5 h^{(a)}_{0i} \omega_i$$
for $a=1,2$.  We first consider \eqref{term-2}.  The numerator $X^{(1)}(\omega) H^{(1)}(\omega) X^{(2)}(\omega)$ is quintic in the variables $\omega_i$ (when viewed as a function from $\F_2^5$ to $\F_2$ for a fixed choice of the $x, h_0, h_i, h_{0i}, h_{ij}$). We have the identity 
\[
\sum_{\omega \in \{0,1\}^5} (-1)^{5 - |\omega|} \prod_{i \in T} \omega_i
=
\begin{cases}
1, & \text{if } T = \{1,2,\dots,5\}, \\[6pt]
0, & \text{otherwise.}
\end{cases}
\]
so only monomials that use all five distinct variables $\omega_1 \dots \omega_5$ survive, and the sum of the coefficients of these monomials after expanding out all the definitions can be expressed as
\begin{equation}\label{ab1}
\frac{\sum^* h^{(1)}_{ab} h^{(1)}_{cd} h^{(2)}_{ef}}{2} \mod 1
\end{equation}
where the sum $\sum^*$ ranges over the $30$ pairs of sets $\{\{a,b\}, \{c,d\}\}, \{e,f\}$ with $\{0,1,2,3,4,5\} = \{a,b\} \cup \{c,d\} \cup \{e,f\}$ such that $0$ lies in one of $\{a,b\}$ or $\{c,d\}$. 

Now consider \eqref{term-1}.  Using the easily verified identities $[a+b]^2 = [a]^2 + [b]^2 + 2[ab]$ and $[a \omega]^2 = [a]^2\omega$ for $a,b \in \F_2$ and $\omega \in \{0,1\}$ (and noting that the map $a \mapsto 2[a]$ is an additive homomorphism), we can expand out 
$$
[H^{(2)}(\omega)]^2 = \left[h^{(2)}_0 + \sum_{i=1}^5 h^{(2)}_{0i} \omega_i\right]^2 = [h^{(2)}_0]^2 + \sum_{i=1}^5 [h^{(2)}_{0i}]^2 \omega_i + 2 [Q(\omega)]$$
where $Q \colon \{0,1\}^5 \to \F_2$ is the quadratic
$$ Q(\omega) \coloneqq \sum_{i=1}^5 h^{(2)}_0 h^{(2)}_{0,1} \omega_i + \sum_{1 \leq i < j \leq 5} h^{(2)}_{0i} h^{(2)}_{0j} \omega_i \omega_j,$$
and similarly
$$ [X^{(1)}(\omega)]^2
= [x^{(1)}]^2 + \sum_{i=1}^5 [h^{(1)}_i]^2 \omega_i + \sum_{1 \leq i < j \leq 5} [h^{(1)}_{ij}]^2 \omega_i \omega_j + 2[R(\omega)]$$
where $R\colon \{0,1\}^5 \to \F_2$ is the quartic
\begin{align*}
R(\omega) &\coloneqq \sum_{i=1}^5 x^{(1)} h^{(1)}_i \omega_i + \sum_{1 \leq i < j \leq 5} (h^{(1)}_i h^{(1)}_j + x^{(1)} h^{(1)}_{ij}) \omega_i \omega_j \\
&\quad + \sum_{1 \leq i < j < k \leq 5} (h^{(1)}_i h^{(1)}_{jk} + h^{(1)}_j h^{(1)}_{ik} + h^{(1)}_k h^{(1)}_{ij}) \omega_i \omega_j \omega_k \\
&\quad + \sum_{1 \leq i < j < k < l \leq 5} (h^{(1)}_{ij} h^{(1)}_{kl} + h^{(1)}_{ik} h^{(1)}_{jl} + h^{(1)}_{il} h^{(1)}_{jk}) \omega_i \omega_j \omega_k \omega_l.
\end{align*}
The product $[X^{(1)}(\omega)]^2 [H^{(2)}(\omega)]^2$ is then quintic (the product of $2Q$ and $2R$ would be sextic, but vanishes modulo $4$), and the alternating sum $\sum_{\omega \in \{0,1\}^5} (-1)^{5-|\omega|}$ is then extracting the $\omega_1 \dots \omega_5$ coefficient, which can only arise from the terms
$$ 2[R(\omega)] \cdot \sum_{i=1}^5 [h^{(2)}_{0i}]^2 \omega_i$$
in the numerator.  Inspecting the quartic terms of $R(\omega)$, we conclude that \eqref{term-1} is of the form
\begin{equation}\label{ab2} \frac{\sum^{**} h^{(1)}_{ab} h^{(1)}_{cd} h^{(2)}_{ef}}{2} \mod 1
\end{equation}
where the sum $\sum^{**}$ ranges over the $15$ pairs of sets $\{\{a,b\}, \{c,d\}\}, \{e,f\}$ with $\{0,1,2,3,4,5\} = \{a,b\} \cup \{c,d\} \cup \{e,f\}$ such that $0$ does not lie in either $\{a,b\}$ or $\{c,d\}$. 
Summing \eqref{ab1}, \eqref{ab2}, we obtain the claim.
This concludes the proof of Theorem \ref{nontriv-cocycle}. 

\section{Two key nilspaces}

We now use the non-trivial cocycle $\rho$ introduced in the previous section to construct a family of finite $5$-step nilspaces $X_{5,r}$ for $1 \leq r \leq 5$ that will play a key role in our counterexamples.  To prove our main result in Theorem \ref{counter} we will use the larger and more complicated nilspace $X_{5,5}$, however in Proposition \ref{injcont} below we obtain a weaker counterexample with significantly less effort using the smaller and simpler nilspace $X_{5,1}$.

Fix $1 \leq r \leq 5$.  We define $X_2$ by \eqref{Y-def}, and let $\rho$ be the non-trivial cocycle from Theorem \ref{nontriv-cocycle}.  We define the nilspace $X_{5,r}$ to be the Cartesian product
$$ X_{5,r} \coloneqq X_2 \times \frac{1}{2^r}\Z/\Z$$
with the $n$-cubes $C^n(X_{5,r})$ defined to be the space of all tuples $((Q,S)(\omega))_{\omega \in \{0,1\}^n}$, where $Q \colon \F_2^n \to X_2$ and $S \colon \F_2^n \to \frac{1}{2^r}\Z/\Z$ are functions (identifying $\{0,1\}^n$ with $\F_2^n$) that obey the following axioms:
\begin{itemize}
    \item[(i)]  $Q$ is a nilspace morphism from $\F_2^n$ to $X_2$ (or equivalently by \eqref{nbox}, that $Q \in C^n(X_2)$).  In other words, $Q = (Q_1,Q_2) \in \Poly^2(\F_2^n \to \F_2^2)$ is a pair of classical quadratic polynomials $Q_1,Q_2 \colon \F_2^n \to \F_2$.  In particular, one has $d^3 Q = 0$.
    \item[(ii)]  $S$ obeys the equation \eqref{S-eq}
for all $x,h_1,\dots,h_6 \in \F_2^n$.  Equivalently, one has $d^6 S = Q^* \rho$, where $Q^* \rho \colon C^6(\F_2^n) \to \frac{1}{2}\Z/\Z$ is the pullback of $\rho$, defined by
$$ Q^* \rho( (x_\omega)_{\omega \in \{0,1\}^6} ) \coloneqq \rho( (Q(x_\omega))_{\omega \in \{0,1\}^6} ).$$
\end{itemize}
More succinctly, one has
\begin{equation}\label{eq-5cubes}
    C^n(X_{5,r}) = \{ (Q,S) \colon \F_2^n \to X_{5,r}: d^3 Q = 0; d^6 S = Q^* \rho \}.
\end{equation}

We will shortly verify that $X_{5,r}$ is indeed a nilspace, but first we establish an important lemma that exploits the strong $2$-homogeneity of $\rho$ to allow one to lift $n$-cubes in $X_2$ to $n$-cubes in $X_{5,r}$.

\begin{lemma}[Lifting lemma]\label{lift}  Let $r \geq 1$ and $n \geq 0$, and let $Q \in C^n(X_2)$.  Then there exists a map $S \colon \F_2^n \to \frac{1}{2^r}\Z/\Z$ such that $(Q,S) \in C^n(X_{5,r})$.  Furthermore, the set of such $S$ forms a coset of $\Poly^5(\F_2^n \to \frac{1}{2^r}\Z/\Z)$.
\end{lemma}

\begin{proof}  We first show existence.  Since $\rho$ is a strongly $2$-homogeneous $5$-cocycle, it is not difficult to see that the pullback $Q^* \rho$ is also.  Hence by Theorem \ref{cocyc-trivial-strong}, $Q^* \rho$ is a $5$-coboundary in $\frac{1}{2}\Z/\Z$, thus there exists $S \colon \F_2^n \to \frac{1}{2}\Z/\Z$ such that $Q^* \rho = d^6 S$, which is precisely the condition \eqref{S-eq}.  Since $\frac{1}{2}\Z/\Z$ is contained in $\frac{1}{2^r}\Z/\Z$, we have obtained an $n$-cube $(Q,S)$ in $X_{5,r}$ as required.

Now suppose that $(Q,S), (Q,S')$ are both $n$-cubes in $X_{5,r}$.  Then $d^6 S = d^6 S' = Q^* \rho$ and hence $d^6 (S-S') = 0$, thus $S$ and $S'$ differ by an element of $\Poly^5(\F_2^n \to \frac{1}{2^r}\Z/\Z)$.  Reversing these implications, we see that the set of $S$ for which $(Q,S) \in C^n(X_{5,r})$ is a coset of $\Poly^5(\F_2^n \to \frac{1}{2^r}\Z/\Z)$ as claimed.
\end{proof}

In fact, with the specific choice of cocycle we have constructed, we can explicitly describe the coset in Lemma \ref{lift}.  

\begin{lemma}[Explicit description of lift]\label{expl}  Let the notation and hypotheses be as in Lemma \ref{lift}.  Write $Q = (Q_1, Q_2)$, thus $Q_1, Q_2 \colon \F_2^n \to \F_2$ are classical quadratic polynomials.  Let $R \in \Poly^3(\F_2^n \to \Z/4\Z)$ be a cubic polynomial such that $2 \frac{R}{4} = \frac{Q_1}{2}\mod 1$ (or equivalently that $R = Q_1 \mod 2$); the existence of such a polynomial follows from \eqref{double}.  Then the coset of $S$ in Lemma \ref{lift} is equal to
$$ \frac{\binom{R}{2} Q_2}{2} + \Poly^5\left(\F_2^n \to \frac{1}{2^r}\Z/\Z\right)$$
where (as in Section \ref{notation-sec}) $\binom{a}{2} \in \F_2$ is equal to $1$ when $a = 2,3 \mod 4$ and $0$ for $a=0,1 \mod 4$.
\end{lemma}

\begin{proof} By Lemma \ref{lift}, it suffices to show that
$$\partial_{h_1} \dots \partial_{h_6} \frac{\binom{R}{2} Q_2}{2}(x) = \rho( (Q(x + \omega \cdot \vec h))_{\omega \in \{0,1\}^6} ) \mod 1$$
for $x \in \F_2^n$ and $\vec h = (h_1,\dots,h_6) \in (\F_2^n)^6$.  By \eqref{had}, it suffices to show that 
\begin{equation}\label{had2}
   \partial_{h_1} \dots \partial_{h_6} \left(\binom{R}{2} Q_2\right) = \sum_{\{\{a,b\},\{c,d\}\},\{e,f\}: \{1,\dots,6\} = \{a,b\} \cup \{c,d\} \cup \{e,f\}} (\partial_{h_a} \partial_{h_b} Q_1) (\partial_{h_c} \partial_{h_d} Q_1) (\partial_{h_e} \partial_{h_f} Q_2) 
\end{equation}
for all $h_1,\dots,h_6 \in \F_2^n$.
The expressions in parentheses on the right-hand side are all constants since $Q$ is quadratic.  Iteratively applying the Leibniz rule \eqref{leibniz} and since $\partial_h\partial_k = \partial_k\partial_h$, we have   
\[
\partial_{h_1}\cdots\partial_{h_6}(P_1 P_2)
=
\sum_{\substack{S,T \subseteq \{1,\dots,6\} \\ S \cup T = \{1,\dots,6\}}}
\biggl( \prod_{i \in S} \partial_{h_i} P_1 \biggr)
\biggl( \prod_{j \in T} \partial_{h_j} P_2 \biggr).
\]
Since $\binom{R}{2}$ is a quartic\footnote{To justify that $\binom{R}{2}$ is quartic, observe that for any $a,b \in \Z/4\Z$ one has
\[
\binom{a+b}{2} \equiv \binom{a}{2} + \binom{b}{2} + (a \bmod 2)(b \bmod 2) \mod{2},
\]
which implies the derivative identity
\begin{equation}\label{fh}
\partial_h \binom{F}{2} = \binom{\partial_h F}{2} + F \partial_h F \mod 2
\end{equation}
for any $F \colon \F_2^n \to \Z/4\Z$ and $h \in \F_2^n$.  Applying this with $F = R$ and iterating in the directions $h_1,\dots,h_5$, and using the Leibniz rule \eqref{leibniz} to expand derivatives of products, one finds that every term in $\partial_{h_1}\cdots\partial_{h_5} \binom{R}{2}$ contains either a fourth derivative of $R$ or a third derivative of $R$ modulo $2$.  These vanish because $R$ is cubic ($d^4 R = 0$) and $R \equiv Q_1 \pmod{2}$ with $Q_1$ quadratic (so $d^3 R \equiv d^3 Q_1 = 0 \pmod{2}$).  Thus $d^5 \binom{R}{2} = 0$, and hence $\binom{R}{2}$ has degree at most $4$.} and $Q_2$ is a quadratic, it follows from the previous identity that in order to show \eqref{had2} it suffices to show that
$$ \partial_{h_1} \dots \partial_{h_4} \binom{R}{2} = \sum_{\{\{a,b\},\{c,d\}\}: \{1,\dots,4\} = \{a,b\} \cup \{c,d\}} (\partial_{h_a} \partial_{h_b} Q_1) (\partial_{h_c} \partial_{h_d} Q_1)$$
or equivalently that $d^4 \binom{R}{2} = \Sym^2( d^2 Q_1 )$. But this latter identity was established in \cite[Example 6.5]{taoziegler} as a consequence of \cite[Lemma 6.4]{taoziegler}. (Alternatively, one can establish the identity  $d^4 \binom{R}{2} = \Sym^2( d^2 Q_1 )$ by several direct applications of the Leibniz rule \eqref{leibniz} using the identity \eqref{fh} for any $F \colon \F_2^n \to \Z/4\Z$ and $h \in \F_2^n$, as well as the identities $R = Q_1 \mod 2$, $d^3 Q_1 = 0$, and $d^4 R = 0$.)
\end{proof}

\begin{proposition} Let $1 \leq r \leq 5$.  Then $X_{5,r}$ is an ergodic $2$-homogeneous $5$-step nilspace, and the projection map $\pi \colon X_{5,r} \to X_2$ given by $\pi(q,s) \coloneqq q$ for $q \in X_2$ and $s \in \frac{1}{2^r}\Z/\Z$ is a nilspace morphism.
\end{proposition}

\begin{proof}  We begin by verifying the nilspace axioms from Definition \ref{nilspace-def}.  The composition axiom is easy: if $(Q,S) \colon \F_2^n \to X_{5,r}$ is an $n$-cube in $X_{5,r}$ and $\phi \colon \{0,1\}^m \to \{0,1\}^n$ is a cube morphism, then one can view $\phi$ as an affine map from $\F_2^m$ to $\F_2^n$, in which case it is clear that $(Q,S) \circ \phi \colon \F_2^m \to X_{5,r}$ is an $m$-cube in $X_{5,r}$.

Now we verify ergodicity.  Let $(Q,S) \colon \F_2 \to X_{5,r}$ be an arbitrary map.  Then $Q$ is linear, so certainly $d^3 Q = 0$.  Since $r \leq 5$, every map $S \colon \F_2 \to \frac{1}{2^r}\Z/\Z$ lies in $\Poly^5(\F_2 \to \frac{1}{2^r}\Z/\Z)$ by Lemma \ref{explicit-desc}, and hence by Lemma \ref{lift} all pairs $(Q,S)$ lie in $C^1(X_{5,r})$, giving the claim.

Now we verify the corner completion axiom.  Let $(Q,S) \colon \F_2^n \backslash \{1\}^n \to X_{5,r}$ be a map such that the restriction of $(Q,S)$ to any $(n-1)$-face of $\{0,1\}^n \equiv \F_2^n$ containing $0^n$ is in $C^{n-1}(X_{5,r})$.  From the corner completion property of $X_2$, we may extend $Q$ to an $n$-cube $Q \colon \F_2^n \to X_2$, and then by Lemma \ref{lift} we can find a lift $(Q,S') \colon \F_2^n \to X_{5,r}$ which is an $n$-cube.  By \eqref{S-eq}, we conclude that the difference $S-S' \colon \F_2^n \backslash \{1\}^n \to \frac{1}{2^r}\Z/\Z$ is a degree $5$ polynomial on each $(n-1)$-face of $\{0,1\}^n \equiv \F_2^n$ containing $0^n$.  By the corner completion property of ${\mathcal D}^5( \frac{1}{2^r}\Z/\Z )$, we may extend $S-S'$ to a degree $5$ polynomial from $\F_2^n$ to $\frac{1}{2^r}\Z/\Z$; the resulting extension $S \colon \F_2^n \to 
\frac{1}{2^r}\Z/\Z$ then obeys \eqref{S-eq}, so that $(Q,S)$ is now extended to an $n$-cube on $X_{5,r}$ as required.  When $n=6$, the extension of $Q$ is unique, and equation \eqref{S-eq} (with $x=0$ and $h_1,\dots,h_6$ the standard basis) also shows that the extension of $S$ is unique, so that $X_{5,r}$ is $5$-step as claimed.

The nilspace morphism property of $\pi$ is clear from chasing definitions, so it remains to verify $2$-homogeneity.  Let $(Q,S) \colon \F_2^n \to X_{5,r}$ be an $n$-cube in $X_{5,r}$; we need to show that $(Q,S)$ is also a nilspace morphism from ${\mathcal D}^1(\F_2^n)$ to $X_{5,r}$.  But an $m$-cube in ${\mathcal D}^1(\F_2^n)$ can be viewed as an affine map $\phi \colon \F_2^m \to \F_2^n$, and then $(Q,S) \circ \phi \colon \F_2^m \to X_{5,r}$ can then be easily verified to obey the axioms (i), (ii) for an $m$-cube in $X_{5,r}$, and so $(Q,S)$ is a nilspace morphism as claimed.
\end{proof}

\begin{remark}  When $r=1$, one can think of $X_{5,1}$ as the skew product $X_2 \rtimes^{(5)}_\rho \frac{1}{2}\Z/\Z$, in the sense of Proposition \ref{skewprod}, and the fact that $X_{5,1}$ is a $2$-homogeneous nilspace can also be established from Lemma \ref{homog-preserv} and Lemma \ref{lift} in this case.  For larger values of $r$, however, the situation is more complicated; the nilspace $X_{5,r}$ appears at first glance to be a $5$-extension of $X_2$ by $\frac{1}{2^r}\Z/\Z$, but the cube structure is slightly smaller than what would arise from such an extension (the equation \eqref{S-eq} provides more constraints on $S$ than the constraint \eqref{z} used to define a skew product, because the shifts $h_1,\dots,h_6$ are not required to be distinct basis vectors).  Instead, by making the (slightly artificial) identification
$$ (q,s) \equiv ( (q,2s), s - \{2s\}/2 )$$
between $X_2 \times \frac{1}{2^r}\Z/\Z$ and $(X_2 \times \frac{1}{2^{r-1}}\Z/\Z) \times (\frac{1}{2}\Z/\Z)$, where $\{\} \colon \R/\Z \to [0,1)$ denotes the fractional part map, we can identify $X_{5,r}$ with the skew product
$$ (X_2 \times \frac{1}{2^{r-1}}\Z/\Z) \rtimes^{(5)}_{\tilde \rho} \frac{1}{2}\Z/\Z$$
where we give $\frac{1}{2^{r-1}}\Z/\Z$ the $2$-adic filtration $(\frac{1}{2^{r-1}}\Z/\Z)_i = \frac{1}{2^{\min(r-i,0)}}\Z/\Z$ for $i \geq 1$ (so that $X_2 \times \frac{1}{2^{r-1}}\Z/\Z$ is a $\max(2,r-1)$-step filtered abelian group), and $\tilde \rho \colon C^6(X_2 \times \frac{1}{2^{r-1}}\Z/\Z) \to \frac{1}{2}\Z/\Z$ is the modified cocycle
$$ \tilde \rho( (q_\omega,t_\omega)_{\omega \in \{0,1\}^6} ) \coloneqq \rho((q_\omega)_{\omega \in \{0,1\}^6}) - \sum_{\omega \in \{0,1\}^6} (-1)^{|\omega|} \{t_\omega\}/2$$
for all $6$-cubes $(q_\omega,t_\omega)_{\omega \in \{0,1\}^6}$ in $X_2 \times \frac{1}{2^{r-1}}\Z/\Z$ (one can check that $2\tilde \rho=0$, so that this cocycle does indeed take values in $\frac{1}{2}\Z/\Z$).  As we will not need this description of $X_{5,r}$ here, we leave the justification of this claim to the interested reader.
\end{remark}

As an application of the smaller $X_{5,1}$ of the two nilspaces $X_{5,r}$, we have

\begin{proposition}\label{injcont}  There is no injective nilspace morphism from $X_{5,1}$ to a $5$-step compact filtered abelian group.
\end{proposition}

This gives a negative answer (in the case $p=2,k=5$) to \cite[Question 5.18]{CGSS}, which asked the more general question of whether every $k$-step compact $p$-homogeneous nilspace has an injective nilspace morphism into a $k$-step compact filtered abelian group.  
As noted in that paper, an affirmative answer to this question for a given value of $p$ and $k$ would imply an affirmative answer to Conjecture \ref{btz-conj} (and hence Conjecture \ref{inv-conj-strong} and Conjecture \ref{inv-conj}) for those values of $p,k$.  Indeed, \cite[Question 5.18]{CGSS} was answered affirmatively for $k \leq p+1$, leading to the corresponding results on Conjectures \ref{btz-conj}, \ref{inv-conj-strong}, \ref{inv-conj} mentioned in the introduction.  Thus, Proposition \ref{injcont} can be viewed as a weaker version of Theorem \ref{counter}.

\begin{proof}  Suppose for contradiction that there was an injective nilspace morphism $\iota \colon X_{5,1} \to G$ from $X_{5,1}$ to some $5$-step compact filtered abelian group $G$.  Let $\mu$ be the finite measure on $G$ defined via Riesz representation as
$$ \int_G f\ d\mu \coloneqq \sum_{(q,s) \in X_{5,1}} f(\iota(q,s)) e(s).$$
This is a non-trivial measure, hence must have a non-zero Fourier coefficient.  In other words, there exists a continuous homomorphism $\xi \colon G \to \T$ such that
$$ \sum_{(q,s) \in X_{5,1}} e(s - P(q,s)) \neq 0$$
where $P \colon X_{5,1} \to \T$ is the map $P \coloneqq \xi \circ \iota$.  By Lemma \ref{k-suffices}, $\xi$ is a quintic polynomial  on the $5$-step filtered abelian group $G$, hence $P$ is a quintic polynomial on $X_{5,1}$.

Now we consider the ``vertical derivative''
$$ \partial_u P(q,s) \coloneqq P(q,s+\frac{1}{2}) - P(q,s)$$
of the polynomial $P$.  We claim that this derivative is constant.  If $(q_0,s_0), (q_1,s_1) \in X_{5,1}$, then the tuple $(Q,S) \colon \F_2^6 \to X_{5,1}$ defined by
$$ (Q,S)(\omega) \coloneqq \left(q_{\omega_1}, s_{\omega_1} + 1_{\omega_2=\dots=\omega_{6}=0} \frac{1}{2}\right)$$
can be verified to obey the axioms (i), (ii) required to be a $6$-cube in $X_{5,1}$ by checking the conditions in the description \eqref{eq-5cubes}.  From the quintic nature of $P$ we conclude that
$$ \sum_{\omega \in \{0,1\}^6} (-1)^{|\omega|} P( (Q,S)(\omega) ) = 0$$
which simplifies to
\begin{equation}\label{pup}
\partial_u P(q_0,s_0) = \partial_u P(q_1,s_1),
\end{equation}
giving the claim.

Another way of phrasing this is that the function $e(P)$ is an eigenfunction of the vertical Koopman operator $V^u$ defined by
$$ V^u F(q,s) \coloneqq F\left(q,s+\frac{1}{2}\right).$$
On the other hand, the function $(q,s) \mapsto e(s)$ is also an eigenfunction of this operator with eigenvalue $e(\frac{1}{2})$.  Since the Koopman operator $V^u$ is unitary, therefore its eigenvectors with distinct eigenvalues are orthogonal, and $e(P)$ has a non-zero inner product with $e(s)$, the eigenvalue of $e(P)$ must also be $e(\frac{1}{2})$, thus 
$$ \partial_u P = \frac{1}{2}.$$
Equivalently, we may write
$$ P(q,s) = s - F(q)$$
for some function $F \colon X_2 \to \T$.  Applying $d^6$ to eliminate the quintic polynomial $P$, we conclude that
$$ 0 = \rho - d^6 F$$
and hence $\rho$ is a $5$-coboundary (in $\T$), contradicting Theorem \ref{nontriv-cocycle}.
\end{proof}

\begin{remark}\label{rom} While the above proposition shows that $X_{5,1}$ cannot be embedded into a finite filtered abelian group, \cite[Theorem 1.7]{CGSS} does show that there is a \emph{fibration} $\pi \colon Y \to X_{5,1}$ (as defined in \cite[Definition 7.1]{gmv}, \cite[Definition 3.3.7]{candela0}) such that $Y$ has the structure of a finite filtered abelian group and $\pi$ has good lifting properties; this result was in particular used in \cite{CGSS} to give an alternate proof of Conjecture \ref{inv-conj} in both high and low characteristic.  In fact, we can explicitly give such an extension.  Let $G$ denote the abelian group $\Z/4\Z$ with the degree $3$ filtration
$$ G_0=G_1=G_2=G; \quad G_3 = 2\Z/4\Z; \quad G_i = \{0\} \forall i>3,$$
and consider the filtered abelian group
$$ Y \coloneqq G \times {\mathcal D}^2(\F_2) \times {\mathcal D}^5(\F_2).$$
One can show using Lemma \ref{expl} that the map $\pi \colon Y \to X_{5,1}$ defined by
$$ \phi( a, b, c ) = \left(a \mod 2, b, \frac{\binom{a}{2} b + c}{2}\right),$$
is a fibration; we leave the details to the interested reader.
\end{remark} 

\section{Counterexample to the strong inverse conjecture}

We now use the larger nilspace $X_{5,5}$ introduced in the previous section to establish Theorem \ref{counter}.  (The reason for using $X_{5,5}$ instead of $X_{5,1}$ will only be apparent near the end of the argument.)

\subsection{Constructing the counterexample}

To locate the counterexample to Conjecture \ref{inv-conj-strong} (for a suitable choice of parameters), we use a probabilistic construction.  Let $n$ be a large parameter (which will eventually be sent to infinity).   We let $(Q,S) \colon \F_2^n \to X_{5,5}$ be an $n$-cube in $X_{5,5}$, chosen uniformly at random from $C^n(X_{5,5})$.  In view of Lemma \ref{lift}, one way to generate such an element is as follows.  First, one generates an $n$-cube $Q \colon \F_2^n \to X_2$ of $X_2$, uniformly at random; in other words, $Q$ is a pair $(Q_1,Q_2)$ of independent classical quadratic polynomials $Q_1,Q_2 \colon \F_2^n \to \F_2$.  By Lemma \ref{lift}, the set of all $S \colon \F_2^n \to \frac{1}{2^5}\Z/\Z$ for which $(Q,S)$ is an $n$-cube in $X_{5,5}$ is a coset (depending on $Q$) of the finite group $\Poly^5(\F_2^n \to \frac{1}{2^5}\Z/\Z)$, and so once $Q$ is chosen, one simply selects an element of this coset uniformly at random, or equivalently one chooses uniformly at random a solution $S \colon \F_2^n \to \frac{1}{2^5}\Z/\Z$ to the equation \eqref{S-eq}.  This gives a uniformly distributed element on the entirety of $C^n(X_{5,5})$, a product of two uniform distributions, because all cosets of $\Poly^5(\F_2^n \to \frac{1}{2^5}\Z/\Z)$ have the same cardinality.

\begin{remark}\label{alt-form} Thanks to Lemma \ref{expl}, we can also generate $(Q,S)$ as 
$$(Q,S) = \left((2R, Q^{(2)}), \frac{\binom{R}{2} Q^{(2)}}{2} + P\right),$$
where $R, Q^{(2)}, P$ are elements of $\Poly^3(\F_2^n \to \Z/4\Z)$, $\Poly^2(\F_2^n \to \F_2)$, and $\Poly^5(\F_2^n \to \frac{1}{2^5} \Z/\Z)$ respectively, chosen uniformly and independently at random; compare with Remark \ref{rom}.  However, we will not make significant use of this representation here.
\end{remark}

The random function $f = e(S)$ will be used as our counterexample (or more precisely, as a sequence of counterexamples as $n \to \infty$) to Conjecture \ref{inv-conj-strong}.  We first record a deterministic lower bound on the $U^6$ norm of $e(S)$:

\begin{lemma}[Deterministic lack of Gowers uniformity]\label{gowers}  Whenever $(Q,S) \colon \F_2^n \to X_{5,5}$ is an $n$-cube in $X_{5,5}$, we have
$$ \|e(S)\|_{U^6(\F_2^n)} \geq \eta$$
for some absolute constant $\eta>0$ (independent of $n$).
\end{lemma}

Informally, this lemma asserts that $S$ behaves (in some weak statistical sense) like a ``pseudo-quintic'', and indeed Conjecture \ref{inv-conj} could now be invoked to conclude that $e(S)$ correlated with an actual (non-classical) quintic polynomial.  For instance, from Remark \ref{alt-form} we see that with high probability $e(S)$ would correlate with the function $e(P)$, where $P$ is as in that remark, as the phase $\frac{\binom{R}{2} Q^{(2)}}{2}$ will vanish approximately three quarters of the time.  However, we will show that (with high probability) such quintic polynomials cannot be (approximately) constructed out of a bounded number of translates of $S$, leading to the proof of Theorem \ref{counter}.

\begin{proof}  From \eqref{S-eq} we have
$$ \E_{x,h_1,\dots,h_6 \in \F_2^n} e((d^6 S)_{h_1,\dots,h_6}(x) ) e( - \rho( (Q(x+\omega \cdot \vec h))_{\omega \in \{0,1\}^6}) ) = 1$$
where $\vec h \coloneqq (h_1,\dots,h_6)$.
Performing a Fourier expansion of $e(-\rho)$ (which one extends arbitrarily to a function on the finite abelian group $X_2^{\{0,1\}^6}$) and using the pigeonhole principle, we conclude that
$$  \E_{x,h_1,\dots,h_6 \in \F_2^n} e((d^6 S)_{h_1,\dots,h_6}(x) ) (-1)^{\sum_{\omega \in \{0,1\}^6} c_\omega \cdot Q(x+\omega \cdot \vec h)} \geq \eta$$
for some absolute constant $\eta > 0$ and some coordinates of Fourier frequencies $c_\omega \in X_2$ (which may depend on $n$ and $S$), using the usual $\F_2$-valued inner product 
$$ (c_1,c_2) \cdot (q_1,q_2) \coloneqq c_1 q_1 + c_2 q_2$$
on the vector space $X_2$.  Applying the Cauchy--Schwarz--Gowers inequality (see e.g., \cite[(5.5)]{gt-ap}) we conclude that
$$ \| e(S) (-1)^{c_{0^6} \cdot Q} \|_{U^6(\F_2^n)} \geq \eta.$$
As $Q$ is of degree $2$ (less than $5$), multiplication by the quadratic phase $(-1)^{c_{0^6} \cdot Q}$ does not affect the $U^6(\F_2^n)$ norm, and the claim follows.
\end{proof}

Now let $\eps \colon \R^+ \to \R^+$ be an increasing function to be chosen later with $\eps(1/m) \to 0$ sufficiently quickly as $m \to \infty$.  Suppose for contradiction that Conjecture \ref{inv-conj-strong} held for $p=2$ and $k=5$.  Then by the above lemma, applying that conjecture to each of the random functions $e(S)$ and then using the law of total probability, there exists $M$ (depending on $\eps()$, but deterministic and independent of $n$) such that, for any $n$, and with the random $n$-cube $(Q,S) \in C^n(X_{5,5})$ chosen as above, and $\vec h = (h_1,\dots,h_M) \in (\F_2^n)^M$ chosen uniformly at random, with probability at least $1/2$, there exist
$1 \leq m \leq M$, $P \in \Poly^5(\F_2^n)$ and a function $F \colon (\frac{1}{2^5}\Z/\Z)^{\F_2^M} \to \C$ (which may depend on $Q, S, h_1,\dots, h_M$), such that
\begin{equation}\label{prop-1}
|\E_{x \in \F_2^n} e(S(x)-P(x))| \geq \frac{1}{m}
\end{equation}
and
$$
|\E_{x \in \F_2^n} e(P(x)) - F( (S(x+a \cdot \vec h))_{a \in \F_2^M})| \leq \eps(m).
$$
(We drop the Lipschitz condition on $F$ as being of little use due to the finite nature of the domain.)  By projecting $F$ to the unit circle we may assume that $F = e(\Phi)$ for some $\Phi \colon (\frac{1}{2^5}\Z/\Z)^{\F_2^M} \to \T$, thus
\begin{equation}\label{prop-2}
|\E_{x \in \F_2^n} e(P(x)) - e(\Phi( (S(x+a \cdot \vec h))_{a \in \F_2^M}))| \leq \eps(m).
\end{equation}

We have two independent sources of randomness present in the above assertions: one coming from the uniformly chosen $n$-cube $(Q,S)$, and one coming from the uniformly chosen sampling vectors $h_1,\dots,h_M$.  It will be convenient to normalize the $h_1,\dots,h_M$ by the following argument.  By Fubini's theorem, we can choose the sampling vectors $h_1,\dots,h_M \in \F_2^n$ \emph{first}, and then choose the $n$-cube $(Q,S) \in C^n(X_{5,5})$ \emph{second}, and it will still be the case with probability at least $1/2$ that we can find $m, P, \Phi$ obeying \eqref{prop-1}, \eqref{prop-2}.  For $n$ sufficiently large (depending on $M$), the probability that the $h_1,\dots,h_M$ are linearly dependent is less than $1/4$ (say).  Deleting this event and applying the pigeonhole principle for the $h_1,\dots,h_M$, we conclude that for all sufficiently large $n$, we may find linearly independent (and now deterministic) $h_1,\dots,h_M \in \F_2^n$ such that, for a uniformly chosen $n$-cube $(Q,S)$ in $X_{5,5}$, with probability at least $1/4$, there exists a quintic polynomial $P \in \Poly^5(\F_2^n)$, a natural number $1 \leq m \leq M$, and a function $\Phi \colon (\frac{1}{2^5}\Z/\Z)^{\F_2^M} \to \T$, obeying the properties \eqref{prop-1}, \eqref{prop-2}.

The above claim is invariant with respect to general linear transformations on $\F_2^n$ (i.e., changes of coordinate basis), so without loss of generality we may take $h_i = e_i$ for $1 \leq i \leq M$, where $e_1,\dots,e_n$ is the standard basis for $\F_2^n$.  Then we can simplify the tuple $(S(x+a \cdot \vec h))_{a \in \F_2^M}$ as $(S(x+(a,0^{n-M})))_{a \in \F_2^M}$. We summarize the situation so far as follows.

\begin{proposition}[$e(S)$ can be approximated by a measurable quintic]\label{quarter} Suppose that Conjecture \ref{inv-conj-strong} holds for $p=2$ and $k=5$, and let $\eps \colon \N \to \R^+$ be a function decreasing to zero.  Then there exists $M \geq 1$ such that for all sufficiently large $n$, and $(Q,S)$ a uniformly chosen $n$-cube in $X_{5,5}$, one has with probability at least $1/4$ that there exist a quintic polynomial $P \in \Poly^5(\F_2^n)$, $1 \leq m \leq M$ and a function $\Phi \colon (\frac{1}{2^5}\Z/\Z)^{\F_2^M} \to \T$ (which are all permitted to depend on $(Q,S)$) such that
\begin{equation}\label{prop-3}
|\E_{x \in \F_2^n} e(P(x)) - e(\Phi( (S(x+(a,0^{n-M})))_{a \in \F_2^M}))| \leq \eps(m)
\end{equation}
and
\begin{equation}\label{prop-4}
|\E_{x \in \F_2^n} e(S(x)-P(x))| \geq \frac{1}{m}
\end{equation}
where we split $\F_2^n$ as $\F_2^M \times \F_2^{n-M}$ (so that an element $a$ of $\F_2^M$ induces a corresponding element $(a,0^{n-M})$ of $\F_2^n$).
\end{proposition}

\subsection{Equidistribution theory for $Q,S$}

In order to extract a contradiction from the estimates \eqref{prop-3}, \eqref{prop-4} and the polynomial nature of $P$, we will need to understand the asymptotic equidistribution properties of the $n$-cube $(Q,S)$ in the following randomly sampled sense.  Given a choice of $n$-cube $(Q,S)$, and a natural number $d$, let $v_1,\dots,v_d \in \F_2^n$ be vectors drawn uniformly and independently from $\F_2^n$ (and also independently of $(Q,S)$), and consider the random functions $(\tilde Q,\tilde S) = (\tilde Q,\tilde S)_{(Q,S),v_1,\dots,v_d} \colon \F_2^{M+d} \to X_{5,5}$ defined by sampling $(Q,S)$ in the directions $e_1,\dots,e_M, v_1,\dots,v_d$, or more precisely by the formula
$$
(\tilde Q,\tilde S)( a_1,\dots,a_M,b_1,\dots,b_d) \coloneqq (Q,S)( a_1 e_1 + \dots + a_M e_M + b_1 v_1 + \dots + b_d v_d )$$
for all $a_1,\dots,a_M,b_1,\dots,b_d \in \F_2$.  This is the composition of the nilspace morphism $(Q,S) \colon {\mathcal D}^1(\F_2^n) \to X_{5,5}$ with a (random) linear transformation from $\F_2^{M+d}$ to $\F_2^n$, and so $(\tilde Q,\tilde S)$ is a (random) nilspace morphism from ${\mathcal D}^1(\F_2^{M+d})$ to $X_{5,5}$, or equivalently a (random) $M+d$-cube in $X_{5,5}$. 
Also, regardless of the choice of sampling vectors $v_1,\dots,v_d$, $(\tilde Q,\tilde S)$ must agree with $(Q,S)$ on $\F_2^M$ in the sense that
\begin{equation}\label{agree}
(\tilde Q,\tilde S)(a, 0^d) = (Q_0,S_0)(a)
\end{equation}
for all $a \in \F_2^M$, where $(Q_0,S_0) \colon \F_2^M \to X_{5,5}$ is the restriction of $(Q,S)$ to $\F_2^M$, defined by the formula
\begin{equation}\label{q0s0}
(Q_0,S_0)(a) \coloneqq (Q,S)(a, 0^{n-M}).
\end{equation}
Note that $(Q_0,S_0)$ is an $M$-cube in $X_{5,5}$, since $(Q,S)$ is an $n$-cube in $X_{5,5}$.

Let 
$$\Sigma^{(d)}_{Q_0,S_0} \subset C^{M+d}(X_{5,5})$$
denote the space of all $M+d$-cubes $(\tilde Q, \tilde S)$ that agree with the $M$-cube $(Q_0,S_0)$ on the face $\F_2^M \times \{0^d\}$ in the sense of \eqref{agree}; this is a non-empty finite set whose cardinality is bounded uniformly in $n$.  For each choice of $(Q,S),d$, let $\mu_{Q,S}^{(d)}$ denote the distribution of the random variable $(\tilde Q,\tilde S)$ generated by the random variables $v_1,\dots,v_d$, thus $\mu_{Q,S}^{(d)}$ is the probability measure on $\Sigma^{(d)}_{Q_0,S_0}$ defined by the formula
$$ \int_{\Sigma^{(d)}_{Q_0,S_0}} G( \tilde Q, \tilde S )\ d\mu_{Q,S}^{(d)}(\tilde Q, \tilde S)
= \E_{v_1,\dots,v_d \in \F_2^n} G( \tilde Q_{Q,v_1,\dots,v_d}, \tilde S_{S,v_1,\dots,v_d} )$$
for any observable $G \colon \Sigma^{(d)}_{Q_0,S_0} \to \C$. Meanwhile, let $\overline{\mu}_{Q_0,S_0}^{(d)}$ denote the uniform probability measure on $\Sigma^{(d)}_{Q_0,S_0}$.  

The following key equidistribution theorem asserts that, for $(Q,S)$ a uniformly chosen $n$-cube, $\mu_{Q,S}^{(d)}$ converges to $\overline{\mu}_{Q_0,S_0}^{(d)}$ ``in probability''.  More precisely:

\begin{theorem}[Equidistribution theorem]\label{equid}  Let $d$ be fixed.  Then, we have
\begin{equation}\label{tv}
 d_{\TV}( \mu_{Q,S}^{(d)}, \overline{\mu}_{Q_0,S_0}^{(d)} ) = o(1)
 \end{equation}
with probability $1-o(1)$, where $o(1)$ denotes any quantity that goes to zero as $n \to \infty$ holding all other parameters not depending on $n$ (such as $d$) fixed.  Here $d_{\TV}$ denotes the total variation distance between probability measures.
\end{theorem}

Informally, this theorem asserts that the condition \eqref{agree} on the $M+d$-cube $(\tilde Q,\tilde S)$ is asymptotically the \emph{only} constraint that could control (or even bias) the distribution of this $M+d$-cube.  One could replace the total variation distance here by any other reasonable metric, since $\mu_{Q,S}^{(d)}$, $\overline{\mu}_{Q_0,S_0}^{(d)}$ are supported on finite sets of cardinality bounded uniformly on $n$.

\begin{proof}  We first establish the equidistribution claim for $Q$ only.  Let $\Sigma^{(d)}_{Q_0} \subset C^{M+d}(X_2)$ be the collection of all $M+d$-cubes $\tilde Q$ in $X_2$ which agree with $Q_0$ on the face $\F_2^M\times \{0^d\}$ in the sense of \eqref{agree}.  We then define $\mu_Q^{(d)}$ as before, and set $\overline{\mu}_{Q_0}^{(d)}$ to be uniform measure on $\Sigma^{(d)}_{Q_0}$.  Observe that the projection map $(\tilde Q,\tilde S) \mapsto \tilde Q$ maps $\Sigma^{(d)}_{Q_0,S_0}$ to $\Sigma^{(d)}_{Q_0}$; by Lemma \ref{lift}, the map is surjective, and the fibers of this map are essentially cosets of the finite group
\begin{equation}\label{k-def}
K \coloneqq \left\{ P \in \Poly^5\left(\F_2^{M+d} \to \frac{1}{2^5}\Z/\Z\right): P(x, 0^d)=0 \forall x \in \F_2^M \right\}.
\end{equation}
In particular, all fibers have the same cardinality, and hence the uniform measure $\overline{\mu}_{Q_0,S_0}^{(d)}$ pushes forward to the uniform measure $\overline{\mu}_{Q_0}^{(d)}$.  Also, by definition the sampling measure $\mu_{Q,S}^{(d)}$ pushes forward to the sampling measure $\mu_Q^{(d)}$.  Hence, in order to establish \eqref{tv} with probability $1-o(1)$, a natural first step would be to first show the weaker claim that
\begin{equation}\label{weak-equi}
d_{\TV}( \mu_{Q}^{(d)}, \overline{\mu}_{Q_0}^{(d)} ) = o(1)
\end{equation}
with probability $1-o(1)$.

We use the second moment method. The set $\Sigma^{(d)}_{Q_0}$ is a (random) coset of the (deterministic) finite group
$$ H \coloneqq \{ P \in \Poly^2(  \F_2^{M+d} \to X_2 ): P(x, 0^d)=0 \forall x \in \F_2^M \}.$$
By the Diaconis--Shahshahani Upper Bound Lemma\footnote{For a probability measure $\mu$ on a finite abelian group $H$ (or any coset thereof), the Diaconis--Shahshahani Upper Bound Lemma is the inequality 
\[
d_{\TV}(\mu, u)^2 \leq \tfrac{1}{4} \sum_{\xi \neq 0}
|\widehat{\mu}(\xi)|^2,
\]
where $u$ denotes the uniform measure and $\widehat{\mu}(\xi)$ are the Fourier coefficients with respect to the characters of $H$.} (cf. \cite[Lemma 1]{diaconis}), it thus suffices to establish the claim
$$ \int_{\Sigma^{(d)}_{Q_0}} e( \xi \cdot (\tilde Q - \tilde Q_*) )\ d\mu_Q^{(d)} = o(1)$$
with probability $1-o(1)$ for any fixed non-trivial character $\xi \colon H \to \T$, where $\tilde Q_* = \tilde Q_*(Q)$ is an arbitrary element of $\Sigma^{(d)}_Q$ (the exact choice is unimportant as it does not affect the magnitude of the left-hand side).  By Chebyshev's inequality, it suffices to show that
$$ \E_{Q,S} \left|\int_{\Sigma^{(d)}_{Q_0}} e( \xi \cdot (\tilde Q - \tilde Q_*) )\ d\mu_Q^{(d)}\right|^2 = o(1).$$
The left-hand side can be rewritten as
$$ \E_{v_1,\dots,v_d,v'_1,\dots,v'_d \in \F_2^n} \E_Q e( \xi \cdot (\tilde Q_{Q,v_1,\dots,v_d} - \tilde Q_{Q,v'_1,\dots,v'_d} ) ).$$
Since $d$ is fixed and $n$ is going to infinity, we see that the vectors $v_1,\dots,v_d$, $v'_1,\dots,v'_d$, $e_1,\dots,e_M$ will be linearly independent with probability $1-o(1)$.  Hence we may restrict to this portion of the average with acceptable error.  Applying a linear change of variables (which does not affect the distribution of the random variable $Q$), we may then normalize $v_i = e_{M+i}$ and $v'_i = e_{M+d+i}$ for $i=1,\dots,d$.  It will thus suffice to show that
$$ \E_Q e( \xi \cdot (\tilde Q_{Q,e_{M+1},\dots,e_{M+d}} - \tilde Q_{Q,e_{M+d+1},\dots,e_{M+2d}} ) ) = o(1).$$
The random variable $Q$ is uniformly distributed over a finite abelian group $\Poly^2(\F_2^n \to X_2 )$, and the expression inside the $e()$ is a homomorphism in $Q$.  Hence by Fourier analysis, the claim follows unless we have the vanishing
\begin{equation}\label{xip-van}
 \xi \cdot (\tilde Q_{Q,e_{M+1},\dots,e_{M+d}} - \tilde Q_{Q,e_{M+d+1},\dots,e_{M+2d}} ) = 0
 \end{equation}
for all quadratic polynomials $Q \colon \F_2^n \to X_2$.  But if we let $P \colon \F_2^{M+d} \to X_2$ be an element of the group $H$ that is not annihilated by $\xi$, then the function $Q \colon \F_2^n \to X_2$ defined by
$$ Q(x_1,\dots,x_n) \coloneqq P(x_1,\dots,x_{M+d})$$
is a quadratic polynomial\footnote{The composition of a homomorphism, here the projection from $\F_2^n$ onto $\F_2^{M+d}$, with a polynomial of degree $d$ is again a polynomial of degree $d$.} for which the left-hand side of \eqref{xip-van} is non-zero.  Thus we have the desired equidistribution \eqref{weak-equi}.

To show full equidistribution, it suffices by the triangle inequality to show, for each element $\tilde Q_*$ of $\Sigma^{(d)}_Q$, that
$$
d_{\TV}( \mu_{Q,S}^{(d)} 1_{\tilde Q = \tilde Q_*}, \overline{\mu}_{Q_0,S_0}^{(d)} 1_{\tilde Q = \tilde Q_*}) = o(1),$$
with probability $1-o(1)$, where $1_{\tilde Q = \tilde Q_*}$ denotes the indicator function to the set $\{ (\tilde Q, \tilde S) \in \Sigma^{(d)}_{Q_0,S_0}: \tilde Q = \tilde Q_*\}$.  Note from \eqref{weak-equi} that with probability $1-o(1)$, both of these measures differ in mass by $o(1)$.  Once one fixes $\tilde Q = \tilde Q_*$, the variable $\tilde S$ ranges in a coset $\tilde S_{(S_0,\tilde Q_*)} + K$ of the finite abelian group $K$ defined in \eqref{k-def}, where we arbitrarily choose one representative $\tilde S_{(S_0,\tilde Q_*)}$ of this coset for each choice of $S_0, \tilde Q_*$.  By Fourier analysis, it thus suffices to show that
$$ \int_{\Sigma^{(d)}_{Q_0,S_0}} e(\xi \cdot (\tilde S - \tilde S_{(S_0,\tilde Q_*)})) 1_{\tilde Q = \tilde Q_*}\ d\mu_{Q,S}^{(d)}(\tilde Q, \tilde S) = o(1)$$
with probability $1-o(1)$ for each non-trivial character $\xi \colon K \to \T$ of $K$.  As before, it suffices by the Chebyshev inequality to show that
$$ \E_{Q,S} |\int_{\Sigma^{(d)}_{Q,S}} e(\xi \cdot (\tilde S - \tilde S_{S_0,\tilde Q_*})) 1_{\tilde Q = \tilde Q_*}\ d\mu_{Q,S}^{(d)}(\tilde Q, \tilde S)|^2 = o(1).$$
The left-hand side can be rewritten as
$$ \E_{v_1,\dots,v_d,v'_1,\dots,v'_d \in \F_2^n} \E_{Q,S} e( \xi \cdot (\tilde S_{S,v_1,\dots,v_d} - \tilde S_{S,v'_1,\dots,v'_d} ) ) 1_{\tilde Q_{Q,v_1,\dots,v_d} = \tilde Q_{Q,v'_1,\dots,v'_d} = \tilde Q_0}.$$
As before we can restrict to the case where $v_1, \dots, v_d$, $v'_1, \dots, v'_d$, $e_1,\dots,e_M$ are linearly independent, and then after a change of basis it suffices to show that
$$ \E_{Q,S} e( \xi \cdot (\tilde S_{S,e_{M+1},\dots,e_{M+d}} - \tilde S_{S,e_{M+d+1},\dots,e_{M+2d}} ) ) 1_{\tilde Q_{Q,e_{M+1},\dots,{M+d}} = \tilde Q_{Q,e_{M+d+1},\dots,e_{M+2d}} = \tilde Q_0} = o(1).$$
Clearly it would suffice to show that
$$ \E_{S} e( \xi \cdot (\tilde S_{S,e_{M+1},\dots,e_{M+d}} - \tilde S_{S,e_{M+d+1},\dots,e_{M+2d}} ) ) = o(1)$$
uniformly over all $Q$ with
$$ \tilde Q_{Q,e_{M+1},\dots,{M+d}} = \tilde Q_{Q,e_{M+d+1},\dots,e_{M+2d}} = \tilde Q_0.$$
For fixed $Q$, $S$ ranges over a coset of $\Poly^5(\F_2^n \to \frac{1}{2^5}\Z/\Z)$ by Lemma \ref{lift},
and the expression inside $e()$ is an (affine) homomorphism of $S$ on this coset.  Thus by Fourier analysis we are done unless the expression
$$ \xi \cdot (\tilde S_{S,e_{M+1},\dots,e_{M+d}} - \tilde S_{S,e_{M+d+1},\dots,e_{M+2d}} ) $$
is constant on this coset, or equivalently that
\begin{equation}\label{xip}
 \xi \cdot (\tilde S_{P,e_{M+1},\dots,e_{M+d}} - \tilde S_{P,e_{M+d+1},\dots,e_{M+2d}} ) = 0
 \end{equation}
for all $P$ in the group $\Poly^5(\F_2^n \to \frac{1}{2^5}\Z/\Z)$.  But if we let $P' \in K$ be an element of $K$ not annihilated by $\xi$, and set
$$ P(x_1,\dots,x_n) \coloneqq P'(x_1,\dots,x_{M+d})$$
then we see that $P$ lies in $\Poly^5(\F_2^n \to \frac{1}{2^5}\Z/\Z)$ and does not obey \eqref{xip}.  This completes the proof of the theorem.
\end{proof}

We conclude

\begin{corollary}[Equidistributed sequence]\label{Equid-cor} 
Suppose that Conjecture \ref{inv-conj-strong} holds for $p=2$ and $k=5$, and let $\eps \colon \N \to \R^+$ be a function decreasing to zero.  Let $M$ be as in Proposition \ref{quarter}.  Then there exist an integer $1 \leq m \leq M$, a function $\Phi \colon (\frac{1}{2^5}\Z/\Z)^{\F_2^M} \to \T$, an $M$-cube $(Q_0,S_0)$ (which also defines $\Sigma^{(d)}_{Q_0,S_0}$ and $\overline{\mu}_{Q_0,S_0}^{(d)}$), and an infinite sequence of dimensions $n \to \infty$ such that the following holds.  For each $n$ in this sequence, there are a deterministic $n$-cube $(Q,S)$ in $X_{5,5}$, with associated $M$-cube $(Q_0,S_0)$, such that one has
\begin{equation}\label{prop-3-alt}
\Big| \E_{x \in \F_2^n} e(P(x)) - e\Big(\Phi\big( (S(x+(a,0^{n-M})))_{a \in \F_2^M}\big)\Big) \Big| \leq 2\eps(m),
\end{equation}
and
\begin{equation}\label{prop-4-alt}
\Big| \E_{x \in \F_2^n} e(S(x)-P(x)) \Big| \geq \frac{1}{m},
\end{equation}
with $\mu_{Q,S}^{(d)}$ converging in total variation norm to $\overline{\mu}_{Q_0,S_0}^{(d)}$ for each $d \geq 0$.
\end{corollary}

\begin{proof} Applying Proposition \ref{quarter}, Theorem \ref{equid}, and a standard diagonal argument, we obtain along a sequence $n$ going to infinity, an $n$-cube $(Q,S)$ in $X_{5,5}$, an integer $1 \leq m \leq M$, a polynomial $P \in \Poly^5(\F_2^n)$, and a function $\Phi \colon (\frac{1}{2^5}\Z/\Z)^{\F_2^M} \to \T$ obeying \eqref{prop-3}, \eqref{prop-4} such that
$$ d_{\TV}( \mu_{Q,S}^{(d)}, \overline{\mu}_{Q_0,S_0}^{(d)} ) \to 0$$
as $n$ goes to infinity along this sequence, for each $d \geq 0$.   The quantity $m$ currently depends on $n$, but it takes only finitely many values, so by the pigeonhole principle we may pass to a subsequence and assume that $m$ is independent of $n$.  Similarly, the number of possible restrictions $(Q_0,S_0)$ of $(Q,S)$ to $\F_2^M$ is bounded independently of $n$, because $(Q_0,S_0)$ is an $M$-cube in the finite nilspace $X_{5,5}$.  Hence by the pigeonhole principle, we may pass to a further subsequence of $n$ and assume that this restriction $(Q_0,S_0)$ is independent of $n$.  Finally, with $\Phi$, we may round $\Phi$ to the nearest multiple of $\eps(m)/100$ in $[0,1]$, at the cost of worsening \eqref{prop-3} to \eqref{prop-3-alt}.  Now the number of possible $\Phi$ is bounded independently of $n$, so by another application of the pigeonhole principle we can make $\Phi$ independent of $n$, giving the claim.
\end{proof}

The next step is to construct a certain finite nilspace $X_{(Q_0,S_0)}$ associated to the $M$-cube $(Q_0,S_0)$, that can be viewed as an abstraction of the random samples $((Q,S)(x + (a,0^{n-M})))_{a \in \F_2^M}$ of $(Q,S)$ in the limit $n \to \infty$ (somewhat in the spirit of the Furstenberg correspondence principle).  The construction is as follows.  As $X_{5,5}$ is $2$-homogeneous, we see from \eqref{nbox} that we have the equivalence
$$ C^M(X_{5,5}) \equiv \Hom_\Box(\F_2^M, X_{5,5}).$$
By either Remark \ref{cube-nilspace} or \ref{morphism-nilspace}, this space has the structure of a finite $5$-step $2$-homogeneous nilspace (it is easy to see that the two nilspace structures given by these remarks agree).  This space will not be ergodic in general, so the equivalence relation $\sim_0$ on this space introduced in Remark \ref{ergodic-decomposition} can be non-trivial.  The morphism $(Q_0,S_0)$ is a point in $\Hom_\Box(\F_2^M, X_{5,5})$, and we define $X_{(Q_0,S_0)}$ to be the equivalence class of this point:
$$ X_{(Q_0,S_0)} \coloneqq \{ (Q',S') \in C^M(X_{5,5}): (Q',S') \sim_0 (Q_0,S_0) \}.$$
This is then an ergodic finite $5$-step $2$-homogeneous nilspace.

For every $s \geq 0$, we define a map $\pi_s \colon \Sigma^{(1+s)}_{Q_0,S_0} \to C^s(X_{(Q_0,S_0)})$ by the formula
$$ \pi_s( \tilde Q, \tilde S ) \coloneqq ( (\tilde Q, \tilde S)(\cdot,1,\omega) )_{\omega \in \{0,1\}^s}$$
for all $(\tilde Q,\tilde S) \in \Sigma^{(1+s)}_{Q_0,S_0}$; thus $\pi_s(\tilde Q,\tilde S)$ is the tuple formed by restricting $(\tilde Q,\tilde S)$ to the affine subspaces $\F_2^M \times (1,\omega)$ of $\F_2^{M+1+s}$ for $\omega \in \{0,1\}^s$.  Let us first check that $\pi_s( \tilde Q, \tilde S )$ lies in $C^s(X_{(Q_0,S_0)})$ as claimed.  Since $(\tilde Q,\tilde S)$ is a $M+1+s$-cube in $X_{5,5}$, the map
$$ (a,\omega) \mapsto (\tilde Q, \tilde S)(a,1,\omega)$$
is a $M+s$-cube in $X_{5,5}$, and hence the map
$$ \omega \mapsto (a \mapsto (\tilde Q, \tilde S)(a,1,\omega))$$
is an $s$-cube in $C^M(X_{5,5})$.  Applying \eqref{nbox}, the tuple $\pi_s(\tilde Q,\tilde S)$ is thus a $s$-cube in $C^M(X_{5,5})$, and thus lies in a single equivalence class of $\sim_0$.  A similar argument shows that the pair
$$ ( (a \mapsto (\tilde Q,\tilde S)(a,0,0^s)), (a \mapsto (\tilde Q,\tilde S)(a,1,0^s)))$$
is a $1$-cube in $C^M(X_{5,5})$, and so the two elements of this pair are also equivalent by $\sim_0$.  By \eqref{agree}, the first map is $(Q_0,S_0)$, and hence $\pi_s(\tilde Q, \tilde S)$ is an $s$-cube in $X_{(Q_0,S_0)}$ as claimed.

Next, we claim that the map $\pi_s$ is surjective.  Let $((Q'_\omega,S'_\omega))_{\omega \in \{0,1\}^s}$ be an $s$-cube in $X_{(Q_0,S_0)}$.  Our goal is to locate an $M+1+s$-cube $(\tilde Q, \tilde S)$ in $X_{5,5}$ such that
$$ (\tilde Q,\tilde S)(a,0,0^s) = (Q_0,S_0)(a)$$
and
$$ (\tilde Q,\tilde S)(a,1,\omega) = (Q'_\omega,S'_\omega)(a)$$
for all $a \in \F_2^M$ and $\omega \in \{0,1\}^s$.  So $(\tilde Q, \tilde S)$ is already partially specified on the set
\begin{equation}\label{foo}
\F_2^M \times (\{(0,0^s)\} \cup (\{1\} \times \F_2^s)).
\end{equation}
By the construction of $X_{(Q_0,S_0)}$, this partially specified function is known to be an $M+s$-cube on 
\begin{equation}\label{lop}
\F_2^M \times \{1\} \times \F_2^s
\end{equation}
and an $M+1$-cube on
\begin{equation}\label{lop-2}
\F_2^M \times \{0,1\} \times \{0^s\}.
\end{equation}
The claim then follows from a large number of applications of the completion axiom for nilspaces (or by \cite[Lemma 3.1.5]{candela0}, after performing a reflection to move $(0^M,1,0^s)$ to the origin).

Now we claim that all the fibers of $\pi_s$ have the same cardinality.  Observe that if $(\tilde Q, \tilde S), (\tilde Q', \tilde S') \in \Sigma^{(1+s)}_{Q_0,S_0}$ have the same image under $\pi_s$, then $\tilde Q - \tilde Q'$ is an element of $\Poly^2(\F_2^{M+1+s} \to X_2)$ that vanishes on the set \eqref{foo}; and if $\tilde Q = \tilde Q'$, then $\tilde S - \tilde S'$ is an element of $\Poly^5(\F_2^{M+1+s} \to \frac{1}{2^5}\Z/\Z)$ that vanishes on \eqref{foo}.  Conversely, if $(\tilde Q,\tilde S) \in \Sigma^{(1+s)}_{Q_0,S_0}$ and $\tilde S - \tilde S'$
 is an element of $\Poly^5(\F_2^{M+1+s} \to \frac{1}{2^5}\Z/\Z)$ that vanishes on \eqref{foo}, then $(\tilde Q,\tilde S')$ is an element of $\Sigma^{(1+s)}_{Q_0,S_0}$ with the same image as $(\tilde Q,\tilde S)$ under $\pi_s$.  To conclude the claim, it suffices to show that whenever $(\tilde Q,\tilde S) \in \Sigma^{(1+s)}_{Q_0,S_0}$ and $\tilde Q - \tilde Q'$ is an element of $\Poly^2(\F_2^{M+1+s} \to X_2)$ that vanishes on \eqref{foo}, then there exists $(\tilde Q',\tilde S') \in \Sigma^{(1+s)}_{Q_0,S_0}$ with the same image as $(\tilde Q,\tilde S)$ under $\pi_s$.  By Lemma \ref{lift}, we can at least find a function $\tilde S'' \colon \F_2^n \to \frac{1}{2}\Z/\Z$ with $(\tilde Q',\tilde S'')$ an $n$-cube in $X_{5,5}$.  If the $\tilde S''-\tilde S$ vanished on \eqref{foo}, we would be done; but the best that can be said at present is that this function is a polynomial of degree $k$ on \eqref{lop} and on \eqref{lop-2}, again thanks to Lemma \ref{lift}.  Applying the completion axiom (or \cite[Lemma 3.1.5]{candela0}) many times, we can then find $P \in \Poly^5(\F_2^n \to \frac{1}{2^5}\Z/\Z)$ which agrees with $\tilde S''-\tilde S$ on \eqref{foo}; setting $\tilde S' \coloneqq \tilde S'' - P$ gives the claim.

 From the above properties of $\pi_s$ we see that $\pi_s$ pushes forward the uniform probability measure $\overline{\mu}^{(1+s)}_{Q_0,S_0}$
 on $\Sigma^{(1+s)}_{Q_0,S_0}$ to the uniform probability measure on $C^s(X_{(Q_0,S_0)})$.  Combining this with Corollary \ref{Equid-cor}, we conclude

\begin{corollary}[Equidistributed sequence, again]\label{Equid-cor-2} 
Suppose that Conjecture \ref{inv-conj-strong} holds for $p=2$ and $k=5$, and let $\eps \colon \N \to \R^+$ be a function decreasing to zero.  Let $M$ be as in Proposition \ref{quarter}, and let $n, (Q,S), (Q_0,S_0), m, \Phi$ be as in Corollary \ref{Equid-cor}.  If for any $s \geq 0$ we select $x,h_1,\dots,h_s \in \F_2^n$ uniformly and independently at random, then the random tuple
\[
\bigl((a \mapsto (Q,S)((a,0^{n-M}) + x + \sum_{i=1}^s \omega_i h_i))\bigr)_{\omega \in \{0,1\}^s}
\]
converges in distribution to the uniform distribution on $C^s(X_{(Q_0,S_0)})$.
\end{corollary}

\begin{proof}  
The distribution of the random tuple above coincides with the distribution $\mu^{(s)}_{Q,S}$ of the random cube $(\tilde Q,\tilde S)_{(Q,S),x,h_1,\dots,h_s}$.  By definition of $\pi_s$, this tuple is exactly the image of $\mu^{(s)}_{Q,S}$ under $\pi_s$.  Similarly, the uniform distribution on $C^s(X_{(Q_0,S_0)})$ is the image of $\overline{\mu}^{(s)}_{Q_0,S_0}$ under $\pi_s$.  Since Corollary~\ref{Equid-cor} asserts that $\mu^{(s)}_{Q,S}$ converges in total variation to $\overline{\mu}^{(s)}_{Q_0,S_0}$, applying $\pi_s$ yields the desired convergence in distribution of the tuple to the uniform distribution on $C^s(X_{(Q_0,S_0)})$.
\end{proof}

In the language of \cite{candela-szegedy-inverse}, this corollary asserts that the sampling map
$$ x \mapsto (a \mapsto (Q,S)((a,0^{n-M}) + x))$$
becomes an asymptotically balanced map from ${\mathcal D}^1(\F_2^n)$ to $X_{(Q_0,S_0)}$ as $n$ goes to infinity along the sequence.

\subsection{Concluding the argument}

With the equidistribution theory for the $n$-cube $(Q,S)$ in hand, we can now return to the task of deriving a contradiction. Let the notation be as in Proposition \ref{quarter} and Corollary \ref{Equid-cor-2}.  

The first step is to use Corollary \ref{Equid-cor-2} to transfer the structural conclusions of Proposition \ref{quarter} to the nilspace $X_{(Q_0,S_0)}$, in a form resembling the proof of Proposition \ref{injcont}.  Let $n$ belong to the sequence from Corollary \ref{Equid-cor}, and let $x,h_1,\dots,h_6$ be chosen uniformly and independently at random from $\F_2^n$.  By Corollary \ref{Equid-cor-2}, the random $6$-cube
\begin{equation}\label{random-tuple} 
\bigl((a \mapsto (Q,S)((a,0^{n-M}) + x + \sum_{i=1}^6 \omega_i h_i))\bigr)_{\omega \in \{0,1\}^6}
\end{equation}
converges in distribution to the uniform distribution on $C^6(X_{(Q_0,S_0)})$, while the random element
\begin{equation}\label{qs-0}
a \mapsto (Q,S)((a,0^{n-M}) + x)
\end{equation}
converges in distribution to the uniform distribution on $X_{(Q_0,S_0)}$.

On the other hand, combining \eqref{prop-3-alt} and \eqref{prop-4} gives (for $\eps(m)$ sufficiently small)
\[
\E_{x \in \F_2^n} \bigl|e(S(x))- e\bigl(\Phi((S(x+a))_{a \in \F_2^M})\bigr) \bigr| \;\geq\; \frac{1}{2m}.
\]
Note that the inner quantity depends on $x$ only through the random element \eqref{qs-0} of $X_{(Q_0,S_0)}$: indeed, for $(Q',S')=\eqref{qs-0}$ one has $S'(0)=S(x)$ and $(S(x+a))_{a\in\F_2^M}=S'$, so the expectation above is
\[
\E_{(Q',S')\sim \eqref{qs-0}} \, \bigl|e(S'(0))-e(\Phi(S'))\bigr|.
\]
The function $(Q',S') \mapsto |e(S'(0))-e(\Phi(S'))|$ is bounded and continuous on the compact nilspace $X_{(Q_0,S_0)}$. Since \eqref{qs-0} converges in distribution to the uniform distribution on $X_{(Q_0,S_0)}$ (by Corollary~\ref{Equid-cor-2}), the Portmanteau theorem implies that this expectation converges to its uniform counterpart. Hence
\begin{equation}\label{prop-5}
\Bigl|\E_{(Q',S') \in X_{(Q_0,S_0)}} e(S'(0)-\Phi(S'))\Bigr| \;\geq\; \frac{1}{2m}.
\end{equation}

Similarly, as $P$ is a quintic polynomial, we have
$$ \E_{x \in \F_2^n; \vec h \in (\F_2^n)^6} \left|e\left( \sum_{\omega \in \{0,1\}^6} (-1)^{|\omega|} P( x + \omega \cdot \vec h ) \right) - 1\right| = 0.$$
 Hence by \eqref{prop-3} and many applications of the triangle inequality
$$ \E_{x \in \F_2^n; \vec h \in (\F_2^n)^6} \left|e\left( \sum_{\omega \in \{0,1\}^6} (-1)^{|\omega|} \Phi( (S(x + a + \omega \cdot \vec h ) )_{a \in \F_2^M} \right) - 1\right| = O(\eps(m))$$
where the implied constant in the $O()$ notation is absolute.  Since the random variable \eqref{random-tuple} converges to the uniform distribution on $C^{6}(X_{(Q_0,S_0)})$, we conclude that
$$ \E_{ (Q',S') \in C^6(X_{(Q_0,S_0)})} \left|e\left( \sum_{\omega \in \{0,1\}^6} (-1)^{|\omega|} \Phi( (S'(a,\omega)) )_{a \in \F_2^M}\right ) - 1\right| = O(\eps(m)).$$
Applying Theorem \ref{poly-stable} (and Markov's inequality), we conclude (for $\eps$ sufficiently rapidly decreasing) that there exists a quintic polynomial $\Phi' \in \Poly^5(X_{(Q_0,S_0)})$ such that
$$ \E_{(Q',S') \in X_{(Q_0,S_0)}} |e( \Phi(S') ) - e( \Phi'(Q',S') )| \leq \frac{1}{4m}$$
and hence by \eqref{prop-5} and the triangle inequality
\begin{equation}\label{eqs}
|\E_{(Q',S') \in X_{(Q_0,S_0)}} e(S'(0) - \Phi'(Q',S')) )| \geq \frac{1}{4m}.
\end{equation}
To take advantage of this correlation, we perform vertical differentiation in the $S'$ direction.  Arguing exactly as in the proof of \eqref{pup}, we see that the vertical derivative
$\Phi'(Q',S'+\frac{1}{2}) - \Phi'(Q',S')$ is constant, and thus $e(\Phi')$ is an eigenfunction of the vertical Koopman operator $V^u$ defined by
$$ V^u F(Q',S') \coloneqq F(Q',S'+\frac{1}{2}).$$
As before, $(Q',S') \mapsto e(S'(0))$ is also an eigenfunction of this operator, with eigenvalue $e(\frac{1}{2})$.  From \eqref{eqs}, these two eigenfunctions of this unitary operator are not orthogonal, and hence the eigenvalue of $e(\Phi')$ must also be $e(\frac{1}{2})$.  Thus, if we place an equivalence relation $\sim$ on $X_{(Q_0,S_0)}$ by declaring $(Q',S') \sim (Q'',S'')$ if $Q'=Q''$ and $S''$ is equal to either $S'$ or $S'+\frac{1}{2}$, then the function
$$ (Q',S') \mapsto S'(0) - \Phi'(Q',S')$$
is invariant with respect to this equivalence and thus can be viewed as a function on the quotient space $X_{(Q_0,S_0)}/\sim$.  In order to exploit this invariance to contradict Theorem \ref{nontriv-cocycle}, we will need to build a ``lifting map'' from $X_2$ to $X_{(Q_0,S_0)}/\sim$
that assigns to each $q \in X_2$ a certain element $(Q^*_q,S^*_q)$ of $X_{(Q_0,S_0)}$ (defined up to the equivalence $\sim$) that has good properties.  More precisely, we will show:

\begin{lemma}[Existence of lift]\label{exist}  One can assign an element $(Q^*_q,S^*_q)$ of $X_{(Q_0,S_0)}$ to each $q \in X_2$ with the following properties:
\begin{itemize}
    \item (Lift)  For each $q \in X_2$, one has $Q^*_q(0) = q$.
    \item (Morphism up to equivalence) For any $6$-cube $(q_\omega)_{\omega \in \{0,1\}^6} \in C^6(X_2)$ in $X_2$, there exists a $6$-cube $((Q'_\omega,S'_\omega))_{\omega \in \{0,1\}^6} \in C^6(X_{(Q_0,S_0)})$ in $X_{(Q_0,S_0)}$ such that $(Q'_\omega,S'_\omega) \sim (Q^*_{q_\omega}, S^*_{q_\omega})$ for all $\omega \in \{0,1\}^6$.
\end{itemize}
\end{lemma}

\begin{remark} Although we will not prove it here, one can show that the quotient space $X_{(Q_0,S_0)}/\sim$ is itself a nilspace which is an extension of the nilspace $X_2$.  The map that sends $q$ to (the equivalence class of) $(Q^*_q,S^*_q)$ can then be viewed as a ``splitting'' of that extension by a section that is itself a nilspace morphism.  It is in order to obtain this lifting that we were forced to use the larger nilspace $X_{5,5}$ instead of the smaller nilspace $X_{5,1}$, as we will need to take advantage of the freedom to modify $S$ by non-classical polynomials, and not merely by classical ones.
\end{remark}

Let us assume this lemma for the moment and obtain the desired contradiction.  Let $(q_\omega)_{\omega \in \{0,1\}^6} \in C^6(X_2)$ be a $6$-cube in $X_2$, and let 
$((Q'_\omega,S'_\omega))_{\omega \in \{0,1\}^6} \in C^6(X_{(Q_0,S_0)})$ be as in the above lemma.  Since $\Phi'$ is quintic on $X_{(Q_0,S_0)}$, we have
$$ \sum_{\omega \in \{0,1\}^6} (-1)^{|\omega|} \Phi'(Q'_\omega,S'_\omega) = 0.$$
Also, from the nilspace structure of $X_{(Q_0,S_0)}$ we have
$$ \sum_{\omega \in \{0,1\}^6} (-1)^{|\omega|} S'_\omega(0) = 
 \rho( (Q'_\omega(0))_{\omega \in \{0,1\}^6} ).$$
Subtracting, we conclude that
$$ \sum_{\omega \in \{0,1\}^6} (-1)^{|\omega|}( S'_\omega(0) - \Phi'(Q'_\omega,S'_\omega)) =  \rho( (Q'_\omega(0))_{\omega \in \{0,1\}^6} ).$$
Both sides are invariant with respect to $\sim$, so we may replace $(Q'_\omega,S'_\omega)$ with $(Q^*_{q_\omega},S^*_{q_\omega})$ in this identity, thus
$$ \sum_{\omega \in \{0,1\}^6} (-1)^{|\omega|}( S^*_{q_\omega}(0) - \Phi'(Q^*_{q_\omega},S^*_{q_\omega})) =  \rho( (Q^*_{q_\omega}(0))_{\omega \in \{0,1\}^6} ).$$
By the lifting property we have $Q^*_{q_\omega}(0) = q_\omega$.  We conclude that
$$ \rho = dF$$
where $F \colon \F_2^2 \to \T$ is the function
$$ F(q) \coloneqq S^*_q(0) - \Phi'(Q^*_q, S^*_q).$$
But this contradicts Theorem \ref{nontriv-cocycle}.

It remains to construct the lift $(Q^*_q, S^*_q)$ in Lemma \ref{exist}.  This will be accomplished by solving a certain system of constraints.  More precisely:

\begin{proposition}[Solving a system of constraints]\label{solving} Let $d \geq 0$, and let $(q_\omega)_{\omega \in \{0,1\}^d}$ be a $d$-cube in $X_2$.  Then there exists a $d$-cube $((Q'_\omega, S'_\omega))_{\omega \in \{0,1\}^d}$ in $X_{(Q_0,S_0)}$ obeying the following constraints:
\begin{itemize}
\item[(1)]  For every $\omega \in \{0,1\}^d$, one has $Q'_\omega(a) = q_\omega + Q_0(a) - Q_0(0)$ for all $a \in \F_2^M$.  In particular, $Q'_\omega(0) = q_\omega$.
\item[(2)] For every $1 \leq l \leq k-1$ and $1 \leq i_1 < \dots < i_l \leq M$, one has
$$ \partial_{e_{i_1}} \dots \partial_{e_{i_l}} S'_\omega(0) = \partial_{e_{i_2}} \dots \partial_{e_{i_l}} \psi_{\omega,i_l}(0) -
\partial_{e_{i_2}} \dots \partial_{e_{i_l}} \psi_{*,i_l}(0)
+ \partial_{e_{i_1}} \dots \partial_{e_{i_l}} S_0(0)$$
where
$$ \psi_{\omega,i_l}(a) \coloneqq \psi( Q'_\omega(a), Q'_\omega(a + e_{i_l} ) )$$
and
$$ \psi_{*,i_l}(a) \coloneqq \psi( Q_0(a), Q_0(a + e_{i_l} ) )$$
for all $a \in \F_2^M$.
\item[(3)]  One has $2S'_\omega(0) = 2S_0(0)$ for all $\omega \in \{0,1\}^d$.
\end{itemize}
Furthermore, this cube is unique up to equivalence in the following sense: if $((Q'_\omega, S'_\omega))_{\omega \in \{0,1\}^d}, ((Q''_\omega, S''_\omega))_{\omega \in \{0,1\}^d} \in C^d(X)$ both obey the properties (1)-(3), then we have $(Q'_\omega,S'_\omega) \sim (Q''_\omega, S''_\omega)$ for all $\omega \in \{0,1\}^d$.
\end{proposition}

Let us assume this proposition for the moment and see how it implies Lemma \ref{exist}.  Applying this proposition with $d=0$, we see that for each $q \in X_2$, we can find $(Q^*_q, S^*_q) \in X_{(Q_0,S_0)}$ obeying the $d=0$ conclusions (1)-(3) of the proposition; in particular, $Q^*_q(0)=q$.  Now let $(q_\omega)_{\omega \in \{0,1\}^6} \in C^6(X_2)$ be a $6$-cube in $X_2$, and let $((Q'_\omega,S'_\omega))_{\omega \in \{0,1\}^6} \in C^6(X_{(Q_0,S_0)})$ be as in the proposition.  For each $\omega \in \{0,1\}^6$, the point $(Q'_\omega,S'_\omega)$ in $X_{(Q_0,S_0)}$ obeys the $d=0$ axioms of (1)-(3) with respect to the $0$-cube $q_\omega$.  Since $(Q^*_{q_\omega}, S^*_{q_\omega})$ does also, we conclude from the uniqueness component of this proposition that $(Q'_\omega,S'_\omega) \sim (Q^*_{q_\omega}, S^*_{q_\omega})$ for all $\omega \in \{0,1\}^d$.  Lemma \ref{exist} follows.

It remains to establish Proposition \ref{solving}.  We first verify the uniqueness aspect.  Suppose we have two cubes $((Q'_\omega, S'_\omega))_{\omega \in \{0,1\}^d}, ((Q''_\omega, S''_\omega))_{\omega \in \{0,1\}^d} \in C^d(X_{(Q_0,S_0)})$ both obeying axioms (1)-(3).  From axiom (1) we see that $Q'_\omega = Q''_\omega$ for all $\omega \in \{0,1\}^d$.  From axiom (2), we see that
\begin{equation}\label{lard}
\partial_{e_{i_1}} \dots \partial_{e_{i_l}} S'_\omega(0)
= \partial_{e_{i_1}} \dots \partial_{e_{i_l}} S''_\omega(0)
\end{equation}
whenever $1 \leq l \leq k-1$ and $1 \leq i_1 < \dots \leq i_l \leq M$.  We claim that the same statement also holds for $l=k$.  Indeed, by construction of $X_{(Q_0,S_0)}$, we can find $(\tilde Q, \tilde S) \in \Sigma^{(1+d)}_{(Q,S)}$ such that
$$ 
(Q'_\omega,S'_\omega)(a) = (\tilde Q, \tilde S)(a, 1, \omega )$$
for all $\omega \in \{0,1\}^d$ and $a \in \F_2^M$.  Since $\tilde S$ agrees with $S$ on $\F_2^M$, we conclude that
$$
\partial_{e_{i_1}} \dots \partial_{e_{i_k}} S'_\omega(0) =
\partial_{e_{i_1}} \dots \partial_{e_{i_k}} S_0(0)
+ \partial_{(0,1,\omega)} \partial_{e_{i_1}} \dots \partial_{e_{i_k}} \tilde S(0).$$
As $(\tilde Q, \tilde S)$ is an $M+1+d$-cube in $X_{5,5}$, the right-hand side is equal to
$$
\partial_{e_{i_1}} \dots \partial_{e_{i_k}} S_0(0)
 + \rho( (\tilde Q(\sum_{j =1}^{k+1} \alpha_j w_j))_{\alpha \in \{0,1\}^{k+1}} )$$
 where $w_j \coloneqq e_{i_j}$ for $j=1,\dots,k$ and $w_{k+1} \coloneqq (0,1,\omega)$.  This expression depends only on $Q'_\omega$, $Q_0$, and $S_0$.  We have a similar formula for $S''_\omega$.  Since $Q'_\omega = Q''_\omega$, we conclude that \eqref{lard} holds for $l=k$.

 Now we claim that \eqref{lard} also holds for $l>k$.  It suffices to show that
$$ \partial_{e_{i_1}} \dots \partial_{e_{i_{k+1}}} S'_\omega(a)
= \partial_{e_{i_1}} \dots \partial_{e_{i_{k+1}}} S''_\omega(a)$$
whenever $a \in \F_2^M$ and $1 \leq i_1 < \dots < i_{k+1} \leq M$.  As $(Q'_\omega,S'_\omega)$ is an $M$-cube in $X_{5,5}$, one has
\begin{align*}
 \partial_{e_{i_1}} \dots \partial_{e_{i_{k+1}}} S'_\omega(a)
&= \Psi( (Q'_\omega((a + \sum_{j =1}^{k+1} \alpha_j e_{i_j}))_{\alpha \in \{0,1\}^{k+1}} ) ).
\end{align*}
Similarly for $S''_\omega$ and $Q''_\omega$.  Since $Q'_\omega = Q''_\omega$, we obtain \eqref{lard} for all $l > k$.

Now that \eqref{lard} has been established for all $l > 0$, we see from Taylor expansion that
$$ S''_\omega = S'_\omega - S'_\omega(0) + S''_\omega(0).$$
From axiom (3), $2( - S'_\omega(0) + S''_\omega(0) ) = -2S(0)+2S(0)=0$, hence for each $\omega \in \{0,1\}^d$, $S''_\omega$ is either equal to $S'_\omega$ or $S'_\omega + \frac{1}{2}$.  Since also $Q'_\omega = Q''_\omega$, we conclude that $(Q'_\omega,S'_\omega) \sim (Q''_\omega,S''_\omega)$.  This completes the proof of uniqueness.

Now we establish existence. Let $d \geq 0$, and let $(q_\omega)_{\omega \in \{0,1\}^d}$ be a $d$-cube in $X_2$.  By the construction of $X_{(Q_0,S_0)}$, our task is to find a $M+1+d$-cube $(\tilde Q, \tilde S)$ in $X_{5,5}$ obeying the following properties:
\begin{itemize}
\item[(0)]  For $a \in \F_2^M$, we have $(\tilde Q,\tilde S)(a,0^{n-M}) = (Q_0,S_0)(a)$.
\item[(1)]  For every $\omega \in \{0,1\}^d$ and $a \in \F_2^M$, one has $\tilde Q(a,1,\omega) = q_\omega + Q_0(a) - Q_0(0)$.
\item[(2)] For every $1 \leq l \leq k-1$ and $1 \leq i_1 < \dots < i_l \leq M$, one has
\begin{equation}\label{psi-S}
\partial_{e_{i_1}} \dots \partial_{e_{i_l}} \tilde S(0, 1, \omega) = \partial_{e_{i_2}} \dots \partial_{e_{i_l}} \psi_{i_l}(0,1,\omega) -
\partial_{e_{i_2}} \dots \partial_{e_{i_l}} \psi_{i_l}(0,0,0)
+ \partial_{e_{i_1}} \dots \partial_{e_{i_l}} \tilde S(0,0,0)
\end{equation}
where
$$ \psi_{i_l}(x) \coloneqq \psi( \tilde Q(x), \tilde Q(x+e_{i_l}) )$$
for all $x \in \F_2^{M+1+d}$.
\item[(3)]  We have $2\tilde S(0^M,1,\omega)=2S_0(0)$ for all $\omega \in \{0,1\}^d$.
\end{itemize}

To obey (1) (and the $\tilde Q$ component of (0)), we define $\tilde Q \colon \F_2^{M+1+d} \to Y$ by the formula
$$ \tilde Q(a,t,\omega) \coloneqq q_\omega + tq_0 - q_0 + Q(a) - t Q(0)$$
for $a \in \F_2^M$, $t \in \F_2$, $\omega \in \F_2^d$.  One easily verifies that $\tilde Q$ is a polynomial of degree $2$ that obeys (2) and the $\tilde Q$ component of (1).  By Lemma \ref{lift}, we can then find a map $\tilde S_0 \colon \F_2^{M+1+d} \to \frac{1}{2^5}\Z/\Z$ such that $(\tilde Q, \tilde S_0)$ is a $M+1+d$-cube in $X_{5,5}$ (and in fact in $X_{5,1}$).  We need to then find an element $\tilde S$ of the coset $\tilde S_0 + \Poly(\F_2^{M+1+d} \to \frac{1}{2^5}\Z/\Z)$ which obeys the following properties:

\begin{itemize}
    \item[(0)] For $a \in \F_2^M$, we have $\tilde S(a,0^{n-M}) = S_0(a)$.
    It is not necessarily the case that $\tilde S$ agrees with $S$ on $\F_2^M$ (so that $(\tilde Q, \tilde S)$ lies in $\Sigma^{(1+d)}_{(Q,S)}$).
    \item[(2)] For every $1 \leq l \leq k-1$ and $1 \leq i_1 < \dots < i_l \leq M$, \eqref{psi-S} holds.
    \item[(3)] We have $2\tilde S(0^M,1,\omega)=2S_0(0)$ for all $\omega \in \{0,1\}^d$.
\end{itemize}

We will enforce each of these properties (0), (2), (3) in turn (making sure that each modification of $\tilde S$ that we make does not destroy any properties that we have already established).

We first locate a function $\tilde S \in \tilde S_0 + \Poly(\F_2^{M+1+d} \to \frac{1}{2^5}\Z/\Z)$ obeying (0).  Observe that $(Q_0,\tilde S(\cdot,0^{1+d}))$ and $(Q_0,S_0)$ are both $M$-cubes in $X_{5,5}$, and hence the restriction of $\tilde S_0 - S$ to $\F_2^M$ lies in $\Poly^5(\F_2^M \to \frac{1}{2^5}\Z/\Z)$.  By composing this polynomial with the obvious projection from $\F_2^{M+1+d}$ to $\F_2^M$, we conclude that $\tilde S_0 - S$ agrees on $\F_2^M \times \{0\} \times \{0^d\}$ with some polynomial in $\Poly^5(\F_2^{M+1+d} \to \frac{1}{2^5}\Z/\Z)$.  Subtracting this polynomial from $\tilde S_0$, we obtain an element $\tilde S$ of $\tilde S_0 + \Poly(\F_2^{M+1+d} \to \frac{1}{2^5}\Z/\Z)$ oyeing property (0).

We now enforce the property (2) by induction on $i_1$. More precisely, we assume inductively that we have found $\tilde S \in \tilde S_0 + \Poly(\F_2^{M+1+d} \to \frac{1}{2^5}\Z/\Z)$ obeying (0) for which (1) has already been established in the case $i_1 < i_*$
for some $1 \leq i_* \leq M$, and wish to modify $\tilde S$ so that it still obeys (0) but now also obeys (1) in the case $i_1 \leq i_*$.

Observe that if we add or subtract to $\tilde S$ a polynomial $P \in \Poly^5(\F_2^{M+1+d} \to \frac{1}{2^5}\Z/\Z)$ which vanishes on $\F_2^M \times \{0\} \times \{0^d\}$, and which also does not depend on the first $i_*-1$ coordinates in the sense that $\partial_{e_i} P = 0$ for $1 \leq i < i_*$, then $\tilde S$ continues to obey (0) and (1) for $i_1 < i_*$ (though again this may destroy property (d)).  We exploit this freedom to modify $\tilde S$ as follows.

First, we use the fact that $\rho = d^5 \psi$ to write the condition \eqref{S-eq} on the $M+1+d$-cube $(\tilde Q, \tilde S)$ as
$$
\partial_{h_1} \dots \partial_{h_5} ( \partial_h \tilde S - \psi( \tilde Q( \cdot ), \tilde Q( \cdot + h ) ) ) = 0$$
for all $h,h_1,\dots,h_5 \in \F_2^{M+1+d}$.  Equivalently, one has
\begin{equation}\label{hd}
\partial_h \tilde S - \psi( \tilde Q( \cdot ), \tilde Q( \cdot + h ) ) \in \Poly^4(\F_2^{M+1+d})
\end{equation}
for each $h \in \F_2^{M+1+d}$.  Applying this with $h = e_{i_*}$, we conclude that the function
$$ P \coloneqq \partial_{e_{i_*}} \tilde S - \psi_{i_*} $$
lies in $\Poly^4(\F_2^{M+1+d})$.  Now we look at the expression
$$ P(a,1,\omega) - P(a,0,0^d) = \partial_{e_{i_*}} \tilde S(a,1,\omega) - \psi_{i_*}(a,1,\omega) - \partial_{e_{i_*}} \tilde S(a,0,0^d) + 
\psi_{i_*}(a,0,0^d)$$
for $a \in 0^{i_*} \times \F_2^{M-i_*}$ and $\omega \in \F_2^d$.  Expanding $P$ out into monomials using Lemma \ref{explicit-desc}, we can write
$$ P(a,1,\omega) - P(a,0,0^d) = \sum_{l=1}^4 \sum_{i_* < i_1 < \dots < i_l \leq M+1+d; i_l > M} \frac{ c_{i_1,\dots,i_l} |x_{i_1}| \dots |x_{i_l}|}{2^{5-l}} \mod 1$$
for some coefficients $c_{i_1,\dots,i_l} \in \Z$, where $(x_1,\dots,x_{M+1+d}) \coloneqq (a,1,\omega)$.  If we then introduce the function $R \colon \F_2^{M+1+d} \to \frac{1}{2^5}\Z/\Z$ by the formula
$$ R(x_1,\dots,x_{M+1+d}) \coloneqq \sum_{l=1}^{4} \sum_{i_* < i_1 < \dots < i_l \leq M+1+d; i_l > M} \frac{ c_{i_1,\dots,i_l} |x_{i_*}| |x_{i_1}| \dots |x_{i_l}|}{2^{k-l}} \mod 1
$$
for $(x_1,\dots,x_{M+1+d}) \in \F_2^M$, then from Lemma \ref{explicit-desc} we see that\footnote{It is here that we need to have worked with $X_{5,5}$ instead of $X_{5,1}$, as we cannot guarantee that the quintic polynomial $R$ will be classical.} $R \in \Poly^5(\F_2^{M+1+d} \to \frac{1}{2^5}\Z/\Z)$ and that
$$ P(a,1,\omega) - P(a,0,0^d) = \partial_{e_{i_*}} R(a,1,\omega)$$
for $a \in 0^{i_*} \times \F_2^{M-i_*}$ and $\omega \in \F_2^d$.
Also $R$ vanishes on $\F_2^M$ and is invariant with respect to the first $i_*$ coordinates, so as discussed above we may freely subtract $R$ from $\tilde S$.  If we do so, then we now have
$$ P(a,1,\omega) - P(a,0,0^d) = 0$$
for all $a \in \F_2^M$ and $\omega \in \F_2^d$, which on further differentiation gives \eqref{psi-S} for $i_1=i_*$ as required.

Finally, we enforce the property (3).  As already observed, if we add or subtract to $\tilde S$ a polynomial $P \in \Poly^5(\F_2^{M+1+d} \to \frac{1}{2^5}\Z/\Z)$ which vanishes on $\F_2^M$, and which also does not depend on the first $M$ coordinates, then the properties (0), (2) remain unaffected.  To exploit this, recall that $\tilde S$ lies in the coset $\tilde S_0 + \Poly^5(\F_2^{M+1+d} \to \frac{1}{2^5}\Z/\Z)$; since $\tilde S_0$ takes values in $\frac{1}{2}\Z/\Z$, we conclude from \eqref{double} that
$$ 2\tilde S \in  \Poly^4(\F_2^{M+1+d} \to \frac{1}{2^4}\Z/\Z)$$
and hence from \eqref{double} again we may write
\begin{equation}\label{2p}
 2 \tilde S = 2P
 \end{equation}
for some $P \in \Poly^5(\F_2^{M+1+d} \to \frac{1}{2^5}\Z/\Z)$.  The function
$$ (a,t,\omega) \mapsto P(0^M,t,\omega) - P(0^M,0,0^d) $$
is then a quintic polynomial on $\F_2^{M+1+d}$ that vanishes on $\F_2^M$ and does not depend on the first $M$ coordinates; if we then define
$$ \tilde S'(a,t,\omega) \coloneqq S(a,t,\omega) - P(0^M,t,\omega) + P(0^M,0,0^d)$$
then $\tilde S'$ lies in $\tilde S_0 + \Poly^5(\F_2^{M+1+d} \to \frac{1}{2^5}\Z/\Z)$, obeys (0) and (2), and for each $\omega \in \{0,1\}^d$ we have
$$ 2\tilde S'(0^M,1,\omega) = 2 P(0^M,0,0^d) = 2 \tilde S(0^M,0,0^d) = 2S_0(0)$$
giving (3).  This completes the proof of Proposition \ref{solving}, and thus Theorem \ref{counter}.

\begin{remark}  If one replaces $X_{5,5}$ by $X_{5,1}$ in the above construction then one no longer obtains a counterexample to Conjecture \ref{inv-conj-strong}.  We sketch the proof of this as follows. By Remark \ref{alt-form}, the pseudo-quintic function $S$ takes the form
$$ S = \frac{\binom{R}{2} Q^{(2)} + P}{2} \mod 1$$
for some randomly chosen polynomials $R \in \Poly^3(\F_2^n \to \Z/4\Z)$, $Q^{(1)}, Q^{(2)} \in \Poly^2(\F_2^n \to \F_2)$, $P \in \Poly^5(\F_2^n \to \F_2)$ with $Q^{(1)} = R \mod 2$; note crucially that $P$ now takes values in the classical range $\F_2$ as opposed to the non-classical range $\frac{1}{2^5}\Z/\Z$.  After many applications of the Leibniz rule \eqref{leibniz} (and \eqref{fh}) we see that for any shifts $a,b,c,d,e \in \F_2^n$ we have the fifth derivative computation
$$ \partial_a \partial_b \partial_c \partial_d \partial_e S = \frac{\partial_a \partial_b Q^{(1)} \partial_c \partial_d Q^{(1)} \partial_e Q^{(2)} + \dots}{2}$$
where the $\dots$ are a sum of terms that are either constants in $\F_2$ (depending on $a,b,c,d,e$), or linear functions that resemble  permutations of $\partial_a \partial_b Q^{(1)} \partial_c \partial_d Q^{(1)} \partial_e Q^{(2)}$ (in fact there are $44$ terms of this latter type).  For $a,b,c,d,e$ chosen at random, it is true with positive probability that $\partial_a \partial_b Q^{(1)} = \partial_c \partial_d Q^{(1)}= 1$, so that the displayed term $\partial_a \partial_b Q^{(1)} \partial_c \partial_d Q^{(1)} \partial_e Q^{(2)}$ simplifies to $\partial_e Q^{(2)}$, while the other permutations of this term vanish.  From this one can conclude that with high probability, and for a given random shift $e$ the linear functions $\partial_e Q^{(2)}$ are measurable in the sense that they are a function of boundedly many shifts of $S$ by $e$ and other random shifts.  Similarly for $\partial_e Q^{(1)}$.  In a similar spirit, we have the fourth derivative computation
$$ \partial_a \partial_b \partial_c \partial_d S = \frac{\partial_a \partial_b Q^{(1)} \partial_c \partial_d Q^{(1)} Q^{(2)} + \dots}{2}$$
where the terms in $\dots$ take values in $\F_2$ and are either permutations of the displayed term, are combinations of functions already known to be measurable, or are linear.  By the preceding argument one can show that with high probability $Q^{(2)}$ is measurable up to a classical linear polynomial; and similarly for $Q^{(1)}$.  Finally, we have the second derivative computation
$$ \partial_a \partial_b S = \frac{\binom{R}{2} \partial_a \partial_b Q^{(2)} + \dots}{2}$$
where the terms in $\dots$ take values in $\F_2$ and are either combinations of functions already known to be measurable, or are cubic.  Repeating the previous argument, we conclude with high probability that $\binom{R}{2}$ (which one can check to be a classical quartic polynomial) is measurable up to a classical cubic polynomial.  Taking advantage of the ability to pointwise multiply in the classical range $\F_2$ using Lemma \ref{prod}, we conclude with high probability that $\binom{R}{2} Q^{(2)}$ is measurable up to a classical quintic polynomial.  Hence $S$ is measurable up to a quintic polynomial, which must then also be measurable since $S$ is measurable.  By a Fourier expansion, one can then show that $S$ correlates with a measurable quintic polynomial, giving Conjecture \ref{inv-conj-strong} in this case. Thus one can explain the need to work with the more complicated space $X_{5,5}$ instead of $X_{5,1}$ in order to destroy the ability to multiply polynomials together by working in non-classical ranges such as $\frac{1}{2^5}\Z/\Z$ instead of $\F_2$.
\end{remark}

\begin{remark}\label{Direct} By combining these constructions with the arguments in Appendix \ref{app}, we obtain a counterexample to Conjecture \ref{btz-conj}.  It is natural to ask whether there is a shortcut approach that could construct the counterexample to Conjecture \ref{btz-conj} more directly, without first building a counterexample to Conjecture \ref{inv-conj-strong}.  Morally speaking, this should proceed by starting with the space $\Hom_\Box({\mathcal D}^1(\F_2^\omega) \to X_{5,5})$, which is a compact $\F_2^\omega$-system that can be naturally equipped with a Haar measure.  This system is not ergodic, but a generic component of the ergodic decomposition should be a $5$-step ergodic $\F_2^\omega$-system that fails to be Abramov of order $5$ (cf., the role of the pair $(Q_0,S_0)$ in the above analysis). The rigorous verification of these claims seems to be of comparable complexity to the arguments just presented, and so we do not detail this more direct approach here.  

On the other hand, in the spirit of Remark \ref{rom}, Candela et al.~\cite{cgss-2023} recently established that this system corresponding to the generic component of the ergodic decomposition admits an extension that is Abramov of order $5$ (and it should even be a Weyl system of order $5$, in the sense of, e.g., \cite{jst-tdsystems}). 
\end{remark}

\appendix

\section{Nilspaces, filtered abelian groups, and non-classical polynomials}\label{nilspace-app}

In this section, we gather several standard (and primarily algebraic) facts about nilspaces, filtered abelian groups, and polynomial maps. Most of the concepts introduced below are not new, but the terminology varies between authors. For example, the notion of a $k$-cocycle (see Definition \ref{nil-cocycle-def}) was first introduced by Antol\'in-Camarena and Szegedy \cite[Definition 2.14]{camarena-szegedy} under the name of ``degree $k$ cocycle''. The same concept was subsequently treated by Candela \cite{candela0}, and later by Gutman, Manners, and Varj\'u \cite[Definition~4.8]{gmvII}, who use the terminology ``$k$-cocycle''. To avoid confusion between the degree of a polynomial and that of a cocycle, we follow the terminology of Gutman--Manners--Varj\'u in this paper. We also note that our Definition~\ref{nil-cocycle-def} is essentially the same as that in \cite{camarena-szegedy}, whereas the version in \cite{gmvII} includes an additional continuity requirement.  

Similarly, the notions of $k$-coboundaries and $k$-extensions, as we use them here, were originally called degree $k$ coboundaries and degree $k$ extensions in \cite{camarena-szegedy}, and were further developed in Candela \cite{candela0}.  

Finally, we emphasize that our concept of a $k$-cocycle is \emph{not} the same as the similarly named notion from homological algebra, where $k$-cocycles are $k$-cochains with vanishing coboundary (in fact, from the homological perspective all our cocycles are $1$-cocycles).

%


\subsection{Nilspaces}\label{nilspace-sec}

Host and Kra \cite{hk-parallel}  introduced a combinatorial framework for cubes in abstract sets in dimensions $2$ and $3$ as an abstraction of the concept of a parallelepiped in a group or dynamical system, and Antol\'in-Camarena and Szegedy \cite{camarena-szegedy} later extended this to all dimensions, thereby defining the general notion of abstract nilspaces. They can be defined in the set-theoretic, topological, and measurable categories, but we will only need to consider finite nilspaces, which allows us to work in the technically simpler set-theoretic category.  We recall the definition of a nilspace, following \cite[Definition 1.2.1]{candela0}:

\begin{definition}[Nilspaces]\label{nilspace-def} A nilspace is a set $X$ together with a collection of sets $C^n(X) \subset X^{\{0,1\}^n}$ for each non-negative integer $n$, satisfying the following axioms: 
\begin{itemize}
    \item[(i)] (Composition) For every $m,n \geq 0$ and every cube morphism $\phi \colon \{0,1\}^m \to \{0,1\}^n$ (by which we mean a function that extends to an affine map from $\R^m$ to $\R^n$) and every $c \in C^n(X)$, we have $c \circ \phi \in C^m(X)$.
    \item[(ii)] ($0$-ergodicity) $C^0(X) = X$.  If we have the stronger property $C^1(X) = X^{\{0,1\}}$, we say that the nilspace is \emph{ergodic} (or \emph{$1$-ergodic}).
    \item[(iii)] (Corner completion) Let $n \geq 1$, and let $c' \colon \{0,1\}^n \backslash \{1\}^n \to X$ be a function such that every restriction of $c'$ to an $(n-1)$-face containing $0^n$ is in $C^{n-1}(X)$. Then there exists $c \in C^n(X)$ such that $c(v) = c'(v)$ for all $v \neq 1^n$. If this $c$ is unique, we say that $X$ is an \emph{$(n-1)$-step nilspace}.
\end{itemize}
Elements of $C^n(X)$ will be referred to as \emph{$n$-cubes} in $X$.

A \emph{nilspace morphism} $\phi \colon X \to Y$ between two nilspaces is a function that preserves $n$-cubes for every $n \geq 0$, in the sense that $(\phi(x_\omega))_{\omega \in \{0,1\}^n} ) \in C^n(Y)$ whenever $(x_\omega)_{\omega \in \{0,1\}^n} \in C^n(X)$.  The space of such morphisms will be denoted $\Hom_\Box(X \to Y)$.
\end{definition}

Clearly the collection of nilspaces and their morphisms form a category.  It is also easy to see that if a nilspace $X$ is $k$-step, then it is also $k'$-step for any $k' \geq k$.

\begin{remark}[Ergodic decomposition]\label{ergodic-decomposition}  In much of the literature (e.g., \cite{candela0}) the term ``nilspace'' is used to denote what we call an ``ergodic nilspace'', but it will be convenient for us to only impose the weaker axiom of $0$-ergodicity in our basic definitions.   In any event, it is often not difficult to reduce to the ergodic case via the following \emph{ergodic decomposition}.  If $X$ is a nilspace, we can define a relation\footnote{This is a special case of a more general class of equivalence relations $\sim_k$ one can define on nilspaces; see \cite[Definition 3.2.3]{candela0}.} $\sim_0$ on $X$ by declaring $x \sim_0 y$ if $(x,y) \in C^1(X)$.  It is not difficult to verify that this is an equivalence relation, that each equivalence class has the structure of an ergodic nilspace, and the original nilspace $X$ is the disjoint union of these ergodic nilspaces; see \cite[Lemma 3.1.8]{candela0}.  Because of this, many of the foundational results on ergodic nilspaces (such as those set out in \cite{candela0}) extend without difficulty to the more general nilspace setting.
\end{remark}

\begin{remark}[Cube spaces as nilspaces]\label{cube-nilspace}  If $X$ is a nilspace and $d \geq 0$, then the collection $C^d(X)$ of $d$-cubes in $X$ is itself a nilspace, with cube structure given by
$$ C^n(C^d(X)) \coloneqq C^{d+n}(X)$$
for all $n \geq 0$, after performing the identification
\begin{equation}\label{nil-ident}
(x_\omega)_{\omega \in \{0,1\}^{d+n}} \equiv ((x_{\omega,\omega'})_{\omega \in \{0,1\}^d})_{\omega' \in \{0,1\}^n}
\end{equation}
that interprets any $(d+n)$-cube $(x_\omega)_{\omega \in \{0,1\}^{d+n}} \in C^{d+n}(X)$ as an $n$-cube of $d$-cubes.  One can easily check that $C^d(X)$ obeys the nilspace axioms, and is $k$-step if $X$ is $k$-step, although we caution that $C^d(X)$ need not be ergodic even when $X$ is ergodic (this is a primary reason why we do not impose ergodicity in our definition of a nilspace).
\end{remark}

\begin{remark}[Morphism spaces as nilspaces]\label{morphism-nilspace}  If $X, Y$ are nilspaces, then the collection $\Hom_\Box(X \to Y)$ of nilspace morphisms from $X$ to $Y$ is itself a nilspace, with the cube structure given by
$$ C^n(\Hom_\Box(X \to Y)) \coloneqq \Hom_\Box(X \to C^n(Y))$$
for all $n \geq 0$, where we view a map from $X$ to $C^n(Y) \subset Y^{\{0,1\}^n}$ as a $\{0,1\}^n$-tuple of maps from $X$ to $Y$ in the obvious fashion.  One can easily check that $\Hom_\Box(X \to Y)$ obeys the nilspace axioms, and is $k$-step if $Y$ is $k$-step.  Again, we caution that $\Hom_\Box(X \to Y)$ need not be ergodic even when $X,Y$ are both ergodic.
\end{remark}

By definition, a nilspace morphism $\phi \colon X \to Y$ has to preserve $n$-cubes for every $n \geq 0$.  But if $Y$ is $k$-step, it turns out one only has to check preservation of $k+1$-cubes:

\begin{lemma}[Preserving $k+1$-cubes suffice]\label{k-suffices}  Let $X$ be a nilspace, $Y$ be a $k$-step nilspace for some $k \geq 0$, and let $\phi \colon X \to Y$ be a map that preserves $k+1$-cubes.  Then $\phi$ is a nilspace morphism.
\end{lemma}

\begin{proof}
    From the composition axiom (i) one easily verifies that if  $\phi$ preserves $k+1$-cubes, then it also preserves $n$-cubes for any $n \leq k+1$.  In the opposite direction, if $\phi$ preserves $k+1$-cubes and $n>k+1$, then $\phi$ maps an $n$-cube to a tuple $(y_\omega)_{\omega \in \{0,1\}^n}$ with the property that every $k+1$-dimensional face of this tuple is a $k+1$-cube.  Using the completion axiom (and the fact that $Y$ is $k'$-step for every $k' \geq k$) one easily then verifies by induction that every $n'$-dimensional face of this tuple is a $n'$-cube for $k+1 \leq n' \leq n$; setting $n'=n$ gives the claim.
\end{proof}

If $F \colon X \to Z$ is a map from a nilspace $X$ to an abelian group $Z = (Z,+)$, we can define the \emph{derivative} $dF \colon C^1(X) \to Z$ on the nilspace $C^1(X)$ by the formula
$$ dF(a,b) \coloneqq F(b)-F(a).$$
We can iterate this construction using Remark \ref{cube-nilspace} to define higher derivatives\footnote{In particular, we caution that $d$ does \emph{not} form a chain complex and should \emph{not}  be interpreted as an exterior derivative: $d^2 \neq 0$.} $d^k F \colon C^k(X) \to Z$ for any $k \geq 0$, with the convention $d^0 F = F$.  Explicitly, we have
$$ d^k F( (x_\omega)_{\omega \in \{0,1\}^k} ) = \sum_{\omega \in \{0,1\}^k} (-1)^{k-|\omega|} F(x_\omega).$$

Now we give a construction for extending a nilspace by a cocycle.

\begin{definition}[Cocycles on nilspaces] \label{nil-cocycle-def} \cite[Definition 2.14]{camarena-szegedy}
Let $X$ be a nilspace, $Z$ be an abelian group, and $k \geq 0$.  A \emph{$k$-cocycle on $X$ taking values in $Z$} is a function $\rho \colon C^{k+1}(X) \to Z$ obeying the following axioms:
\begin{itemize}
\item[(i)] (Symmetry)  If $(x_\omega)_{\omega \in \{0,1\}^{k+1}} \in C^{k+1}(X)$ is a $k+1$-cube in $X$, and $\sigma \colon \{0,1\}^{k+1} \to \{0,1\}^{k+1}$ is any map formed by permuting the $k+1$ coordinates, then
$$ \rho( (x_{\sigma(\omega)})_{\omega \in \{0,1\}^{k+1}} ) = \rho( (x_\omega)_{\omega \in \{0,1\}^{k+1}} ).$$
\item[(ii)]  (Cocycle)  If $x, y, z \in C^k(X)$ are $k$-cubes with\footnote{We denote by $(x,y)$ the $k+1$-cube which is the concatenation of the $k$-cubes $x,y$.} $(x,y), (y,z) \in C^1(C^k(X)) \equiv C^{k+1}(X)$ are $k+1$-cubes (which implies that $(x,z)$ is also a $k+1$-cube, thanks to Remark \ref{cube-nilspace}), then
$$ \rho(x,z) = \rho(x,y) + \rho(y,z).$$
\end{itemize}
We say that $\rho \colon C^{k+1}(X) \to Z$ is a \emph{$k$-coboundary on $X$ taking values in $Z$} if we have $\rho = d^{k+1} F$ for some $F \colon X \to Z$.
\end{definition}

\begin{example}  A $1$-cocycle is a map $\rho \colon C^2(X) \to Z$ obeying the symmetry axiom
$$ \rho( x_{00}, x_{01}, x_{10}, x_{11} ) = \rho( x_{00}, x_{10}, x_{01}, x_{11} ) $$
for all $(x_{00}, x_{01}, x_{10}, x_{11}) \in C^2(X)$, and the cocycle axiom
$$ \rho( x_0, x_1, z_0, z_1 ) = \rho( x_0, x_1, y_0, y_1) + \rho(y_0, y_1, z_0, z_1)$$
whenever $(x_0, x_1, y_0, y_1), (y_0, y_1, z_0, z_1) \in C^2(X)$.  A $1$-coboundary is a map $\rho \colon C^2(X) \to Z$ of the form
$$ \rho( x_{00}, x_{01}, x_{10}, x_{11} ) = F(x_{00}) - F(x_{01}) - F(x_{10}) + F(x_{11})$$
for all $(x_{00}, x_{01}, x_{10}, x_{11}) \in C^2(X)$.
\end{example}

It is easy to see that every $k$-coboundary is a $k$-cocycle; indeed, the collection of $k$-coboundaries forms a subgroup of the abelian group of $k$-cocycles.  However, it will be crucial for our main results that the converse is not always true, so that nilspaces can have non-trivial ``$k$-cohomology''.

\begin{remark}  Axiom (ii) and the nilspace axioms imply that $\rho(x,x)=0$ for all $x \in C^k(X)$, and that $\rho(x,y) = -\rho(y,x)$ for all $(x,y) \in C^{k+1}(X)$.  As a consequence, the symmetry axiom (i) is equivalent to the stronger axiom
$$ \rho( (x_{\theta(\omega)})_{\omega \in \{0,1\}^{k+1}} ) = (-1)^{r(\theta)} \rho( (x_\omega)_{\omega \in \{0,1\}^{k+1}} )$$
whenever $\theta \colon \{0,1\}^{k+1} \to \{0,1\}^{k+1}$ is a cube morphism and $r(\theta)$ is the number of $1$s in $\theta(0^{k+1})$ (informally, $r(\theta)$ is the number of face reflections needed to generate $\theta$).  This alternate formulation of axiom (i) is the one used in \cite[Definition 3.3.14]{candela0}.
\end{remark}

Now we introduce a key construction which originated from \cite[Proposition 3.1 and (5)]{camarena-szegedy}. 

\begin{proposition}[Skew products]\label{skewprod}  Let $k \geq 0$, let  $X$ be a $k$-step nilspace, and let $\rho \colon C^{k+1}(X) \to Z$ be a $k$-cocycle on $X$ taking values in an abelian group $Z$.  Then we can define a nilspace $X \rtimes^{(k)}_\rho Z$ to be the Cartesian product $X \times Z$ whose $n$-cubes for $n \geq 0$ consist of those tuples $((x_\omega,z_\omega))_{\omega \in \{0,1\}^n}$ for which $(x_\omega)_{\omega \in \{0,1\}^n}$ is an $n$-cube in $X$, and one has the equation
\begin{equation}\label{z}
 \sum_{\omega \in \{0,1\}^{k+1}} (-1)^{k+1-|\omega|} z_{\phi(\omega)} = \rho( (x_{\phi(\omega)})_{\omega \in \{0,1\}^{k+1}})
 \end{equation}
whenever $\phi \colon \{0,1\}^{k+1} \to \{0,1\}^n$ is a $k+1$-dimensional face of $\{0,1\}^n$ (this condition is vacuous when $n < k+1$).  If $X$ is $k$-step, then so is $X \rtimes^{(k)}_\rho Z$.

Finally, every $n$-cube $(x_\omega)_{\omega \in \{0,1\}^n}$ in $X$ has at least one lift $((x_\omega, z_\omega))_{\omega \in \{0,1\}^n}$ to an $n$-cube in $X \rtimes^{(k)}_\rho Z$.
\end{proposition}

\begin{proof}  The claim that $X \rtimes^{(k)}_\rho Z$ is a nilspace is \cite[Proposition 3.3.26]{candela0} (with slightly different notation).  The conclusion about the $k$-step nature of $X \rtimes^{(k)}_\rho Z$ follows from the $k$-step nature of $X$ and the equation \eqref{z} applied to the identity face $\phi \colon \{0,1\}^{k+1} \to \{0,1\}^{k+1}$, which constrains the final component $z_{1^{k+1}}$ of the $z_\omega$ in terms of the other components $z_\omega$ and the base $k+1$-cube $(x_\omega)_{\omega \in \{0,1\}^{k+1}}$.

To prove the final claim, we set $z_\omega \coloneqq 0$ for $|\omega| < k+1$, and whenever $|\omega| = k+1$ we set
$$
z_\omega \coloneqq \rho( (x_{\phi_\omega(\alpha)})_{\alpha \in \{0,1\}^{k+1}})$$
where $\phi_\omega \colon \{0,1\}^{k+1} \to \{0,1\}^n$ is the unique face map that sends $1^{k+1}$ to $\omega$.  The tuple $((x_\omega,z_\omega))_{|\omega| \leq k+1}$ then is an $n'$-cube on $X \rtimes^{(k)}_\rho Z$ when restricted to any $n'$-face in $\{ \omega \in \{0,1\}^n: |\omega| \leq k+1\}$ with $n' \leq k+1$.  By multiple applications of the completion axiom on the $k$-step nilspace $X \rtimes^{(k)}_\rho Z$ (or by \cite[Lemma 3.1.5]{candela0}), we may (uniquely) complete this tuple to an $n$-cube $((x_\omega,z_\omega))_{\omega \in \{0,1\}^n}$ on $X \rtimes^{(k)}_\rho Z$, whose first coordinates $x_\omega$ must agree with the original $n$-cube $(x_\omega)_{\omega \in \{0,1\}^n}$ on $X$ since $X$ is $k$-step.  This gives the claim.
\end{proof}

We refer to $X \rtimes^{(k)}_\rho Z$ as the \emph{$k$-skew product} of the nilspace $X$ and the abelian group $Z$ by the cocycle $\rho$.  The map $\pi \colon X \rtimes^{(k)}_\rho Z \to X$ defined by $\pi(x,z) \coloneqq x$ will be called the \emph{factor map}; it is immediate that this is a nilspace morphism.

\begin{example}  If $Z$ is an abelian group, then the $k$-step nilspace ${\mathcal D}^k(Z)$ (defined in the next section) can be thought of as the skew product $\mathrm{pt} \rtimes^{(k)}_0 Z$ of a point $\mathrm{pt}$ and $Z$ by the zero cocycle $0$.
\end{example}

\begin{example}  If $\rho = d^{k+1} F$ is a $k$-coboundary, then the skew product $X \rtimes^{(k)}_\rho Z$ is isomorphic as a nilspace to the product nilspace $X \times {\mathcal D}^k(Z) = X \rtimes_0 Z$, with the isomorphism defined by mapping $(x,z)$ to $(x,z-F(x))$.  More generally, adding or subtracting a $k$-coboundary from a cocycle does not affect the skew product up to nilspace isomorphism.  
\end{example}

\begin{remark} In \cite[Definition 3.3.13]{candela0}, a more abstract notion of a
$k$-extension of a nilspace $X$ is defined, and it is shown in \cite[Lemma 3.3.21]{candela0} that any such extension can be written as a $k$-skew product $X \rtimes^{(k)}_\rho Z$ for some $k$-cocycle after specifying a section of the extension; the $k$-coboundaries correspond to those extensions which are \emph{split}.  It is also shown in \cite[Theorem 3.2.19, Lemma 3.3.28]{candela0} that an ergodic $k$-step nilspace can be expressed as a tower
$$ \mathrm{pt} \rtimes^{(1)}_{\rho_1} Z_1 \rtimes^{(2)}_{\rho_2} Z_2 \dots \rtimes^{(k)}_{\rho_k} Z_k$$
of $k$ successive skew products with abelian groups $Z_1,\dots,Z_k$ (where we apply the skew product construction from left to right).  However, we will not need these results here.
\end{remark}

\subsection{Filtered abelian groups}

The nilspaces that we shall consider in this paper shall be constructed out of filtered abelian\footnote{One can also build nilspace structures out of non-abelian filtered groups, and in particular out of nilpotent groups; see for instance \cite[\S 2.2]{candela0}.  However, we will not need these more general nilspaces.} groups, and their extension by cocycles.  We first review the definition of a filtered abelian group.

\begin{definition}[Filtered abelian group] (see e.g., \cite[\S 6]{green-tao-nilmanifolds}) A \emph{filtered abelian group} $G = (G,(G_i)_{i \geq 0})$ is an abelian group $G = (G,+)$ (which we will usually think of as being discrete), equipped with a filtration
$$ G = G_0 \geq G_1 \geq G_2 \geq \dots$$
of subgroups $G_i$.  If $G_1=G_0=G$, we say that the filtered abelian group is \emph{ergodic}.

A \emph{filtered homomorphism} from one filtered group $G = (G,(G_i)_{i \geq 0})$ to another $H = (H,(H_i)_{i \geq 0})$ is a group homomorphism $\phi \colon G \to H$ such that $\phi(G_i) \leq H_i$ for all $i \geq 0$.  

If $G$ is a filtered group and $k \geq 0$, we define the \emph{$k^{\mathrm{th}}$ Host--Kra group $G^{[k]} \leq G^{\{0,1\}^k}$} of $G$ to be the filtered abelian group of tuples of the form
\begin{equation}\label{tup}
 ( \sum_{\alpha \in \{0,1\}^k} h_\alpha \prod_{i: \alpha_i=1} \omega_i  )_{\omega \in \{0,1\}^k}
 \end{equation}
where $h_\alpha \in G_{|\alpha|}$ for all $\alpha \in \{0,1\}^k$, where $|\alpha| \coloneqq \alpha_1+\dots+\alpha_n$, and with the subgroup $(G^{[k]})_i$ of the filtered abelian group $G^{[k]}$ defined to be the group of tuples of the form
\eqref{tup} with $h_\alpha \in G_{|\alpha|+i}$ for all $\alpha \in \{0,1\}^k$.  One can easily verify that $G^{[k]}$ is also a filtered abelian group.

If $G_i=\{0\}$ for $i > d$, we say that the filtered group $G$ is \emph{of degree at most $d$}.
An abelian group $G$ is given the \emph{degree $d$ filtration} for some $d \geq 0$ if $G_i = G$ for $i \leq d$ and $G_i = \{0\}$ for $i>d$, in which case we denote the associated filtered abelian group as ${\mathcal D}^d(G)$ (cf. \cite[Definition 2.2.30]{candela0}).
\end{definition}

\begin{example}\label{example-cubes}  After some routine relabeling, we have
$$G^{[0]} = G = \{ x: x \in G \},$$
\begin{equation}\label{g1}
G^{[1]} = \{ (x, x+h_1): x \in G; h_1 \in G_1 \}
\end{equation}
and
\begin{equation}\label{g2-param}
G^{[2]} = \{ (x, x+h_1, x+h_2, x+h_1+h_2+h_{12}): x \in G; h_1,h_2 \in G_1; h_{12} \in G_2 \}
\end{equation}
and
\begin{equation}\label{g3-param}
\begin{split}
G^{[3]} = &\{ (x, x+h_1, x+h_2, x+h_3, x+h_1+h_2+h_{12}, x+h_1+h_3+h_{23}, \\ & \quad x+h_2+h_3+h_{13}, x+h_1+h_2+h_3+h_{12}+h_{13}+h_{23}+h_{123}): \\ & \quad  x \in G; h_1,h_2,h_3 \in G_1; h_{12},h_{13},h_{23} \in G_2; h_{123} \in G_3 \}.
\end{split}
\end{equation}
In the case when $G$ has the degree $1$ filtration ${\mathcal D}^1(G)$, one can omit the $h_{12}, h_{13}, h_{23}, h_{123}$ terms in the above formulae.

From the construction one has a canonical identification \begin{equation}\label{gk-iter}
(G^{[d]})^{[n]} \equiv G^{[d+n]}
\end{equation}
of filtered abelian groups for any $d,n \geq 0$ defined by
$$ (g_\omega)_{\omega \in \{0,1\}^{d+n}} \equiv ((g_{(\omega,\omega')})_{\omega \in \{0,1\}^d})_{\omega' \in \{0,1\}^n}$$
for all $(g_\omega)_{\omega \in \{0,1\}^{d+n}} \in G^{[d+n]}$; compare with \eqref{nil-ident}.
\end{example}

Every filtered abelian group can be viewed as a nilspace.

\begin{lemma}[Filtered groups are nilspaces]\label{lem-filtgrpnilspace}  If $G = (G, (G_i)_{i\geq 0})$ is a filtered abelian group, then $G$ can be given the structure of a nilspace by setting $C^n(G) \coloneqq G^{[n]}$.  This  will be an ergodic nilspace if and only if $G$ is ergodic.  If $k \geq 0$, then $G$ is of degree at most $k$ as a filtered abelian group if and only if it is a $k$-step nilspace. 
\end{lemma}

\begin{proof}  See \cite[Proposition 2.2.8]{candela0} (which in fact proves this result even in the non-abelian case).  The proof in \cite{candela0} is written only in the ergodic case,  but an inspection of the arguments reveals that it also holds in the non-ergodic setting.
\end{proof}

\begin{remark}  If $G$ is a filtered abelian group, then we may potentially have defined two nilspace structures on $G^{[k]}$; one arising from applying the above lemma to the filtered abelian group $G^{[k]}$, and the other by applying the above lemma to $G$ and then using the nilspace structure on $n$-cubes $C^n(G)$ from Remark \ref{cube-nilspace}.  However, it is easy to see that these two nilspace structures coincide.
\end{remark}

In view of the above lemma, we can now define nilspace morphisms between filtered abelian groups.
  As it turns out, these nilspace morphisms have a nice characterisation in terms of difference operators.  If $G,H$ are (filtered) abelian groups and $h \in G$ is a shift, we define the shift operator $T^h$ and the difference operator $\partial_h$ on functions $f \colon G \to H$ by the formula
$$ T^h f(x) \coloneqq f(x+h)$$
and
$$ \partial_h f(x) \coloneqq f(x+h)-f(x),$$
thus $\partial_h = T^h - 1$.  Clearly these operators commute with each other, with $h \mapsto T^h$ being an action of $G$; we also  note the cocycle identity
\begin{equation}\label{cocycle-identity}
\partial_{h+k} = \partial_h + T^h \partial_k
\end{equation}
for any $h,k \in G$.

\begin{lemma}[Characterization of nilspace morphisms] \label{cubechar} Let $f \colon G \to H$ be a map from one filtered abelian group $G = (G,(G_i)_{i \geq 0})$ to another $H = (H,(H_i)_{i \geq 0})$.  Then $f$ is a nilspace morphism if and only if
$$ \partial_{h_1} \dots \partial_{h_l} f(x) \in H_{i_1+\dots+i_l}$$
    for all $l \geq 0$, $i_1,\dots,i_l \geq 0$, $x \in G$, and $h_j \in G_{i_j}$ for $j=1,\dots,l$.  In fact, it suffices to check this condition for $h_j \in E_{i_j}$, where for each $i$, $E_i$ is a set of generators for $G_i$.
\end{lemma}

\begin{proof} See \cite[Theorem B.3, Proposition B.8]{gtz} or \cite[Theorem 2.2.14]{candela0} (the latter statement is written in the ergodic case, but the proof extends without difficulty to the non-ergodic setting).
\end{proof}

As one corollary of this lemma, we see that the space $\Hom_\Box(G \to H)$ of nilspace morphisms from one filtered abelian group $G$ to another $H$ is an abelian group, which contains the space of filtered homomorphisms from $G$ to $H$ as a subgroup.  In fact $\Hom_\Box(G \to H)$ naturally has the structure of a filtered abelian group, in a manner consistent with the nilspace structure on $\Hom_\Box(G \to H)$ already constructed in Remark \ref{cube-nilspace}: see \cite[Proposition B.6]{gtz}.  The translation operators $x \mapsto x+h$ on $G$ are also nilspace morphisms for any $h \in H$.

\subsection{Polynomials}

We now define the notion of a polynomial:

\begin{definition}[Polynomials]\label{poly-def}  If $X$ is nilspace, $H$ is an abelian group, and $d \geq 0$, a \emph{polynomial of degree at most $d$} from $X$ to $H$ is a nilspace morphism from $X$ to ${\mathcal D}^d(H)$.  When $X$ is a filtered abelian group $G$, we can equivalently define a polynomial by requiring that
$$ \partial_{h_1} \dots \partial_{h_l} P = 0$$
whenever $i_1,\dots,i_l \geq 0$ are such that $i_1+\dots+i_l > d$, and $h_j \in G_{i_j}$ for $j=1,\dots,l$; see \cite[Theorem 2.2.14]{candela0} for a proof of this equivalence.
The space of such polynomials will be denoted $\Poly^d(X \to H)$, thus
$$ \Poly^d(X \to H) \equiv \Hom_\Box(X \to {\mathcal D}^d(H)).$$
In particular, $\Poly^d(X \to H)$ is an abelian group, and when $X$ is a filtered abelian group it acquires a translation action $h \mapsto T^h$ of $G$.  If $H = \T$, we abbreviate $\Poly^d(X \to \T)$ as $\Poly^d(X)$, and refer to elements of $\Poly^d(X)$ as \emph{non-classical polynomials of degree at most $d$} on $X$.  By convention, we set $\Poly^d(X \to H) = \{0\}$ for $d<0$.  For an abelian group $G$, we often abbreviate $\Poly^d({\mathcal D}^1(G) \to H)$ as $\Poly^d(G \to H)$ (and $\Poly^d({\mathcal D}^1(G))$ as $\Poly^d(G)$).
\end{definition}

From the definitions we see that we can define polynomials recursively on a filtered abelian group $G$: a map $P \colon G \to H$ lies in $\Poly^d(G \to H)$ if and only if $\partial_h P \in \Poly^{d-i}(G \to H)$ for all $i \geq 1$ and $h \in G_i$.  We remark that \emph{classical polynomials} correspond to the case when $H$ is a field $\F$, and $G$ is a vector space over that field (equipped with the degree $1$ filtration).

\begin{remark} The space of polynomials $\Poly^d(G)$ in a filtered abelian group $G$ is sensitive to the filtration structure on $G$. For instance, the function $P \colon \Z/2\Z \to \T$ defined by $P(x) \coloneqq x/2$ is a polynomial of degree $1$ if $\Z/2\Z$ is given the degree $1$ filtration ${\mathcal D}^1(\Z/2\Z)$, but is a polynomial of degree $2$ if $\Z/2\Z$ is instead given the degree $2$ filtration ${\mathcal D}^2(\Z/2\Z)$.  Informally, the difference operator $\partial_1$ is a first-order operator in the former case, but a second-order operator in the latter case.
\end{remark}

If $P \colon G \to H$ is a map from a filtered abelian group $G$ to an abelian group $H$, recall from Section \ref{nilspace-sec} that we can define derivatives $d^k P \colon G^{[k]} \to H$ for any $k \geq 0$.
By expanding all the definitions, we obtain a familiar-looking relationship between polynomials and derivatives:

\begin{proposition}[Polynomials and derivatives]\label{polyderiv}  Let $P \colon G \to H$ be a map from a filtered abelian group $G$ to an abelian group $H$.  If $k \geq -1$, then $P$ is a polynomial of degree at most $k$ if and only if $d^{k+1} P = 0$.  In particular, for $k \geq 0$, we see that $P$ is a polynomial of degree at most $k$ if and only if $dP$ is a polynomial of degree at most $k-1$.
\end{proposition}

As one application of this proposition, we have the following familiar-looking result about multiplication of polynomials (cf. \cite[Exercise 1.6.10]{tao2012higher}):

\begin{lemma}[Products of polynomials]\label{prod}  Let $G$ be a filtered abelian group, and let $R$ be a ring.  If $P_1 \colon G \to R$, $P_2 \colon G \to R$ are polynomials of degree at most $d_1,d_2$ respectively, then $P_1 P_2 \colon G \to R$ is a polynomial of degree at most $d_1+d_2$.
\end{lemma}

\begin{proof}  Observe the Leibniz rule
\begin{equation}\label{leibniz}
 \partial_h (P_1 P_2) = (\partial_h P_1) P_2 + (T^h P_1) \partial_h P_2
 = (\partial_h P_1) P_2 + P_1 \partial_h P_2 + (\partial_h P_1) \partial_h P_2
 \end{equation}
for any $h \in G$.  The claim now follows by induction on the combined degree $d_1+d_2$.
\end{proof}

If $G$ is a filtered abelian group which is also an elementary abelian $2$-group, then by \eqref{2h} we have $2\partial_h = -\partial_h^2$ for any $h \in G$.  When combined with Proposition \ref{polyderiv}, this gives

\begin{proposition}[Doubling lowers degree in $2$-groups]\label{double-prop}  Let $G$ be a filtered abelian group that is also an elementary abelian $2$-group, and let $H$ be an abelian group.  If $P \in \Poly^k(G \to H)$ for some $k \geq 1$, then $2P \in \Poly^{k-1}(G \to H)$.
\end{proposition}

In fact this property holds in the larger class of $2$-homogeneous filtered abelian groups, but we will not need to establish this fact here.

In the case of non-classical polynomials on a finite-dimensional vector space $\F_2^n$ over the field of two elements, we have an explicit description of such polynomials:

\begin{lemma}[Explicit description of polynomials]\label{explicit-desc}  Let $n \geq 0$ and $d \geq 0$.  Then a function $P \colon \F_2^n \to \T$ is of degree at most $d$ if and only if it takes the form
$$ P(x_1,\dots,x_n) = \alpha + \sum_{k=1}^d \sum_{1 \leq i_1 < \dots < i_k \leq n} \frac{c_{i_1,\dots,i_k} |x_{i_1}| \dots |x_{i_k}|}{2^{d+1-k}} \mod 1
 $$
 for all $x_1,\dots,x_n \in \F_2$ and some $0 \leq \alpha < 1$ and some integers $0 \leq c_{i_1,\dots,i_k} < 2^{d+1-k}$, where $|x| \coloneqq 1_{x=1}$.  The coefficients $\alpha$ and $c_{i_1,\dots,i_k}$ are uniquely determined.  Indeed we have
 $$ \alpha = P(0) \mod 1$$
 and
 $$ \frac{c_{i_1,\dots,i_k}}{2^{d+1-k}} = \partial_{i_1} \dots \partial_{i_k} P(0) \mod 1.$$
 \end{lemma}

 \begin{proof} This follows from \cite[Lemma 1.7(iii)]{taoziegler}, with the latter identities following from a routine calculation. There is an analogous classification of polynomials in other characteristics than $2$, but we will only need the characteristic two theory here.
 \end{proof}

 One quick corollary of this lemma is the \emph{exact roots property}
 \begin{equation}\label{double}
 \Poly^d(\F_2^n) = 2 \cdot \Poly^{d+1}(\F_2^n)
 \end{equation}
 for all $d \geq 0$, refining Proposition \ref{double-prop} in this case; thus, every polynomial $P$ of degree $d$ can be expressed in the form $P=2Q$ for some polynomial $Q$ of degree $d+1$, and conversely if $Q$ is of degree $d+1$ then $2Q$ is of degree $d$; see \cite[Lemma 1.7(v)]{taoziegler}.  In a similar spirit, we have

\begin{lemma}[Inverting $1+T^e$]\label{inv-lem}  Let $n \geq 1$, let $e$ be a non-zero vector in $\F_2^n$, let $d \in \Z$, and let $P \colon \F_2^n \to \T$ be a polynomial of degree at most $d$ with $\partial_e P = 0$.  Then one can write $P = (1+T^e) Q$ where $Q \colon \F_2^n \to \T$ is a polynomial of degree at most $d+1$.
\end{lemma}

\begin{proof} If $d<0$ then $P$ vanishes and we can simply take $Q=0$.  Hence we may assume $d\geq 0$.  Applying a change of variables we may assume $e=e_n$ is the final generator of $\F_2^n$. 
By Lemma \ref{explicit-desc}, we can write the $e_n$-invariant polynomial $P$ explicitly as
 $$ P(x_1,\dots,x_n) = \alpha + \sum_{k=1}^d \sum_{1 \leq i_1 < \dots < i_k \leq n-1} \frac{c_{i_1,\dots,i_k} |x_{i_1}| \dots |x_{i_k}|}{2^{d+1-k}} \hbox{ mod } 1
 $$
 for all $x_1,\dots,x_n \in \F_2$ and some $0 \leq \alpha < 1$ and some integers $0 \leq c_{i_1,\dots,i_k} < 2^{d+1-k}$, where $|x| = 1_{x=1}$.  We then define $Q(x_1,\dots,x_n)$ explicitly by the formula
 $$ Q(x_1,\dots,x_n) = \frac{\alpha}{2} + \sum_{k=1}^d \sum_{1 \leq i_1 < \dots < i_k \leq n-1} \frac{c_{i_1,\dots,i_k} |x_{i_1}| \dots |x_{i_k}| |x_n|}{2^{d+1-k}} \hbox{ mod } 1.
 $$
 From Lemma \ref{explicit-desc} again, $Q$ is a polynomial of degree at most $d+1$, and the identity $P = (1+T^e) Q$ follows from direct calculation.
\end{proof}

We will use the following stability property for polynomials on nilspaces, which we phrase in the setting of finite nilspaces as this is all we will need here.

\begin{theorem}[Stability of polynomials]\label{poly-stable}  For every $k \geq 0$ and $\eps>0$ there exists $\delta>0$ such that if $X$ is a finite ergodic nilspace, and $\phi \colon X \to \T$ is a function such that
$$ \left|e\left( \sum_{\omega \in \{0,1\}^{k+1}} (-1)^{k+1-|\omega|} \phi(x_\omega) \right) - 1 \right| \leq \delta$$
for at least $1-\delta$ of the $k+1$-cubes $(x_\omega)_{\omega \in \{0,1\}^{k+1}}$ in $X$, then there exists a polynomial $P \in \Poly^k(X)$ such that
$$ \E_{x \in X} |e(\phi(x)) - e(P(x))| \leq \eps.$$
\end{theorem}

\begin{proof} This is a special case of \cite[Theorem 4.2]{candela-szegedy-inverse} (with $Y$ the compact nilspace ${\mathcal D}^k(\T)$ with metric $d(x,y) \coloneqq |e(x)-e(y)|$), noting that for a finite ergodic nilspace we can use the uniform probability measure on $C^n(X)$ as a Haar measure on that space.
\end{proof}

\subsection{$p$-homogeneous nilspaces}

The following definition was introduced in \cite{CGSS}:

\begin{definition}[$p$-homogeneity]\label{homogeneity-def}\cite[Definitions 1.2, 3.1]{CGSS}
  Let $p$ be a prime.  A nilspace $X$ is said to be \emph{$p$-homogeneous} if, whenever $n \geq 0$ and $f \colon {\mathcal D}^1(\Z^n) \to X$ is a nilspace morphism, then the periodization $\tilde f \colon {\mathcal D}^1(\F_p^n) \to X$, defined by restricting $f$ to $\{0,\dots,p-1\}^n$ and then extending periodically, is also a nilspace morphism.

  A nilspace $X$ is said to be \emph{weakly $p$-homogeneous} if, for every $n$-cube $(x_\omega)_{\omega \in \{0,1\}^n} \in C^n(X)$ for some $n \geq 0$, there exists a nilspace morphism $\tilde f \colon {\mathcal D}^1(\F_p^n) \to X$ such that $\tilde f(\omega) = x_\omega$ for all $\omega \in \{0,1\}^n$ (viewing $\{0,1\}^n$ as a subset of $\F_p^n$).  
\end{definition}

In \cite[Remark 3.3]{CGSS} it is noted that $p$-homogeneity implies weak $p$-homogeneity, and that the two concepts are equivalent when $p=2$.  In \cite[Question 3.4]{CGSS} it is posed as an open question whether these two concepts are equivalent for $p>2$; we do not address this question here.  From this remark, we see that $X$ is $2$-homogeneous (or equivalently, weakly $2$-homogeneous), if and only if we have a bijection
\begin{equation}\label{nbox}
 C^n(X) \equiv \Hom_\Box( {\mathcal D}^1(\F_2^n) \to X)
\end{equation}
for any $n \geq 0$, where we identify maps from $\F_2^n$ to $X$ with tuples in $X^{\{0,1\}^n}$ by identifying $\{0,1\}^n$ with $\F_2^n$.  From this identification we obtain the following consequence:

\begin{lemma}[Preserving $2$-homogeneity]\label{homog-preserv}  Let $k \geq 0$, let $X$ be a $2$-homogeneous $k$-step nilspace, and let $X \rtimes^{(k)}_\rho Z$ be a $k$-skew product of that nilspace with an elementary abelian $2$-group $Z$.  Then $X \rtimes^{(k)}_\rho Z$ is $2$-homogeneous if and only if, for any $n \geq 0$, every nilspace morphism $\phi \colon {\mathcal D}^1(\F_2^n) \to X$ has a lift $\tilde \phi \colon {\mathcal D}^1(\F_2^n) \to X \rtimes^{(k)}_\rho Z$, thus $\tilde \phi$ is a nilspace morphism with $\phi = \pi \circ \tilde \phi$, where $\pi \colon X \rtimes^{(k)}_\rho Z \to X$ is the factor map.
\end{lemma}

\begin{remark} There is an analogue of this result for general $p$, but it is more difficult to prove; see \cite[Proposition 3.12]{CGSS}.
\end{remark}

\begin{proof}  We first prove the ``only if'' direction.  Suppose that $X \rtimes^{(k)}_\rho Z$ is $2$-homogeneous, and $\phi \colon {\mathcal D}^1(\F_2^n) \to X$ is a nilspace morphism.  By \eqref{nbox} for the $2$-homogeneous nilspace $X$, we may view $\phi$ as an $n$-cube on $X$, which has a lift to an $n$-cube on $X \rtimes^{(k)}_\rho Z$ by Proposition \ref{skewprod}.  Applying \eqref{nbox} again to the $2$-homogeneous nilspace $X \rtimes^{(k)}_\rho Z$, we obtain the claim.

Now we prove the ``if'' direction.  Let $n \geq 0$, and let $((x_\omega,z_\omega))_{\omega \in \{0,1\}^n} \in C^n(X \rtimes^{(k)}_\rho Z)$ be an $n$-cube in $X \rtimes^{(k)}_\rho Z$.  We would like to interpret this $n$-cube as a nilspace morphism from $\F_2^n$ to $X \rtimes^{(k)}_\rho Z$.  As $X$ is already $2$-homogeneous, we know that the $n$-cube $(x_\omega)_{\omega \in \{0,1\}^n}$ can already be identified with a nilspace morphism $\phi$ from $\F_2^n$ to $X$, which by hypothesis can be lifted to a nilspace morphism $\tilde \phi$ from $\F_2^n$ to $X \rtimes^{(k)}_\rho Z$.  In particular, we can write
\begin{equation}\label{tpz}
\tilde \phi(\omega) = (x_\omega, z_\omega + P(\omega))
\end{equation}
for all $\omega \in \{0,1\}^n$ and some map $P \colon \F_2^n \to Z$ (identifying $\{0,1\}^n$ with $\F_2^n$).  

Since $((x_\omega,z_\omega))_{\omega \in \{0,1\}^n} \in C^n(X \rtimes^{(k)}_\rho Z)$ is an $n$-cube, we have from \eqref{z} that
$$ \sum_{\omega \in \{0,1\}^{k+1}} (-1)^{k+1-|\omega|} z_{\iota(\omega)} = \rho( (x_{\iota(\omega)})_{\omega \in \{0,1\}^{k+1} )}$$
whenever $\iota \colon \{0,1\}^{k+1} \to \{0,1\}^n$ is a $k+1$-dimensional face of $\{0,1\}^n$.  As $(\tilde \phi(\omega))_{\omega \in \{0,1\}^n}$ is also an $n$-cube, the same statement is true with $z_\omega$ replaced by $z_\omega + P(\omega)$.  Subtracting, we conclude that
$$ \sum_{\omega \in \{0,1\}^{k+1}} (-1)^{k+1-|\omega|} P(\iota(\omega)) = 0$$
for all $k+1$-dimensional faces.  Equivalently, we have
$$ \partial_{e_{i_1}} \dots \partial_{e_{i_{k+1}}} P = 0$$
whenever $1 \leq i_1 < \dots < i_{k+1} \leq n$, where $e_1,\dots,e_n$ is the standard basis of $\F_2^n$.  For any $i=1,\dots,n$, we have
$$ \partial_{e_i} \partial_{e_i} = \partial_{2e_i} - 2 \partial_{e_i} = - 2 \partial_{e_i}$$
since $2e_i = 0$ (cf. \eqref{2h}); since $Z$ is assumed to be an elementary abelian $2$-group, we thus also have
$$ \partial_{e_{i_1}} \dots \partial_{e_{i_{k+1}}} P = 0$$
whenever two of the $i_1,\dots,i_{k+1}$ are equal.  We conclude that $P \in \Poly^k( \F_2^n \to Z)$.

Now let $(a_\omega)_{\omega \in \{0,1\}^{k+1}} \in C^{k+1}({\mathcal D}^1(\F_2^n))$ be a $k+1$-cube in $\F_2^n$ (with the degree $1$ filtration).  As $\tilde \phi$ is a nilspace morphism, $(\tilde \phi(a_\omega))_{\omega \in \{0,1\}^{k+1}}$ is a $k+1$-cube in $X \rtimes^{(k)}_\rho Z$, which in particular implies from \eqref{tpz} that
$$  \sum_{\omega \in \{0,1\}^{k+1}} (-1)^{k+1-|\omega|} 
(z_{a_\omega} + P(a_\omega)) = \rho( (x_{a_\omega})_{\omega \in \{0,1\}^{k+1}} ).$$
Since $P$ is a polynomial, we also have
$$  \sum_{\omega \in \{0,1\}^{k+1}} (-1)^{k+1-|\omega|} 
P(a_\omega) = 0;$$
subtracting, we conclude that
$$  \sum_{\omega \in \{0,1\}^{k+1}} (-1)^{k+1-|\omega|} 
z_{a_\omega} = \rho( (x_{a_\omega})_{\omega \in \{0,1\}^{k+1}} ).$$
As a consequence, we see that $(x_{a_\omega}, z_{a_{\omega}})_{\omega \in \{0,1\}^{k+1}}$ is a $k+1$-cube in $X \rtimes^{(k)}_\rho Z$.  Thus the map $a \mapsto (x_a, z_a)$ preserves $k+1$-cubes, and is thus a nilspace morphism from ${\mathcal D}^1(\F_2^n)$ to $X \rtimes^{(k)}_\rho Z$ thanks to Lemma \ref{k-suffices}.  This gives the claim.
\end{proof}

Finally, we remark that the notion of $p$-homogeneity greatly simplifies in the case of ergodic filtered abelian groups:

\begin{proposition}[$p$-homogeneous filtered abelian groups]\label{phom}  Let $G$ be a filtered ergodic abelian group, and $p$ a prime.  Then $G$ is $p$-homogeneous if and only if $p \cdot G_i \leq G_{i+1}$ for all $i \geq 1$.
\end{proposition}

\begin{proof}  See \cite[Theorem 1.4]{CGSS}.  In fact the ergodicity hypothesis can be dropped here, but we will not need to use this fact.
\end{proof}

\section{Deducing the strong inverse conjecture from the BTZ conjecture}\label{app}

We now prove Theorem \ref{implication}, by refining the correspondence principle argument used in \cite{taozieglerhigh}. Our arguments here follow \cite{taozieglerhigh} fairly closely, and familiarity with that argument will be assumed here. 

Fix $p,k, \eta, \eps()$; all quantities below are permitted to depend on these parameters.     Suppose for contradiction that Conjecture \ref{btz-conj} was true, but Conjecture \ref{inv-conj-strong} failed for the indicated choice of $\eta, \eps()$.  Without loss of generality we may assume $\eps(m) \leq \frac{1}{m}$ (for instance). Then for every $M$, there exists $G = \F_p^n$ for some $n = n_{M}$ and a function $f = f_{M} \colon G \to \mathbb{D}$ with $\|f\|_{U^{k+1}(G)} \geq \eta$, such that if $h_1,\dots,h_M \in G$ are chosen independently and uniformly at random, then with probability greater than $1/2$, there does \emph{not} exist $1 \leq m \leq M$ and $P \in \Poly^k(G)$ and a function $F \colon \C^{\F_p^M} \to \C$
of Lipschitz constant at most $M$, such that
$$|\E_{x \in G} f(x) e(-P(x))| \geq \frac{1}{m}$$
and
$$ |\E_{x \in G} e(P(x)) - F( (f(x+\sum_{i=1}^M a_i h_i))_{(a_1,\dots,a_M) \in \F_p^M})| \leq \eps(m).$$
We use the following construction of a \emph{sampling sequence} from \cite{taozieglerhigh}:

\begin{proposition}[Existence of accurate sampling sequence]\label{prop-sampling}  Let $\eps_0 > 0$.  Then there exists a sequence of scales
$$ 0 = H_0 < H_1 < \dots$$
such that for any $G = \F_p^n$ and $f \colon G \to \mathbb{D}$, if $v_1,v_2,v_3,\dots \in G$ are chosen uniformly and independently at random, then with probability at least $1-\eps_0$, the following ``accurate sampling'' statement holds: for every sequence
$$ 0 \leq r_0 < r_1 < r_2 < \dots < r_{k+1}$$
and every Lipschitz $F \colon \mathbb{D}^{\{0,1\}^{k+1} \times \F_p^{r_0}} \to \C$, we have
$$ \E_{x \in G} |F_{f, r_0,\dots,r_{k+1}}(x) - F_f(x)| \leq \frac{\|F\|_{\mathrm{Lip}}}{r_1}$$
where
$$ F_{f,r_0,\dots,r_{k+1}}(x) \coloneqq \E_{\vec a_1 \in \F_p^{H_{r_1}}, \dots, \vec a_{k+1} \in \F_p^{H_{r_{k+1}}}} F( (f(x + \omega \cdot {\mathbf u} + \vec b \cdot \vec v_0))_{\omega \in \{0,1\}^{k+1}, \vec b \in \F_p^{H_{r_0}}} )$$
with
$$ \mathbf{u} \coloneqq (\vec a_1 \cdot \vec v_1,\dots, \vec a_{k+1} \cdot \vec v_{k+1}); \quad \vec v_j \coloneqq (v_1,\dots,v_{H_{r_j}}), j=0,\dots,k+1$$
and
$$ F_{f,r_0}(x) \coloneqq \E_{h_1,\dots,h_{k+1} \in G} F((f(x+\omega \cdot \mathbf{h} + \vec b \cdot \vec v_0))_{\omega \in \{0,1\}^{k+1}, \vec b \in \F_p^{H_{r_0}}} )$$
where $\mathbf{h} \coloneqq (h_1,\dots,h_{k+1})$.
\end{proposition}

\begin{proof}  See \cite[Proposition 3.13]{taozieglerhigh} (with some mild relabeling, for instance replacing $k$ by $k+1$).  In that proposition the sampling property was only asserted to hold with positive probability, but an inspection of the proof shows that it can be established with probability at least $1-\eps_0$ for any fixed $\eps_0>0$.
\end{proof}

For each $M$, we apply the above proposition with $\eps_0 =1/2$, $n = n_{M}$, and $f = f_{M}$ to conclude that the claimed accurate sampling property holds for randomly chosen $v_1,v_2,\dots \in \F_p^n$ with probability at least $1/2$.  By combining this with the construction of $f_{M}$, we conclude that there exists (deterministic) $v_i = v_{i,M} \in \F_p^n$ for all $i \geq 1$ with the accurate sampling property, and also the property that there does \emph{not} exist $1 \leq m \leq M$, $P \in \Poly^k(G)$ and a function $F \colon \C^{\F_p^M} \to \C$
of Lipschitz constant at most $M$, such that
$$|\E_{x \in G} f(x) e(-P(x))| \geq \frac{1}{m}$$
and
$$ |\E_{x \in G} e(P(x)) - F( (f(x+\sum_{i=1}^M a_i v_i))_{(a_1,\dots,a_M) \in \F_p^M})| \leq \eps(m).$$

We fix this data for each $M$.  Following \cite{taozieglerhigh}, we now introduce the universal Furstenberg space $X \coloneqq \mathbb{D}^{\F_p^\omega}$ of functions $\zeta \colon \F_p^\omega \to \mathbb{D}$ with the product $\sigma$-algebra and shift action
$$ T^h \zeta(x) \coloneqq \zeta(x+h).$$
As in \cite[\S 4]{taozieglerhigh}, for each $M$, the above data generate an invariant probability measure $\mu_{M}$ on $X$ by the formula
$$ \mu_{M} \coloneqq \E_{x \in \F_p^{n_{M}}} \delta_{\zeta_{M,x}}$$
where $\zeta_{M,x} \in X$ is given by the formula
$$ \zeta_{M,x}( (a_i)_{i=1}^\infty ) \coloneqq f_{M}( x + \sum_{i=1}^\infty a_i v_{i,M} ).$$
By Prokhorov's theorem, we may restrict $M$ to a subsequence and assume that $\mu_M$ converges weakly to an invariant probability measure $\mu$.  Henceforth $X$ is understood to be endowed with $\mu$.

Let $f_\infty \colon X \to \mathbb{D}$ be the coordinate function
$$ f_\infty(\zeta) \coloneqq \zeta(0).$$
As noted in \cite[(4.3)]{taozieglerhigh}, we have the identity
\begin{equation}\label{43}
\begin{split}
&\int_X F( T_{\vec a_1} f_\infty, \dots, T_{\vec a_{l}} f_\infty)\ d\mu_M
\\
&\quad= \E_{x \in \F_p^{n_M}} F\left( f_M\left(x + \sum_{i=1}^\infty a_{1,i} v_{i,M}\right), \dots, f_M\left(x + \sum_{i=1}^\infty a_{l,i} v_{i,M} \right)\right)
\end{split}
\end{equation}
for any $l,m \geq 1$, any $\vec a_j = (a_{j,i})_{i=1}^\infty \in \F_p^\omega$ for $j=1,\dots,l$, and any continuous function $F \colon \mathbb{D}^\ell \to \C$.  This allows us to pass back and forth between integral expressions on $X$ (using the measure $\mu_M$) and combinatorial averages on $\F_p^{n_M}$.

In \cite[Lemma 4.2]{taozieglerhigh}, it is shown that the $\sigma$-algebra of $X$ is generated by $f_\infty$ and its shifts.  By \cite[Lemma 4.3]{taozieglerhigh} the identity \eqref{43} was used to show that $X$ is an ergodic $\F_p^\omega$-system; from \cite[Lemma 4.4]{taozieglerhigh} this identity was also used to show that
$$ \|f_\infty\|_{U^{k+1}(X)} \geq \eta.$$
Applying the hypothesis that Conjecture \ref{btz-conj} held, we can find $P \in \Poly^k(X)$ and some $m$ such that
$$\left|\int_X f_\infty e(-P)\ d\mu\right| > \frac{3}{m}$$
(say).   Let $c: \R^+ \to \R^+$ be a decreasing function to be chosen later (depending on $k,p$) such that $c(\eps) \to 0$ decays to zero sufficiently rapidly as $\eps \to 0$.  Then, as $X$ is generated by $f_\infty$ and its shifts, we see that there exists a natural number $M_0$ and shifts $\vec b_1,\dots, \vec b_{M_0} in \F_p^\omega$, and a Lipschitz function $F \colon \mathbb{D}^{M_0} \to \mathbb{D}$ of Lipschitz constant at most $M_0$ such that
\begin{equation}\label{sad}
 \int_X |e(P) - F( T_{\vec b_1} f_\infty, \dots, T_{\vec b_{M_0}} f_\infty)|\ d\mu < c(\eps(m))
 \end{equation}
so in particular by the triangle inequality (if $c$ decays rapidly enough)
$$ \left|\int_X f_\infty \overline{F}( T_{\vec b_1} f_\infty, \dots, T_{\vec b_{M_0}} f_\infty)\ d\mu\right| > \frac{2}{m}.$$
By vague convergence we thus have
$$\left|\int_X f_\infty \overline{F}( T_{\vec b_1} f_\infty, \dots, T_{\vec b_{M_0}} f_\infty)\ d\mu_M\right| \geq \frac{2}{m}.
$$
for arbitrarily large $M$ (in particular we can assume $M > M_0, m$).  Applying \eqref{43}, we conclude that
\begin{equation}\label{core}
\left| \E_{x \in \F_p^{n_M}} f_\infty(x)
\overline{F}( f_M(x + \sum_{i=1}^\infty b_{1,i} v_{i,M}), \dots, f_M(x+ \sum_{i=1}^\infty b_{M_0,i} v_{i,M})\right| \geq \frac{2}{m}.
\end{equation}

Now let $r_1$ be sufficiently large depending on $M_0$, $\vec b_1,\dots,\vec b_{M_0}$, and $c(\eps(m))$, and set $r_j \coloneqq r_1+j-1$ for $j=2,\dots,k+1$.  Using the triangle inequality as in the argument after \cite[(4.5)]{taozieglerhigh}, we conclude from \eqref{sad} that
$$ \E_{\vec a_1 \in \F_p^{H_{r_1}}, \dots, \vec a_{k+1} \in \F_p^{H_{r_{k+1}}}} \int_X |\Delta_{\vec a_1} \dots \Delta_{\vec a_{k+1}} F( T_{\vec b_1} f_\infty, \dots, T_{\vec b_{M_0}} f_\infty) - 1|\ d\mu_M \ll c(\eps(m))$$
for all sufficiently large $M$ along the indicated subsequence, where we use $X \ll Y$ to denote the estimate $X \leq CY$ for some $C$ depending only on $p,k$, and we use the notation $\Delta_a f(x) \coloneqq f(x+a) \overline{f(x)}$.  Continuing the argument after \cite[[(4.5)]{taozieglerhigh}, we can use   \eqref{43} and the accurate sampling sequence property to conclude (for $r_1$ large enough) that
\begin{align*}
    &
\E_{x, h_1,\dots,h_{k+1} \in \F_p^{n_M}} \left|\Delta_{h_1} \dots \Delta_{h_{k+1}} F\left( f_M\left(x + \sum_{i=1}^\infty b_{1,i} v_{i,M}\right), \dots, f_M\left(x+ \sum_{i=1}^\infty b_{M_0,i} v_{i,M}\right)\right) - 1\right| \\
&\quad \ll c(\eps(m))
\end{align*}
where the operators $\Delta_h$ are applied in the $x$ variable.
Applying \cite[Lemma 4.5]{taozieglerhigh} (or \cite[Theorem 4.2]{candela-szegedy-inverse}), and assuming that the function $c$ decays sufficiently rapidly, we may find a polynomial $P_M \in \Poly^k(\F_p^{n_M})$ such that
$$ \E_{x \in \F_p^{n_M}} \left| F\left( f_M\left(x + \sum_{i=1}^\infty b_{1,i} v_{i,M}\right), \dots, f_M\left(x+ \sum_{i=1}^\infty b_{M_0,i} v_{i,M}\right)\right) - e(P_M(x))\right| \leq \eps(m)$$
From this, \eqref{core}, and the triangle inequality (recalling that $\eps(m) \leq 1/m$) we conclude that
$$  |\E_{x \in \F_p^{n_M}} f_\infty(x) e(-P_M(x))| > \frac{1}{m}.$$
But this contradicts the construction of the sampling sequence $v_{i,M}$.  This concludes the proof of Theorem \ref{implication}.

\section*{Declarations}

Or Shalom was supported by NSF grant DMS-1926686 and ISF grant 3056/21. Terence Tao was supported by a Simons Investigator grant, the James and Carol Collins Chair, the Mathematical Analysis \& Application Research Fund, and by NSF grants DMS-1764034 and DMS-2347850. Asgar Jamneshan was funded by the Deutsche Forschungsgemeinschaft (DFG, German Research Foundation) - 547294463.

The authors have no relevant financial or non-financial interests to disclose. On behalf of all authors, the corresponding author states that there is no conflict of interest.

\end{document}